\documentclass{amsart}

\usepackage[utf8]{inputenc}
\usepackage[english]{babel}

\usepackage{amsthm,amssymb,latexsym,amsmath}
\numberwithin{equation}{section}
\numberwithin{figure}{section}
\numberwithin{table}{section}

\usepackage[all]{xy}
\usepackage{tikz}
\usepackage{tikz-cd}
\usetikzlibrary{decorations.pathreplacing}

\usepackage{hyperref}

\usepackage{color}

\newcommand{\Z}{{\mathbb Z}}

\newcommand{\Q}{\mathbb{Q}}

\newcommand{\p}[1]{{\mathbb{P}^{#1}}}

\newcommand{\op}[1]{{\mathcal O}_{\mathbb{P}^{#1}}}

\newcommand{\ox}{{\mathcal O}_{X}}

\newcommand{\OC}{{\mathcal O}_{C}}

\newcommand{\ch}{\operatorname{ch}}

\newcommand{\cale}{{\mathcal E}}

\newcommand{\calg}{{\mathcal G}}
\newcommand{\calh}{{\mathcal H}}
\newcommand{\cali}{{\mathcal I}}

\newcommand{\call}{{\mathcal L}}

\newcommand{\calo}{{\mathcal O}}
\newcommand{\calp}{{\mathcal P}}

\newcommand{\calr}{{\mathcal R}}
\newcommand{\cals}{{\mathcal S}}
\newcommand{\calt}{{\mathcal T}}
\newcommand{\calz}{{\mathcal Z}}

\newcommand{\Ext}{\operatorname{Ext}}
\newcommand{\Hom}{\operatorname{Hom}}

\DeclareMathOperator{\coker}{coker}
\DeclareMathOperator{\im}{im}

\DeclareMathOperator{\rk}{{rk}}

\DeclareMathOperator{\Pic}{{Pic}}
\DeclareMathOperator{\ext}{{ext}}
\DeclareMathOperator{\Coh}{{Coh}}

\newcommand{\lra}{\longrightarrow}
\newcommand{\into}{\hookrightarrow}
\newcommand{\onto}{\twoheadrightarrow}

\newtheorem{theorem}{Theorem}[section]
\newtheorem{mthm}{Main Theorem}
\newtheorem{proposition}[theorem]{Proposition}

\newtheorem{lemma}[theorem]{Lemma}
\newtheorem{corollary}[theorem]{Corollary}

\theoremstyle{definition}
\newtheorem{remark}[theorem]{Remark}
\newtheorem{example}[theorem]{Example}
\newtheorem{definition}[theorem]{{\bf Definition}}

\tikzcdset{scale cd/.style={every label/.append style={scale=#1},
    cells={nodes={scale=#1}}}}

\begin{document}

\title[Modular Serre correspondence via stable pairs]{Modular Serre correspondence via stable pairs}

\author{Marcos Jardim}
\address{IMECC - UNICAMP \\ Departamento de Matem\'atica \\
Rua S\'ergio Buarque de Holanda, 651\\ 13083-970 Campinas-SP, Brazil}
\email{jardim@ime.unicamp.br}
\author{Dapeng Mu}
\address{IMECC - UNICAMP \\ Departamento de Matem\'atica \\
Rua S\'ergio Buarque de Holanda, 651\\ 13083-970 Campinas-SP, Brazil}
\email{mudapeng1990@gmail.com}

\begin{abstract}
A stable pair on a projective variety consists of a sheaf and a
global section subject to stability conditions parameterized by rational
polynomials. We will show that for a smooth projective threefold and a
class of a rank 2 sheaf, there are two stability chambers (in the space of rational polynomials under the lexicographic order) for which the moduli spaces of semistable pairs admit morphisms to
a Gieseker moduli space of rank 2 semistable sheaves and a Hilbert scheme,
respectively. In the latter moduli space, every semistable pair corresponds to a closed subscheme of codimension 2 with an extension class, providing a generalization of the Serre correspondence. These two moduli spaces are related by finitely many wall-crossings. We provide explicit descriptions of those wall-crossings for certain fixed numerical classes. In particular, these wall-crossings preserve the connectedness of the moduli space of semistable pairs. 
\end{abstract}

\maketitle

\tableofcontents


\section{Introduction}

Serre correspondence, introduced in \cite{hartshorne1978stable}, established a connection between a rank $2$ vector bundle and a one-dimensional scheme on $\p3$.
Subsequent works, such as \cite{hartshorne1980stable}, \cite{arrondo2007home}, and \cite{comaschi2024instanton}, extended this correspondence to rank 2 reflexive sheaves, locally free of higher rank, and to torsion-free sheaves of arbitrary rank on smooth projective varieties, respectively.
For the rank $2$ case, let $L$ be a fixed line bundle on a smooth projective variety $X$ such that $H^1(L^{\vee})=H^2(L^{\vee})=0$. The Serre correspondence provides a one-to-one correspondence between:
\begin{enumerate}
\item pairs $(E,s)$ consisting of a rank $2$ torsion-free sheaf $E$ with $\det(E)=L$, and a non-trivial section $s\in H^0(X,E)$ whose zero scheme has codimension equal to 2. 
\item pairs $(Z, \xi)$ consisting of a codimension $2$ subscheme $Z\subset X$ and a section \linebreak
$\xi\in H^0(X, \omega_Z\otimes\omega_X^{-1}\otimes L^{\vee})$. 
\end{enumerate}
However, this correspondence is a set-theoretic one, and it does not involve moduli spaces or stability conditions. Nonetheless, the Serre correspondence has been widely used to construct sheaves on projective varieties with special properties, like instanton sheaves or Ulrich bundles. 
In addition, one can study the deformation of stable sheaves via the deformation of the associated subschemes given by the Serre correspondence \cite{Kleppe}; to mention one example, in \cite{jardim2017two,almeida2022irreducible} the authors proved the connectedness of the Gieseker moduli space for certain rank $2$ classes on $\p3$ precisely using the deformation of the corresponding schemes.

The primary objective of this paper is to demonstrate that the Serre correspondence admits a modular interpretation, namely, it can be presented as a morphism (or, more generally, a rational map) between suitable moduli spaces.
Specifically, this would involve defining moduli functors that assign to every scheme $T$ flat families of pairs $(\mathbf{E},\mathbf{s})$ and $(\mathbf{Z},\mathbf{\xi})$ over $X\times T$, respectively, such that the fibers over closed points recover the classical Serre correspondence, and proving that these functors admit good moduli spaces. 
Building on these foundational results, this work aims to construct a modular version of the Serre correspondence, for which the corresponding moduli space of pairs serves as a bridge between the Gieseker moduli space of rank 2 sheaves and the Hilbert scheme of codimension 2 subschemes.

Let $(X, \calo_X(1))$ be a smooth projective threefold over a field $\kappa$ of characteristic 0 with a fixed polarization. A \textit{coherent pair} $(E,s)$ on $X$ consists of a coherent sheaf $E\in \Coh(X)$ and a section $s\in H^0(X,E)$. 
To construct the moduli space of pairs, we use the family of stability parametrized by polynomials $\delta\in\Q[t]$ with positive leading coefficient, which was first introduced in \cite{huybrechts1992stable} for curves and surfaces, and then generalized to coherent systems in \cite{le1993syst, he1996espaces} and further studied in \cite{lin2018moduli, Wandel_2015}.
For a given coherent pair $(E,s)$, define its reduced Hilbert polynomial as 
\[
p_{(E,s)}^{\delta}(t):=\frac{P_E(t)+\varepsilon\delta}{\rk(E)}.
\]
where $\varepsilon=1$ if $s\neq 0$, and $\varepsilon=0$ otherwise. A pair $(F,s')$ is a sub-pair of $(E,s)$ if $F \overset{\iota}{\hookrightarrow} E$ is a sub-sheaf and the section $s$ factors through $s'$ up to a constant factor, i.e. $s = \lambda\cdot\iota\circ s'$ for some $\lambda\in\kappa^*$.
The pair $(E,s)$ is said to be \textit{$\delta$-(semi)stable} if every sub-pair $(F, s')\hookrightarrow
(E,s)$ satisfies $p_{(F, s')}^{\delta}(t)(\leq)<p_{(E, s)}^{\delta}(t)$, where polynomials are ordered lexicographically. 
For a given numerical class $v\in H^*(X, \mathbb{Q})$ encoding the Chern classes of $E$, the coarse moduli space $\cals^{\delta}_X(v)$ of $\delta$-semistable pairs with class $v$ was constructed in \cite{Wandel_2015, lin2018moduli}.
As $\delta$ varies, the moduli space $\cals^{\delta}_X(v)$ may not stay constant.
This naturally leads to the \emph{wall-crossing phenomenon}, which arises when a stability condition varies with a parameter.
The critical values for $\delta$ where the moduli space changes are called \emph{walls} and the region between two adjacent walls is called a \emph{chamber}.

In this paper, we focus on the wall-crossings for the moduli space 
$$ \cals^{\delta}_{X,L}(v):=\{(E,s)\in \cals^{\delta}_{X}(v)|\det(E)=L\}, $$
where $v$ is the class of a rank 2 torsion-free sheaf, and $L\in \Pic(X)$ is a fixed line bundle.
For any class $v$, there is a largest wall $W_{\emptyset}$, called the \emph{collapsing wall}, such that every pair is $\delta$-unstable for $\delta>W_{\emptyset}$. 
In the region $(0, W_{\emptyset})$, there are two significant walls $W_G\leq W_T$ that reveal interesting geometric interpretations for the moduli space $\cals^{\delta}_{X, L}(v)$, connecting the Gieseker moduli with the Hilbert scheme.

Denote by $\calg_X(v)$, $\calg_{X,L}(v)$, and $\calh_X(\tilde{P})$ the Gieseker moduli space for class $v$, the subscheme of $\calg_X(v)$ parameterizing semistable sheaves with fixed determinant $L$, and the Hilbert scheme for the fixed Hilbert polynomial $\tilde{P}$. The following statement presents our first main result, with Figure \ref{wall-crossing} illustrating the situation.

\begin{mthm}[Theorem \ref{Gie-Hilb}]
For a fixed rank $2$ class $v\in H^*(X, \mathbb{Q})$ and a fixed line bundle $L$, there is a collapsing wall $W_{\emptyset}\in \mathbb{Q}[t]_{>0}$ after which no pair is $\delta$-semistable. 
There are, potentially, two more walls $W_G$ and $W_T$ such that the moduli space $\cals^{\delta}_{X, L}(v)$ admits a morphism to $\calg_{X,L}(v)$ when $\delta\in (0,W_G)$ and a morphism to $\calh_X(\tilde{P})$ when $\delta\in(W_T, W_{\emptyset})$, where $\tilde{P}=P_{\ox}(t) + P_{L^{\vee}}(t) - P_{E\otimes L^{\vee}}(t)$.

\end{mthm}

\begin{figure}[ht]
\begin{tikzpicture}
\draw[thick, ->] (-0.5,0)--(10,0);
\draw (9,-0.2) -- (9, 0.2);
\node at (9, -0.5){$W_{\emptyset}$};
\draw (2,-0.2) -- (2, 0.2);
\node at (2, -0.5){$W_{G}$};
\node at (4.5, -0.5){$\cdots$};
\draw (7,-0.2) -- (7, 0.2);
\node at (7, -0.5){$W_{T}$};
\node at (0.5, 0.5){$\mathbb{Q}[t]-$axis for $\delta$};
\node at (0, -0.5){$0$};
\draw (0,-0.2) -- (0, 0.2);
\draw [decorate,decoration={brace,amplitude=5pt,mirror,raise=5ex}]
   (7,0) -- (9,0) node[midway,yshift=-4em]{``Hilbert" chamber};
\draw [decorate,decoration={brace,amplitude=5pt,mirror,raise=5ex}]
   (0,0) -- (2,0) node[midway,yshift=-4em]{``Gieseker" chamber};
\end{tikzpicture}
\caption{Wall-crossings for a rank 2 class.}
\label{wall-crossing}
\end{figure}
\begin{itemize}
\item In the \textit{Hilbert chamber} $(W_T, W_{\emptyset})$, every $\delta$-stable pair $(E,s)$ is \emph{saturated}, meaning that 
$\coker\{\ox\stackrel{s}{\to}E\}$ is a torsion-free sheaf of rank 1 (see Definition \ref{defn:sat} below), and it is therefore isomorphic to $\cali_Y\otimes\det(E)$ for some subscheme $Y\subset X$ of codimension at least 2.
\item In the \textit{Gieseker chamber} $(0, W_G)$, every $\delta$-stable pair $(E,s)$ is \emph{very stable}, which means that $E$ is Gieseker semistable, see Definition \ref{defn:vs} below.
\end{itemize}

The \textit{modular Serre correspondence} mentioned in the title of the paper arises on each irreducible component $\calp\subset \cals^{\delta}_{X, L}(v)$ where a general pair in $\calp$ is both saturated and very stable. 
For such components, there exist rational maps from $\calp$ both to (an irreducible component of) the Gieseker moduli space and to (an irreducible component of) the Hilbert scheme, as illustrated in Figure \ref{Serrecor}. 
This correspondence establishes a connection between the deformation of sheaves and the deformation of codimension 2 schemes. 
Notably, the components studied in \cite[Section 10]{almeida2022irreducible} are of this type. 
However, not every irreducible component of $\cals^{\delta}_{X, L}(v)$ admits both rational maps; nonetheless, we prove that every irreducible component of $\cals^{\delta}_{X, L}(v)$ admits at least one of these rational maps. In particular, there are actual morphisms $\Psi:\cals^{\delta}_{X, L}(v)\to\calg_{X,L}(v)$, when $\delta\in(0,W_G)$, and $\Gamma:\cals^{\delta}_{X, L}(v)\to\calh_X(\tilde{P})$, when $\delta\in(W_T,W_\emptyset)$, as stated in Main Theorem 1.

\begin{figure}[ht]
\begin{tikzcd}   &\calp\arrow[dl, dotted, "\Psi", swap]\arrow[dr, dotted, "\Gamma"]&\\
\calg_{X,L}(v)&&\calh_X(\tilde{P})\\
\end{tikzcd}
\caption{Modular Serre correspondence, where $\calp$ is an irreducible component of $\cals^{\delta}_{X,L}(v)$ whose general pair is both saturated and very stable. Here, $\Gamma(E,s)=\big(L/\coker(s)\big)\otimes L^\vee$, while
$\Psi(E,s)=E$.}
\label{Serrecor}
\end{figure}
We emphasize that the walls $W_G$ and $W_T$ may not occur at all (see the example in Section \ref{subsec:no-walls}), hence the word \textit{potentially} in the statement of Main Theorem 1. 

\bigskip

The second part of the paper is dedicated to studying how the moduli space of stable pairs $\cals^{\delta}_{X, L}(v)$ changes as $\delta$ crosses a wall. For the sake of concreteness, we will restrict ourselves to the case $X=\p3$, though a similar study can be done on other varieties once one possesses good descriptions of the Gieseker moduli space and of the corresponding Hilbert scheme. In addition, the careful consideration of the case $X=\p3$ is motivated by the observation that any Hilbert scheme $\calh_{\p3}(\tilde{P})$ is connected for any $\tilde P$, implying that the moduli of stable pairs $\cals_{\p3}^\delta(v)$ is also connected when $\delta\in(W_T,W_{\emptyset})$. 
If connectedness is preserved as $\delta$ crosses all the walls until the Gieseker chamber, then one concludes that the Gieseker moduli space is also connected.  

As $\delta$ moves away from the Hilbert chamber and gets smaller, it will cross a set of walls $\{W_i\}$ $(i=0,1,...,l_v)$ (Figure \ref{Figure:0-dim walls}) defined by the following short exact sequences of pairs:
\[
W_i: 0\lra (\mathcal{I}_{A_i}(1), 1)\lra (E, s)\lra (\mathcal{I}_{P_i}(c_1(v)-1), 0)\lra 0.
\]  
In the above sequences, $A_i\subset \p3$ is a planar one-dimensional scheme, $P_i\subset \p3$ is a zero-dimensional scheme of length $i$, and $l_v$ is an integer determined by class $v$.

\begin{figure}[ht]
\begin{tikzpicture}
\draw[->,thick] (-0.5,0)--(11,0);
\draw (10,-0.2) -- (10, 0.2);
\node at (10, -0.5){$W_{\emptyset}$};
\draw (9,-0.2) -- (9, 0.2);
\node at (9, -0.5){$W_{0}$};
\draw (8,-0.2) -- (8, 0.2);
\node at (8, -0.5){$W_{1}$};
\draw (7,-0.2) -- (7, 0.2);
\node at (7, -0.5){$W_{2}$};
\node at (5, -0.5){$\cdots$};
\node at (1.5, -0.5){$\cdots$};
\draw (3,-0.2) -- (3, 0.2);
\node at (3, -0.5){$W_{l_{v}}$};
\node at (0.5, 0.5){$\mathbb{Q}[t]-$axis for $\delta$};
\node at (0, -0.5){$0$};
\draw (0,-0.2) -- (0, 0.2);
\draw [decorate,decoration={brace,amplitude=5pt,mirror,raise=5ex}]
 (9,0) -- (10,0) node[midway,yshift=-4em]{``Hilbert" chamber};
\end{tikzpicture}
\caption{0-dimensional walls}
\label{Figure:0-dim walls}
\end{figure}

Correspondingly, this set of walls determines subsets within certain irreducible components of $\calh_{\p3}(\tilde{P})$. 
To be precise, let $\calh_{pl}$ be the union of the irreducible components of $\calh_{\p3}(\tilde{P})$ parameterizing the union of a planar curve with a 0-dimensional scheme not necessarily contained in the same plane as the curve. 
When $\calh_{pl}$ is irreducible, let $\mathcal{X}_i\subset \calh_{\p3}(\tilde{P})$ ($i=0,1,..., l_v$) be the subscheme consisting of one-dimensional schemes of the form$\{Y=Y_1\cup Q_i\}$, such that $Y_1$ is a 1-dimensional scheme contained in a plane $H$, and $Q_i$ is 0-dimensional scheme of length $i$ not contained in $H$.
Then, define $\mathcal{Z}_i:=\overline{\mathcal{X}_i}\backslash \overline{\mathcal{X}_{i-1}}$, ($\mathcal{X}_{-1}:=\emptyset$), and it follows that $\displaystyle\calh_{pl}=\bigcup_{i=0}^{l_v} \mathcal{Z}_i$.
Furthermore, denote by $\tilde{\calh}_{pl}\subset \cals^{\delta}_{\p3}(v)$ (resp. $\tilde{\mathcal{Z}}_i\subset \cals^{\delta}_{\p3}(v)$) the subscheme lying over $\calh_{pl}$ (resp. $\mathcal{Z}_i$).
We state our main result in the following proposition, which concerns the wall-crossings at $\{W_i\}$ and the correspondence between the set of walls $\{W_i\}$ and the stratification $\{\mathcal{Z}_i\}$ of the Hilbert scheme.
In this paper, a \textit{flip} is a rational map between two varieties that is an isomorphism outside subschemes of codimension $\geq 2$; it is realized by two small contractions to a common base.
A \textit{divisorial contraction} is a morphism collapsing a subvariety of codimension 1 to a subscheme of larger codimension.

\begin{mthm}[Theorem \ref{flipblowdownremoval}]
When $\delta$ crosses $W_i$ as it decreases, an open subscheme of $\tilde{\calh}_{pl}$ undergoes 
\begin{itemize}
    \item a flip for $i=0,1,2,..., l_v-2$,
    \item a divisorial contraction for $i=l_v-1$, and
    \item a removal of an open subscheme for $i=l_v$.
\end{itemize}
Wall-crossing at $W_i$ removes the one-dimensional schemes in $\mathcal{Z}_i$ from the Hilbert scheme. 
Furthermore, all these wall-crossings preserve the connectedness of the moduli space of semistable pairs.
\end{mthm}

To illustrate our results, we further explore concrete examples on $\p3$ to highlight the power of $\delta$-stability in linking the Gieseker moduli space for a rank $2$ class to the Hilbert scheme of space one-dimensional schemes. The selected examples present different wall-crossing phenomena, (birational) transformations of moduli spaces, and the strata in the Hilbert scheme associated with wall-crossings. To be precise:

\begin{itemize}
\item We begin in Section \ref{subsec:no-walls} with a simple case where no wall-crossings occur. In this case, the moduli space of stable pairs is a flag variety, and we interpret its two classical projections as fibrations over the Gieseker moduli and the Hilbert scheme.
\item The second example, presented in Section \ref{subsec:2}, involves two wall-crossings, both defined by 0-dimensional schemes, which are described in Section \ref{sec:critical}.
\item In the third example, presented in Section \ref{subsec:3}, we encounter a wall that is not defined by 0-dimensional schemes. 
\item The fourth example, presented in Section \ref{subsec:4}, demonstrates that wall-crossings can remove one-dimensional schemes that are non-reduced in the Hilbert scheme.
\item In the last example, presented in Section \ref{subsec:5}, we use wall-crossing techniques to prove that the Gieseker moduli space for the semistable rank $2$ sheaves with Chern classes $(c_1=0,c_2=2,c_2=2)$ has exactly two irreducible components, a result previously known to specialists but never formally proved and published.
\end{itemize}

\bigskip

The paper is organized as follows. 
In Section \ref{sec:pairs}, we will recall basic facts about coherent pairs and their stability conditions, and then introduce three special walls for a rank $2$ class.
In Section \ref{sec:moduli}, we will first review the moduli theory for coherent pairs. Next, we will show the general wall-crossing picture for a rank $2$ class on a smooth projective threefold, proving Main Theorem 1 and revealing the link between the Gieseker moduli space and the Hilbert scheme. 
In Section \ref{sec:wallx}, we will present general results concerning wall crossings, including the corresponding locally closed stratification in the Hilbert scheme, the (birational) transformations of the moduli space, and the generalized criterion for the stability of a sheaf.
In Section \ref{sec:critical}, we will give full descriptions of wall-crossings at the set of walls $\{W_i\}$ defined by $0-$dimensional schemes on $\p3$.
Lastly, in Section \ref{sec:ex}, we will present the complete wall-crossings for several rank 2 classes on $\p3$, along with an application in which we determine the number of components in two Gieseker moduli spaces, illustrating the utility of wall-crossing techniques.

\subsection*{Acknowledgments}
MJ is supported by the CNPQ grant number 305601/2022-9 and the FAPESP Thematic Projects 2018/21391-1 and 2021/04065-6. DM was supported by the FAPESP post-doctoral grant number 2020/03499-0 and the grant number 2023/06829-9.


\section{Coherent pairs} \label{sec:pairs}

In this section, we start with setting up all the notations and basic definitions studied throughout the paper, following \cite{he1996espaces,le1993syst}. 
Then, we introduce three special walls (critical values) for the stability parameter. The chambers defined by those walls have specific geometric interpretations and are closely related to the goal of this paper. 

Let $(X,\ox(1))$ be a smooth projective variety over an algebraically closed field $\kappa$; set $n=\dim(X)$. By a \textit{numerical class}, we mean a vector
$$ v=(v^0, v^1, ..., v^n)\in \mathbb{Z}\times \mathbb{Z}\times \frac{1}{2}\mathbb{Z}\times \cdots\times \frac{1}{n!}\mathbb{Z}; $$
$v$ is \textit{normalized} if $v^1/v^0\in(-1,0]$; and $v$ is a \textit{rank 2 class} if $v^0=2\cdot H^n$.

If $E$ is a sheaf, then we define
$$ v(E)=(\ch_0(E)\cdot H^n,\ch_1(E)\cdot H^{n-1},\ch_2(E)\cdot H^{n-2},\dots,\ch_n(E)). $$
If $v=v(E)$, then we also define the twisted class $v(k):=v(E\otimes\ox(k\cdot H))$ where $k\in \mathbb{Z}$.
As a convention, we keep using the notation, $c_i(v):=v^i$.

\subsection{Basic facts about coherent pairs}
A \textit{coherent pair} $(E,s)$ on $X$ consists of a coherent sheaf $E$ and a (possibly zero) global section $s\in H^0(E)$. We say that $(E,s)$ is \textit{pure} if $E$ is a pure sheaf and $s$ is a non-trivial section.

A morphism $(F,s')\to(E,s)$ between coherent pairs is a pair $(\phi,\lambda)$, where $\phi\in\Hom(F,E)$ and $\lambda\in\kappa^*$, such that $\phi\circ s'=\lambda\cdot s$, i.e. the following diagram commutes
$$ \xymatrix{ \ox \ar[r]^{\lambda}\ar[d]_{s'} & \ox \ar[d]^s \\
F \ar[r]^{\phi} & E}.$$
We call $(F,s')$ a sub-pair of $(E,s)$ if $F \overset{\iota}{\hookrightarrow} E$ is a sub-sheaf and the section $s$ factors through $s'$ up to a constant factor, i.e. $s = \lambda\cdot\iota\circ s'$ with $\lambda\in\kappa^*$. If $s'\neq 0$, then $s$ induces a zero section for the quotient sheaf $E/F$. So the corresponding \emph{quotient pair} is $(E/F, 0)$, and we have the following short exact sequence of pairs 
$$ 0 \longrightarrow (F,s') \longrightarrow(E,s) \longrightarrow (E/F,0) \longrightarrow 0. $$
If $s'=0$, then $s$ induces a non-zero section $s''$ of the quotient sheaf $E/F$. In this case, we have a sequence of pairs as follows
$$ 0 \longrightarrow (F,0) \longrightarrow (E,s) \longrightarrow (E/F,s'') \longrightarrow 0,  $$
where $s''$ is given by the composition $\ox\stackrel{s}{\to}E\onto E/F$.

\begin{remark}
To simplify notations, we also denote a pair $(E, s\neq 0)$ by $(E, 1)$ if there is no need to specify the nonzero section $s$.
\end{remark}

Let $\delta\in\Q[t]_{>0}$ be a rational polynomial with a positive leading coefficient. The reduced $\delta$-Hilbert polynomial of a coherent pair $(E,s)$ is defined by 
\begin{equation}{\label{Hilb poly}}
p_{(E,s)}^\delta(t):=\frac{P_E(t)+\epsilon\cdot\delta}{\rk(E)},
\end{equation}
where $P_E(t)$ and $\rk(E)$ denote the Hilbert polynomial and the rank of $E$ (in the sense of \cite[Section 1.2]{Huybrechts_1997}). $\epsilon$ is defined as 
\[
\epsilon=
\begin{cases}
1\ \text{if}\ s\neq0\\ 
0\ \text{if}\ s=0.\\ 
\end{cases}
\]

Following \cite[Definition 1.4]{Wandel_2015} and \cite[Definition 2.7]{lin2018moduli}, a coherent pair $(E,s)$ is \textit{$\delta$-(semi)stable} if it is pure, and every proper, nonzero sub-pair $(F, s')$ of $(E, s)$ satisfies the following inequality
$$ p_{(F,s')}^\delta(t) ~ (\leq) < ~ p_{(E,s)}^\delta(t) , $$
where “$\geq$" (or “$<$") indicates the lexicographic order on $\Q[t]$. A pair $(E,s)$ is \textit{strictly $\delta$-semistable} if it is $\delta$-semistable but not $\delta$-stable.

Analogous to Gieseker stability for coherent sheaves, Harder--Narasimhan and Jordan--H\"{o}lder filtrations exist for coherent pairs as well, see for instance \cite{lin2018moduli} or \cite[Proposition 2.3]{he1996espaces}. More precisely, for any pair $(E, s)$ on $X$ and any fixed $\delta\in \mathbb{Q}[t]_{> 0}$, there exists a unique filtration 
$$ 0 = (E_0, s_0) \subset (E_1, s_1) \subset \cdots  \subset (E_n, s_n)=(E, s) $$ 
such that $gr_i(E, s):=(E_i, s_i)/(E_{i-1}, s_{i-1})$ ($i=1,2,...,n$) are $\delta$-semistable with $p_{gr_1(E, s)}^\delta(t)>\cdots >p_{gr_n(E, s)}^\delta(t)$. Moreover, for a semistable pair $(E, s)$, there is a filtration
$$ 0 = (E_0, s_0) \subset (E_1, s_1) \subset \cdots  \subset (E_m, s_m)=(E, s) $$ 
such that the factors $gr_i(E, s)$ ($i=1,2...,m$) are $\delta$-stable with
$$ p_{gr_1(E, s)}^\delta(t)=\cdots=p_{gr_n(E, s)}^\delta(t). $$
Two semistable pairs $(E,s)$ and $(F,s')$ are S-equivalent if $$\bigoplus gr_i(E,s)\simeq\bigoplus gr_i(F,s').$$

In \cite[Section 1]{he1996espaces}, the author defined the functors $\Ext$ and $\mathcal{E}xt$ for coherent systems. The definitions rely on the construction of injective objects in the category of coherent systems. 
We skip this construction but just recall a result that provides the relation between extension classes of coherent sheaves and coherent pairs.
We will use this result for computations in later sections.
Interested readers may consult \cite[Theorem 1.3, Corollary 1.6]{he1996espaces} for more details.

\begin{lemma}[\cite{he1996espaces} Corollary 1.6]{\label{LES}}
If $\Lambda=(E, s)$ and $\Lambda'=(E', s')$ are two coherent pairs on $X$, then there is a long exact sequence as follows
\begin{align*}
0 &\lra \Hom(\Lambda, \Lambda') \lra \Hom(E, E') \lra \Hom(\mathbb{C}\cdot(s), H^0(E')/\mathbb{C}\cdot(s')) \lra \\
&\lra \Ext^1(\Lambda, \Lambda') \lra \Ext^1(E, E') \lra \Hom(\mathbb{C}\cdot(s), H^1(E')) \lra \\
 &\lra \Ext^2(\Lambda, \Lambda') \lra \Ext^2(E, E') \lra \Hom(\mathbb{C}\cdot(s), H^2(E')) \lra \cdots
\end{align*}
Here, $\mathbb{C}\cdot(s)$ denotes the $\mathbb{C}$-vector space over by the section $s$.
\end{lemma}

\subsection{Stability of pairs vs. stability of sheaves}

The next two results provide simple examples of $\delta$-stable pairs. Recall that if $E$ is a torsion-free sheaf and $F\into E$ is a subsheaf, then
$$ p_E(t) - p_{F}(t) = \big(\mu(E)-\mu(F)\big) \cdot t^{n-1} + \text{lower order terms}, $$
where $\mu(E)=\deg(E)/\rk(E)$ denotes the Mumford slope of $E$.
\begin{lemma}
Assume that $E$ is a $\mu$-stable torsion-free sheaf with $h^0(E)>0$, and let $s\in H^0(E)$ be a non-trivial section.
If $\deg(\delta)<n-1$, then the pair $(E,s)$ is $\delta$-stable.
\end{lemma}
\begin{proof}
When $\deg(\delta)<n-1$ we have that $p_{(E,s)}^\delta(t) - p_{(F,s')}^\delta(t)$ (where $s'$ is possibly trivial) has degree $n-1$ and its leading coefficient is precisely $\mu(E)-\mu(F)$; this difference is positive for any $F$ because $E$ is $\mu$-stable. Therefore, $p_{(E,s)}^\delta(t) - p_{(F,s')}^\delta(t)>0$ and the pair $(E,s)$ must be $\delta$-stable.
\end{proof}

\begin{proposition} \label{prop:stability}
Let $E$ be a torsion-free sheaf of rank $r$ with $h^0(E)>0$, and fix $\delta<1/(\rk(E)-1)$. If a pair $(E,s)$ is $\delta$-stable, then $E$ is semistable.
\end{proposition}

\begin{proof}
If $E$ is not semistable, let $F\into E$ be a destabilizing sheaf, so that $p_E(t)$ is smaller than $p_{F}(t)$. Then $(F,0)\into(E,s)$ is a sub-pair, and 
$$ p_{(E,s)}^\delta(t) - p_{(F,0)}^\delta(t) = p_E(t) - p_{F}(t) + \dfrac{1}{\rk(E)}\delta.$$
There are two possibilities.
\begin{itemize}
\item if $\deg\big(p_E(t)-p_{F}(t)\big)>0$ (which occurs when $E$ is not $\mu$-semistable, for instance), then also  $p_{(E,s)}^\delta(t)-p_{(F,0)}^\delta(t)<0$ because $\deg(\delta)=0$;
\item if $\deg\big(p_E(t)-p_{F}(t)\big)=0$, then in fact
\begin{equation}\label{eq:diff-p}
p_E(t)-p_{F}(t) = \dfrac{\chi(E)}{\rk(E)} - \dfrac{\chi(F)}{\rk(F)} \in \Z\left[ \dfrac{1}{\rk(E)\rk(F)}\right],
\end{equation}
thus $p_E(t)-p_{F}(t)<0$ if and only if $p_E(t)-p_{F}(t)\le-1/\rk(E)\rk(F)$. It follows that
$$ p_{(E,s)}^\delta(t) - p_{(F,0)}^\delta(t) \le \dfrac{-1}{\rk(E)\rk(F)} + \dfrac{1}{\rk(E)}\delta = \dfrac{1}{\rk(E)} \left(\dfrac{-1}{\rk(F)} + \delta \right).$$
By the definition of stability of a sheaf $E$, it is enough to consider sub-sheaves $F$, s.t. $\rk(F)\leq \rk(E)-1$. Therefore,
$$p_{(E,s)}^\delta(t) - p_{(F,0)}^\delta(t)\leq \frac{1}{\rk(E)}\left(\frac{-1}{\rk(E)-1}+\delta \right)< 0, $$
and the pair $(E,s)$ is not $\delta$-stable.
\end{itemize}
This completes the proof.
\end{proof}

\begin{proposition} \label{prop:converse}
Let $E$ be a torsion-free sheaf of rank $r$ with $h^0(E)>0$, and fix $\delta<1/\rk(E)$. If $E$ is stable, then, for every non-trivial section $s\in H^0(E)$, the pair $(E,s)$ is $\delta$-stable.
\end{proposition}
\begin{proof}
Assume that the pair $(E,s)$ with $E$ being a torsion-free sheaf is not $\delta$-stable; let $(F,s')$ be the destabilizing sub-pair. Again, one must analyze two possibilities.

First, if $s'=0$ then
$$ 0 \ge p_{(E,s)}^\delta(t) - p_{(F,0)}^\delta(t) = p_E(t) - p_{F}(t) + \dfrac{1}{\rk(E)}\delta > p_E(t) - p_{F}(t) $$
because $\delta>0$, so $F$ destabilizes $E$, and $E$ is not stable.

When $s'\ne0$, then 
$$ 0 \ge p_{(E,s)}^\delta(t) - p_{(F,0)}^\delta(t) = p_E(t) - p_{F}(t) + \left(\dfrac{1}{\rk(E)}-\dfrac{1}{\rk(F)}\right)\delta. $$
Since $\deg(\delta)=0$, we have that
\begin{equation} \label{eq:degs}
\deg\left( p_{(E,s)}^\delta(t) - p_{(F,0)}^\delta(t) \right) = \deg \Big( p_E(t) - p_{F}(t) \Big).
\end{equation}
Again, there are two cases to be analyzed.
\begin{itemize}
\item if the degrees in display \eqref{eq:degs} are positive, then $p_E(t) - p_{F}(t)<0$, so $E$ is not stable.
\item if the degrees in display \eqref{eq:degs} are zero, then
$$ p_E(t) - p_{F}(t) = \dfrac{\rk(F)\chi(E)-\rk(E)\chi(F)}{\rk(E)\rk(F)} \le  \left(\dfrac{1}{\rk(F)}-\dfrac{1}{\rk(E)}\right)\delta. $$
It follows that
$$ \rk(F)\chi(E)-\rk(E)\chi(F) \le \big(\rk(E)-\rk(F)\big)\delta < 1 - \dfrac{\rk(F)}{\rk(E)}<1. $$
Since the left-hand side of the inequality is an integer, we have that  
$\rk(F)\chi(E)-\rk(E)\chi(F)\le0$, implying again that $F$ destabilizes $E$.
\end{itemize}
\end{proof}

\begin{proposition} \label{prop:semistable}
Let $E$ be a properly semistable torsion-free sheaf. 
For any fixed $\delta\in \mathbb{Q}_{>0}$, a pair $(E,s)$ is $\delta$-stable if for any $E\stackrel{g}{\onto}G$, where $G$ is semistable with $p_E(t)=p_G(t)$, we have $g\circ s\ne 0$. 
In particular, $E$ is indecomposable.
\end{proposition}

\begin{proof}
Suppose $(E,s)$ is unstable for some $\delta>0$. Let 
$$ 0 \lra (F,s') \lra (E,s) \lra (G, s'') \lra 0 $$
be a short exact sequence of pairs in which 
$(F,s')$ destabilizes $(E,s)$. 
We have the following inequality 
\begin{equation}{\label{ineq:destabilize}}
0<p_{F,s'}^{\delta}-p_{E,s}^{\delta}=p_F(t)-p_E(t)+\left(\frac{1}{\rk(F)}-\frac{1}{\rk(E)}\right)\delta.
\end{equation}
Because $E$ is properly semistable, we have $p_F(t)-p_E(t)\leq 0$, and the equality holds if and only if $F$ (resp. $G$) is extended by the Jordan-H\"{o}lder factors of $E$. 
Therefore, we must have $\left(\frac{1}{\rk(F)}-\frac{1}{\rk(E)}\right)\delta>0$, and in particular, $s'\neq 0$. This indicates that $g\circ s=0$, which contradicts the assumption that $g\circ s\neq 0$. 

Suppose $E$ splits as $E\simeq F\oplus G$ with $p_E(t)=p_F(t)=p_G(t)$, and assume, without loss of generality, that the restriction of $s$ to $F$ (call it $s'$) is nonzero.
Then it is straightforward to check that $(F, s')\hookrightarrow (E,s)$ destabilizes $(E,s)$, hence $(E,s)$ is unstable.
\end{proof}

\begin{example}
Set $X=\p3$ and $C,D\subset\p3$ be closed, connected subschemes of codimension at least 2 with $P_{\mathcal{O}_{C}}(t)=P_{\mathcal{O}_{D}}(t)$; assume that $C$ is planar, but $D$ is not. If $E$ is given by an extension of the form
$$ 0 \lra \cali_C(1) \lra E \lra \cali_D(1) \lra 0 , $$
then $H^0(E)=H^0(\cali_C(1))$, thus there is no section $s\in H^0(E)$ for which $(E,s)$ is $\delta$-stable. On the other hand, if $E$ is given by an extension of the form
$$ 0 \lra \cali_D(1) \lra E \lra \cali_C(1) \lra 0 , $$
then $H^0(E)\simeq H^0(\cali_C(1))$, thus $(E,s)$ is $\delta$-stable for every $s\in H^0(E)$. 
\end{example}


\subsection{Critical values (walls) for $\delta$} \label{subsec:critical}

For any given numerical class $v$, we define a wall (critical value) as follows.
\begin{definition}
    A positive polynomial $W\in \mathbb{Q}[t]_{>0}$ is called a \textit{wall} or a \textit{critical value} for $v$ if there exists a strictly $\delta$-semistable pair $(E,s)$ for $\delta=W$ such that $v(E)=v$. 
\end{definition}

\begin{remark}
    If $W\in \mathbb{Q}[t]_{>0}$ is a critical value or a wall on the rational polynomial axis, then the sets of $\delta$-stable objects for $\delta>W$ and $\delta<W$ are different in general. 
    What follows is an example illustrating this.
\end{remark}

The first wall to observe is defined by the natural sub-pair $(\ox, \mathbf{1}_{\ox})\into (E,s)$
($s\ne0$). 
Note that
$$ p_{(E,s)}^\delta(t) - p_{(\ox, \mathbf{1}_{\ox})}^\delta(t) = \dfrac{1}{\rk(E)}P_E - P_{\ox} + \dfrac{1-\rk(E)}{\rk(E)}\delta . $$
Set the above equation equal to zero. The solution for $\delta$ is 
\[
\delta=W_{P_E}:=\dfrac{1}{\rk(E)-1}\Big( P_E - \rk(E)\cdot P_{\ox}(t) \Big),
\]
which indicates the following three situations
\begin{enumerate}
    \item[$\delta>W_{P_E}$:] Every pair $(E,s)$ is unstable, destabilized by $(\ox, \mathbf{1}_{\ox})$.
    \item[$\delta=W_{P_E}$:] There exists semistable pairs $(E,s)\simeq (\calo_X,\mathbf{1}_{\ox})\oplus (\cali_A(c_1(v)),0)$ ($A\subset X$ is a one-dimensional scheme).
    \item[$\delta<W_{P_E}$:] There exists stable pairs $(E,s)$ under some reasonable assumptions.
\end{enumerate}

Therefore, $\delta=W_{P_E}$ is a \textit{critical value} in the sense that when $\delta$ varies in $\mathbb{Q}[t]_{>0}$, the set of $\delta$-stable pairs is changed when $\delta$ crosses $W_{P_E}$. 
Moreover, this also motivates the following definition of a \textit{collapsing wall}.

\begin{definition}{\label{colwall}}
The \textit{collapsing wall} for any given class $v$ is defined by the destabilizing pair $(\ox, \mathbf{1}_{\ox})\hookrightarrow (E,s)$, where $v(E)=v$. 
The corresponding critical value is denoted by $W_{\emptyset}$ and expressed as
$$W_{\emptyset}=\dfrac{1}{\rk(E)-1}\Big( P_E - \rk(E)\cdot P_{\ox}(t) \Big).$$
Every pair $(E,s)$ with $v(E)=v$ is $\delta$-unstable for $\delta>W_{\emptyset}$.
\end{definition}

\begin{remark}{\label{k>0}}
In practice, we will always work with a class $v(k)$ for which $h^0(E(k))>0$; otherwise, there would not exist nontrivial stable pairs. In this case, the collapsing wall for $v(k)$ is given by

\begin{align*}
W_{\emptyset}=&\displaystyle\frac{1}{\rk(E)-1}\cdot(P_{E(k)}(t)-\rk(E)\cdot P_{\ox}(t)) \\
=&\displaystyle\frac{\deg(E)+k\cdot\rk(E)}{(\rk(E)-1)(n-1)!}\cdot t^{n-1}+\text{lower order terms},\\
\end{align*}
which is positive if $k>0$. 
In this case, we have a region $\delta\in (0, W_{\emptyset})$ in which there are interesting wall-crossings.
\end{remark}

Following Remark \ref{k>0}, one can prove that there are only finitely many walls in the chamber $(0, W_{\emptyset})$. It was proved in \cite[Theorem 4.2]{he1996espaces} for coherent systems. 
We state her result below for the case of stable pairs. 

\begin{theorem}[\cite{he1996espaces}]{\label{Thm:finite}}
    For a given numerical class $v$ (or equivalently $P\in \mathbb{Q}[t]_{>0}$), there are only a finite number of critical values (walls) for class $v$. All such critical values have degree less than the degree of $P$. Moreover, the moduli space remains constant between two consecutive critical values. 
\end{theorem}

In \cite[Remark 6.3]{bradlow2003coherent} the authors pointed out that ``Once a coherent system (coherent pair, in particular) is removed, it can never return. 
In contrast to this, it can happen that a coherent system once added may have to be later removed." 
The first statement remains to be true for stable pairs, and for the second statement, we will give a criterion in Lemma \ref{criterion:sat} for ``when will a stable pair added be later removed" in some situations.

\subsection{Special walls for a rank 2 class}\label{rk 2 wall} 

In this subsection, we focus on a smooth n-dimensional projective variety $(X, \ox(1))$ and a fixed normalized class $v$ of rank $2$.
In section $2.3$, we introduced the collapsing wall $W_{\emptyset}$ for $v(k)$. 
Here, we will introduce two more walls between $\delta=0$ and $\delta=W_{\emptyset}$ that are related to the Gieseker moduli space and the Hilbert scheme.

If a rank $2$ pair $(E,s)$ of class $v(k)$ is $\delta$-semistable, then its Jordan-H\"{o}lder factors must be $(L, 1)$ and $(L', 0)$, in which $L$ and $L'$ are rank $1$ torsion-free sheaves. 
Equivalently, both $L$ and $L'$ are twisted ideal sheaves on $X$.
Therefore, a \textit{wall} for a rank $2$ class must be defined by the following short exact sequence of pairs
\begin{equation}{\label{rank 2 wall}}
0\lra (L, 1)\lra (E, s)\lra (L',0)\lra 0.    
\end{equation}

\subsubsection{The wall for saturated pairs}{\label{sat wall}} 
Let $(E, s)$ be a pair of class $v(k)$. 
Denote the cokernel of $s$ by $Q$, i.e. there is a short exact sequence  $0 \to \calo_{X} \overset{s}{\to} E \to Q \to 0$.
We start by introducing the following two definitions.

\begin{definition} \label{defn:sat}
    A pure coherent pair $(E,s)$ is called \textit{saturated} if the cokernel of the induced monomorphism $s:\ox\into E$ is torsion-free.
    We say that the pair $(E, s)$ corresponds to a curve $Y$ if $Q\cong \mathcal{I}_Y(c_1(E))$.
\end{definition}

\begin{definition}{\label{def:tor-free wall}}
    Define $W_T\in \mathbb{Q}[t]_{>0}$ to be the maximum critical value (possibly empty) for $v(k)$ that is less than $W_{\emptyset}$. 
\end{definition}

The critical value $W_T$ in Definition \ref{def:tor-free wall} is well-defined because there are finitely many walls in the region $\delta\in (0, W_{\emptyset})$ (Theorem \ref{Thm:finite}). 
We will see in the next two results that every stable pair in the chamber $(W_T, W_{\emptyset})$ is saturated, and this chamber is relatively big.

\begin{proposition} \label{prop:saturated}
For any numerical class $v(k)$ $(k>0)$, set $\delta:=W_{\emptyset}-\varepsilon$ where 
$W_{\emptyset}$ is the collapsing wall for $v(k)$ and $\varepsilon$ is a positive polynomial of degree at most $n-2$; 
let $(E,s)$ be a pair with $v(E)=v(k)$. 
If $(E,s)$ is pure and $\delta$-stable, then $(E,s)$ is saturated.
\end{proposition}
\begin{proof}
If $(E,s)$ is not saturated, let $T$ be the maximal torsion sub-sheaf of \linebreak $Q:=\coker(s)$, and let $F$ be the kernel of the composed morphism $E\onto Q\onto Q/T$. This situation can be summarized in the following commutative diagram
\begin{equation} \label{diag:saturation}
\begin{split}
\xymatrix@-2ex{ 
&  &  & 0 \ar[d] \\
& 0\ar[d] &  & T \ar[d] &  \\
0\ar[r] & \ox \ar[r]^s \ar[d]_{s'} & E \ar[r] \ar@{=}[d] & Q \ar[r]\ar[d] & 0 \\
0\ar[r] & F \ar[r] & E \ar[r] & Q/T \ar[r]\ar[d] & 0 \\
&  &  & 0. }
\end{split}
\end{equation}
We then obtain the exact sequence $\ox\stackrel{s'}{\into} F \onto T$, so that $\rk(F)=1$ and
$$ \deg(F)=\deg(T)>0. $$ In addition, we obtain a sub-pair $(F,s')\into(E,s)$ with $s'\ne0$; note that
\begin{align*}
p_{(E,s)}^\delta(t) - p_{(F,s')}^\delta(t) & =    
p_E(t) - p_{F}(t) + \left(\dfrac{1}{\rk(E)}-1\right)\delta \\
& = p_{\ox}(t) - p_F(t) + \left(1-\dfrac{1}{\rk(E)} \right)\varepsilon \\
& = -\deg(F) \dfrac{t^{n-1}}{(n-1)!} + \text{lower order terms} < 0,
\end{align*} 
since $\deg(\varepsilon)\le n-2$. Therefore, $(E,s)$ is not $\delta$-stable.
\end{proof}

\begin{corollary}{\label{torfreechamber}}
    Any $\delta$-semistable pair $(E, s)$, for $\delta\in (W_T, W_{\emptyset})$, is saturated.
\end{corollary}

\begin{proof}
    If not, then diagram \ref{diag:saturation} in Proposition \ref{prop:saturated} implies that there is a sub-pair $(F, s')\hookrightarrow(E, s)$ which defines a wall at $\delta=W_{F}$. 
    By definition, $W_T\geq W_F$, and one checks that $(E,s)$ is unstable if $\delta>W_F$.
    Therefore, $(E,s)$ is $\delta$-unstable for $\delta\in (W_T, W_{\emptyset})$, which leads to a contradiction that $(E,s)$ is $\delta$-semistable.
\end{proof}

\begin{remark} 
We will compute a group of walls that appear before the collapsing wall in Section \ref{subsec:wall0dim}. There, we will see that $W_T$ is defined by the following sequence:
\[
0\lra (\cali_A(1), 1)\lra (E,s)\lra (\ox(c_1(v)-1), 0)\lra 0.
\]
One computes directly that 
$$W_T=\frac{\deg(E)+(k-1)\cdot \rk(E)}{(\rk(E)-1)(n-1)!}\cdot t^{n-1}+\text{lower order terms}.$$
Hence, $\displaystyle W_{\emptyset}-W_T=\frac{\rk(E)}{(\rk(E)-1)(n-1)!}\cdot t^{n-1}+\text{lower order terms}.$
In particular, the degree of $W_{\emptyset}-W_T$ is $n-1$, which proves Proposition \ref{prop:saturated}. 
Nevertheless, diagram (\ref{diag:saturation}) motivates the ``torsion-free" wall in Definition \ref{def:tor-free wall} and leads to Corollary \ref{torfreechamber}.
\end{remark}

\subsubsection{The first wall from the left}
We have seen in Proposition \ref{prop:stability} that when $\delta$ is sufficiently close to $0$, the underlying sheaf $E$ in a $\delta$-stable pair $(E,s)$ is stable. 

\begin{definition} \label{defn:vs}
For a fixed $\delta\in \mathbb{Q}[t]_{>0}$, a coherent pair $(E,s)$ is called \emph{very stable} if $(E,s)$ is $\delta$-stable and $E$ is Gieseker semistable. 
\end{definition}

Define $W_G$ to be the minimum critical value in $(0, W_{\emptyset})$. So $W_G$ is the first wall as $\delta$ varies from $0$ to $W_{\emptyset}$. 

For the case $c_1(v)=0$, $W_G$ can be sufficiently close to $0$, and this wall is not necessarily defined by a line bundle. 
For the case $c_1(v)=-1$, the first wall can be better controlled. 
We will provide some nice results below when $c_1(v)=-1$ and the underlying sheaf $E$ in a pair $(E,s)$ defining the wall $W_1$ is reflexive. 

\begin{proposition}{\label{Lem:1wallleft}}
    Let $v$ be a numerical class of rank $2$ with $c_1(v)=-1$. Then,
    \begin{enumerate}
        \item     
    For $k\in \mathbb{Z}_{>0}$, the minimum numerical wall $W_G$ for $v(k)$ is defined by the following sequence
    \begin{equation}{\label{1wallleft}}
    0\lra (\calo_{\p3}(k-1), 1) \lra (E(k), 1) \lra (\cali_Z(k), 0) \lra 0.
    \end{equation}
in which $v(E)=v$, and $Z\subset X$ is a codimension $2$ subvariety of $X$. 
\item
Moreover, if $E$ is reflexive, then the above sequence defines an (actual) wall for $v(k)$, if and only if $H^0(X, E(1))\neq 0$. 
\item 
In addition, as $\delta$ gets larger and crosses $W_G$, the following pairs (set theoretically) will become unstable.
\[
\{(E,s) ~|~ E\text{ is reflexive}, 0\neq s\in H^0(X, E(1)), H^0(X,E)=0 \}.
\]
\end{enumerate}
\end{proposition}

\begin{proof}
We know that a wall for $v(k)$ is defined by the following sequence (which is sequence (\ref{rank 2 wall}) in the present situation)
\[
0\lra (\cali_Z(l), 1)\lra (E(k),s)\lra (\cali_{Z'}(l'), 0)\lra 0,
\]
where $Z\in X$ and $Z'\subset X$ are sub-varieties of codimension at most $2$, and $l, l'\in \mathbb{Z}$.
The corresponding critical value $\delta$ is expressed as
\[
\delta=P_{E(k)}(t)-2P_{\cali_Z(l)}(t),
\]
which is a decreasing function of $P_{\cali_Z(l)}(t)$.
For $\delta$ to be positive, we must have $l\leq l'$.
So, with the fact that $c_1(E)=-1$, $\delta$ gets its minimum value when $l=k-1$ and $l'=k$, and the wall becomes 
\[
0\lra (\cali_Z(k-1), 1)\lra (E(k), 1)\lra (\cali_{Z'}(k), 0)\lra 0.
\]
Moreover, to get the minimum $\delta$, $P_{\cali_Z(k-1)}(t)$ should get its maximum. This is when $Z$ is the empty set, i.e. $\cali_Z(k-1)=\calo_X(k-1)$.
Therefore, the minimum numerical wall is defined by the sequence in display (\ref{1wallleft}), and this proves Statement (1).

For the second statement, if sequence (\ref{1wallleft}) defines a wall, then obviously, $H^0(X, E(1))\neq 0$. Conversely, if $H^0(X, E(1))\neq 0$, then let $Q$ be the cokernel of $s$, i.e. 
\[
0\lra \calo_X\lra E(1)\lra Q\lra 0.
\]
It is sufficient to prove that $Q$ is torsion-free. Then, the sequence 
\[
0\lra (\calo_X(k-1), 1)\lra (E(k),s)\lra (Q(k),0)\lra 0 
\]
defines a wall, and the above arguments imply that it is the minimum one.

Suppose $Q$ is not torsion-free. Let $T\hookrightarrow Q$ be its maximal torsion subsheaf. Consider the composition $E(1)\onto Q\onto Q/T$, and denote by $K$ the kernel of this composition. 
We have the following diagram:

\[
\begin{tikzcd}
&&&T\arrow[d]&\\
0\arrow[r]& \calo_{X}\arrow[r]\arrow[d]&E(1)\arrow[r]\arrow[d,"="]&Q\arrow[r]\arrow[d]&0\\
0\arrow[r]& K\arrow[r]\arrow[d]&E(1)\arrow[r]&Q/T\arrow[r]&0\\
&T&&&.\\
\end{tikzcd}
\]
Since $E$ is reflexive and $Q/T$ is torsion-free, $K$ is reflexive. Besides, $\rk(K)=1$, which implies that $K$ is a line bundle (\cite[Proposition 1.9]{hartshorne1980stable}).
Thus, $K\simeq\calo_X(s)$ for some $s>0$, and there is a morphism $K\simeq\calo_X(s)\lra E(1)$, which is a section of $H^0(X, E(1-s))$. Since $1-s<0$, we must have $H^0(X, E(1-s))=0$, which is absurd. 

Lastly, Statement (3) is straightforward by definition.
\end{proof}


\section{Moduli spaces of coherent pairs} \label{sec:moduli}

Let $X$ be a smooth projective variety over $\kappa$.
In \cite{lin2018moduli}, \cite{Wandel_2015}, \cite{he1996espaces}, the moduli space of stable pairs on $X$ was constructed using GIT, and the deformation theory was studied. In this section, we first recollect some of the general results and then focus on the case of rank $2$, showing a general wall-crossing picture that relates a Gieseker moduli space to a Hilbert scheme. 

\subsection{Moduli space of stable pairs}

In order to set up the moduli problem for coherent pairs, we need to introduce the appropriate notion of family and equivalence of families. 
For any numerical class $v$, let $P$ be the corresponding Hilbert polynomial for $v$. As a convention, we will use the Hilbert polynomial $P$ to define moduli functors rather than using the numerical class $v$.

A \textit{family of coherent pairs} on $X$ parametrized by a scheme $S$ of finite type of $\kappa$ is a pair $(\mathbf{E}, \mathbf{s})$ consisting of a coherent sheaf $\mathbf{E}$ on $X\times S$, and a morphism $\mathbf{s}:\mathcal{O}_{S}\to \pi_{S*}\mathbf{E}$ where $\pi_S:X\times S\to S$ is the projection onto the second factor. 
The morphism $\mathbf{s}$ is equivalent to $\pi_S^*\mathbf{s}: \mathcal{O}_{S\times X}\to \mathbf{E}$ by the adjoint property of $\pi_{S*}$ and $\pi_S^*$. 

A family $(\mathbf{E},\mathbf{s})$ over $S$ is \textit{flat} if $\mathbf{E}$ is flat over $S$. This is a special case of flat coherent systems introduced in \cite[Definition 1.8, Proposition 1.9]{he1996espaces}.

Two families $(\mathbf{E}, s)$ and $(\mathbf{F}, s')$ are equivalent if and only if there exist an isomorphism $\phi: \mathbf{E}\to \mathbf{F}$ and the commutative diagram holds
\[
\begin{tikzcd}
\calo_{X\times T}\arrow[r, "s"]\arrow[d, "\mathbf{1}"]&\mathbf{E}\arrow[d, "\phi"]\\
\calo_{X\times T}\arrow[r, "s'"]&\mathbf{F}.\\
\end{tikzcd}
\]

Consider the moduli functor $\mathfrak{S}^{\delta}_X(P):Sch\to Set$ that assigns to every scheme $R$ isomorphism classes of flat families of $\delta$-semistable pairs $(\mathbf{E},\mathbf{s})$ on $X\times_S R$ over $R$, such that the fiber $(\mathbf{E}_r, \mathbf{s}_r)$ at $r\in R$ is semi-stable with Hilbert polynomial $P_{\mathbf{E}_r}(t)=P$. 
In particular, $\mathbf{E}_r$ is torsion-free for all $r\in R$.
Define a sub-functor $\mathfrak{S}^{\delta,st}_X(P)$ concerning only stable pairs. The next theorem shows that there exists a (coarse) fine moduli space that (co-)represents those functors. 

\begin{theorem}[\cite{Wandel_2015} Theorem 3.8, \cite{lin2018moduli} Theorem 1.1, \cite{he1996espaces} Theorem 2.6]{\label{moduliofpairs}}
Let $(X, \calo_X(1))$ be a smooth polarized projective variety and $\delta\in \mathbb{Q}[t]_{> 0}$ a positive rational polynomial. Then, there exists a coarse moduli space $\cals^{\delta}_{X}(P)$ of $\delta$-semistable pairs co-representing the moduli functor $\mathfrak{S}^{\delta}_X(P)$. Two pairs correspond to the same point in $\cals^{\delta}_X(P)$ if and only if they are $S-$equivalent. 
Moreover, there is an open subset $\cals^{\delta,st}_X(P) \subset \cals^{\delta}_X(P)$ corresponding to stable pairs. This subset is a fine moduli space of $\delta$-stable pairs, i.e., it represents the functor $\mathfrak{S}^{\delta,st}_X(P)$.
\end{theorem}


\subsubsection{Deformation theory}

In \cite{lin2018moduli} and \cite{he1996espaces}, the tangent and obstruction theories for the moduli space of semistable pairs are studied in different contexts. In this subsection, we briefly recall their results. 
Then, we give a simple proof that for a pair $(E,s)$ with $H^1(X, E)=0$, the tangent spaces described in \cite{lin2018moduli} and \cite{he1996espaces} coincide. 

Let $\mathcal{A}rt_\kappa$ be the category of local Artinian $\kappa$-algebras with residue field $\kappa$. 
Let $A, B\in \mathcal{A}rt_\kappa$ and 
\begin{equation}{\label{smallextension}}
0\lra K \lra B \overset{\sigma}{\lra} A \lra 0
\end{equation}
be a small extension, i.e., $m_B K=0$. For a semistable pair $\Lambda=(E, s)$ on $X$, let $I^{\bullet}:=[\mathcal{O}_X\overset{s}{\to} E]$ be the corresponding complex concentrated in degree $0$ and $1$.  

\begin{theorem}[\cite{lin2018moduli} Theorem 1.2]{\label{tanobs1}}
Suppose $s_A: \mathcal{O}_X\otimes_\kappa A\to E_A$ is a morphism over $X_A:=X\times_{Spec(\kappa)}Spec(A)$ extending $s$, where $E_A$ is a coherent sheaf flat over $A$. There is a class 
$$ {\rm ob}(s_A, \sigma)\in \Ext^1(I^{\bullet}, E\otimes K)$$
such that there exists a flat extension of $s_A$ over $X_B$ if and only if ${\rm ob}(s_A, \sigma)=0$. If extensions exist, the space of extensions is a torsor under 
$$\Hom(I^{\bullet}, E\otimes K).$$
\end{theorem}

In \cite{he1996espaces}, the tangent and obstruction are described by extension classes of coherent pairs (coherent systems).

\begin{theorem}[\cite{he1996espaces} Definition 3.9 and Theorem 3.12]{\label{tanobs2}}
Let $\Lambda=(E,s)$ be an  $A-$flat coherent system on $X\times Spec(A)/Spec(A)$. The obstruction of extending $\Lambda$ to $Spec(B)$ is contained in the cohomology class $\Ext_A^2(\Lambda, \Lambda\otimes K)$. 

Let $\Lambda_0=(E_0, s_0)$ be a $\delta$-stable pair for some non-negative rational polynomial $\delta$, and $a$ is the corresponding point in the moduli space $\cals^{\delta}_{X}(P)$.  The tangent space at $a$ is isomorphic to $\Ext^1(\Lambda_0,\Lambda_0)$, and if $\Ext^2(\Lambda_0,\Lambda_0)=0$, then the moduli space $\cals^{\delta}_{X}(P)$ is smooth at $a$.
\end{theorem}

We check directly in the next lemma that for a stable pair $\Lambda=(E, s)$, in which $\rk(E)=2$ and $H^1(X, E)=0$, Theorem \ref{tanobs1} and Theorem \ref{tanobs2} compute the same tangent space at $\Lambda$.

\begin{lemma}
Let $\Lambda=(E,s)$ be a coherent pair of class $v$ with $ch_0(v)=2$. Assume that $(E,s)$ is $\delta$-stable for some $\delta\in \mathbb{Q}[t]_{>0}$, and let $I^{\bullet}=[\mathcal{O}_X \to E]$ be the corresponding complex concentrated in degree $0$ and $1$. If $H^1(X, E)=0$, then $\Hom(I^{\bullet}, E)\simeq \Ext^1(\Lambda, \Lambda)$. 
\end{lemma}

\begin{proof}
Let $Q$ be the cokernel of $s$, i.e. 
$$0\lra \mathcal{O}_X \overset{s}{\lra} E \lra Q \lra 0.$$ 
Then
$I^{\bullet}=[\mathcal{O}_X \overset{s}{\to} E]\simeq Q[-1]$, and $\Hom(I^{\bullet}, E)=\Hom(Q[-1], E)=\Ext^1(Q,E)$. 
Apply the functor $\Hom(-, E)$ to the sequence above, and we have the following long exact sequence
\[
\begin{tikzcd}
0 \arrow[r] & \Hom(Q,E)\arrow[r]& \Hom(E,E) \arrow[r] & \Hom(\mathcal{O}_X, E)&\\
\arrow[r]&\Ext^1(Q, E) \arrow[r]& \Ext^1(E, E)\arrow[r] & 0=\Ext^1(\mathcal{O}_X, E).\\
\end{tikzcd}
\]
It implies that $\text{hom}(I^{\bullet}, E)=\ext^1(E,E)+\text{hom}(\mathcal{O}_X,E)+\hom(Q,E)-\text{hom}(E,E)$.

On the other hand, considering the long exact sequence of coherent systems in Lemma \ref{LES}, we have
\[
\begin{tikzcd}
0 \arrow[r] & \Hom(\Lambda,\Lambda)\arrow[r]& \Hom(E,E) \arrow[r] & \Hom(s\cdot\mathbb{C}, H^0(X, E)/s\cdot \mathbb{C})\\
\arrow[r]&\Ext^1(\Lambda, \Lambda) \arrow[r]& \Ext^1(E, E)\arrow[r] & 0=\Hom(s\cdot\mathbb{C}, H^1(X, E)).\\
\end{tikzcd}
\]
This implies that
$$ \ext^1(\Lambda, \Lambda)=\ext^1(E,E)+\text{hom}(s\cdot\mathbb{C}, H^0(X, E)/s\cdot \mathbb{C})+\text{hom}(\Lambda, \Lambda)-\text{hom}(E,E). $$ 
Stability of $\Lambda$ gives $\text{hom}(\Lambda, \Lambda)=1$. 
Compare the expressions for  $\ext^1(\Lambda, \Lambda)$ and $\text{hom}(I^{\bullet}, E)$, and the claim would follow if $\Hom(Q,E)=0$. 

Suppose $\Hom(Q,E)\neq 0$. Consider the short exact sequence of pairs
$$0\lra (\mathcal{O}_X, \mathbf{1}_{\ox}) \lra (E, 1)\lra (Q, 0) \lra 0.$$

If $Q$ is torsion-free, then both $(\mathcal{O}_X, \mathbf{1}_{\ox})$ and $(Q,0)$ are stable for every stability parameter $\delta\in \mathbb{Q}[t]_{> 0}$. A non-zero morphism $Q\to E$ induces a non-zero morphism of pairs $\phi: (Q, 0) \to (E, 1)$. Consider the following diagram:

\[
\begin{tikzcd}
&&(Q, 0)\arrow[d, "\phi"]\arrow[dl, "\phi_1", swap]\arrow[dr, "\phi_2"]&&\\
0 \arrow[r] & (\mathcal{O}_X, \mathbf{1}_{\ox})\arrow[r]& (E, 1) \arrow[r, "p"] & (Q, 0) \arrow[r] &0.\\
\end{tikzcd}
\]
Since all the pairs in the diagram are stable, we must have
$$ p^{\delta}_{(\mathcal{O}_X, \mathbf{1})}(t)<p^{\delta}_{(E, 1)}(t) <p^{\delta}_{(Q,0)}(t). $$ 
Also, $\phi_1=0$ because $p^{\delta}_{(Q,0)}(t)>p^{\delta}_{(\mathcal{O}_X, \mathbf{1})}(t)$. 
Therefore, $\phi_2\neq 0$, which leads to $\phi_2\simeq\mathbb{C}\cdot id_{(Q,0)}$ since $(Q, 0)$ is $\delta$-stable. 
The above sequence of pairs is then split, and it contradicts the assumption that $(E, 1)$ is $\delta$-stable.

If $Q$ is not torsion-free. Let $T\hookrightarrow Q$ be the maximum torsion sub-sheaf, so $L:=Q/T$ is torsion-free. Since $\Hom((T, 0), (E, 1))=0$, a non-zero morphism $(Q,0) \to (E, 1)$ induces a non-zero morphism $(L, 0) \to (E,1)$. Let $(K, 1)$ be the kernel of the composition $(E, 1) \twoheadrightarrow (Q,0)\twoheadrightarrow (L, 0)$, and consider the following diagram   

\[
\begin{tikzcd}
&&(L, 0)\arrow[d, "\phi"]\arrow[dl, "\phi_1", swap]\arrow[dr, "\phi_2"]&&\\
0 \arrow[r] & (K, 1)\arrow[r]& (E, 1) \arrow[r, "p"] & (L, 0) \arrow[r] &0.\\
\end{tikzcd}
\]
For the same reason, $\phi_1$ has to be zero, which indicates that the sequence is split. Then, we get a contradiction that $(E,1)$ cannot be $\delta$-stable. 
Therefore, $\Hom(Q, E)=0$, and the proof is complete.
\end{proof}

\subsection{Wall-crossings for a rank 2 class}

When the rational polynomial $\delta\in \mathbb{Q}[t]_{> 0}$ varies with respect to the lexicographic order, there will possibly be some walls.
We call the region between two adjacent walls a \emph{chamber}. 
Such wall-crossing behaviors for coherent pairs were studied in \cite{Thaddeus1992StablePL} on a smooth curve and in \cite{he1996espaces} on $\mathbb{P}^2$. 

For the rest of this work, we will focus on a smooth projective threefold $X$ with a fixed polarization $\mathcal{O}_X(1)$. 
For a numerical class $v$, let $P$ be the corresponding Hilbert polynomial. 
Again, we will use $P$, rather than $v$, for moduli functors as a convention. 
Let $\mathfrak{G}_X(P)$ and $\mathfrak{H}_X(P)$ be the functors for the Gieseker moduli space and the Hilbert scheme for a fixed Hilbert polynomial $P$. To be precise:

\begin{itemize}
\item $\mathfrak{G}_X(P)$ denotes the functor $Sch\to Set$ that assigns every scheme $R$ to a torsion-free sheaf $\mathbf{E}$ on $X\times R$, such that the restriction $\mathbf{E}|_r$ to each fiber is Gieseker semistable with Hilbert polynomial $P$. Denote by $\calg_X(P)$ the coarse moduli space co-representing $\mathfrak{G}_X(P)$.
\item $\mathfrak{H}_X(P)$ denotes the Hilbert scheme functor $Sch\to Set$ that assign every scheme $R$ to the quotient $\calo_{X\times R}\onto \calo_Z$ with the Hilbert polynomial at each fiber equal to $P$. Denote by $\calh_X(P)$ the corresponding Hilbert scheme.
\end{itemize}

\begin{proposition}{\label{Mod:Gieseker}}
Fix a polynomial $P$ of degree $n$ and rank $r$. For every constant $0\le\delta<1/r$, there is a forgetful morphism $\Psi:\cals_X^\delta(P)\to \calg_X(P)$ given by \linebreak $\Psi(E,s)=E$. In addition, if $E$ is stable, then $\Psi^{-1}(E)=\mathbb{P}H^0(E)$.
\end{proposition}
\begin{proof}
When $\delta<1/r$, Proposition \ref{prop:stability} allows for the definition of a forgetful natural transformation between the moduli functors
$$ \Pi ~:~ \mathfrak{S}_X^{\delta}(P) \lra \mathfrak{G}_X(P) $$
given by
$$ \Pi(T)(\mathbf{E},\mathbf{s}) = \mathbf{E} $$
sending a flat family of $\delta$-semistable pairs to a flat family of semistable sheaves on $X$ parametrized by the same scheme $T$.
Consider the transformation of moduli functor $\mathfrak{G}_X(P)\to h_{\cals}$ and the composition $\mathfrak{S}^{\delta}_X(P)\to \mathfrak{G}_X(P)\to h_{\calg}$, where $\cals$ and $\calg$ are the coarse moduli spaces of semistable pairs and sheaves, and $h_{\cals}$ (resp. $h_{\calg}$) is the functor $\Hom(-, \cals)$ (resp. $\Hom(-, \calg)$). 
It follows from the universal property that there is an induced morphism of schemes $\cals\to \calg$. 
Lastly, Proposition \ref{prop:converse} implies that the fiber of this morphism is $\mathbb{P}H^0(E)$.
\end{proof}

\begin{remark}
In general, it is not clear whether $\Psi:\cals_X^\delta(P)\to \calg_X(P)$ is surjective; however, Proposition \ref{prop:converse} implies that the image of $\Psi$ contains $\calg_X^{\rm st}(P)$ (the stable locus in $\calg_X(P)$).
\end{remark}

\begin{remark}
If $E\in \calg_X(P)$ is strictly semistable, the fiber of $\Psi$ over $E$ is more intricate than $H^0(E)$. 
To illustrate this, we will consider the case when $\rk(E)=2$.
In this case, $E$ fits into the following short exact sequence:
\[
0\lra \cali_C\otimes L\lra E\lra \cali_D\otimes M\lra 0
\]
for some $C,D\subset X$ (codimension $\geq$ 2) and $L,M\in \Pic(X)$ with $P_{\calo_C}(t)=P_{\calo_D}(t)$ and $P_L(t)=P_M(t)$.
Then, $\Psi^{-1}(E)=(\eta, s)$ where $\eta\in \Ext^1(\cali_D\otimes M, \cali_C\otimes L)$ or $\eta\in \Ext^1(\cali_C\otimes L, \cali_D\otimes M)$, and $s$ is a section of the quotient sheaf, i.e. $\cali_D\otimes M$ or $\cali_C\otimes L$.
\end{remark}

Consider rank $2$ pairs $(E,s)$ with 
$v(E)=v$.
Define $L:=\det(E)$. 
When $(E,s)$ is saturated, we obtain the short exact sequence 
\begin{equation} \label{eq:sqc}
0 \lra \ox \stackrel{s}{\lra} E \lra \cali_C\otimes L \lra 0
\end{equation}
where $C\subset X$ is a closed subscheme of codimension at least 2. 
Note that the Hilbert polynomial $\tilde{P}$ of the subscheme $C$ is given by the formula
\begin{equation}\label{tilde P}
    \begin{array}{rl}
    \tilde{P}:=&P_{\OC}(t)=P_{\ox}(t) + P_{L^{\vee}}(t) - P_{E\otimes L^{\vee}}(t)  \\
         =&\displaystyle P_{\ox}(t)+\int_X e^{-c_1(v)}\cdot \ch(\calo_X(t))\cdot Td_X\\
         &\displaystyle -\int_X (\sum_{i=0}^3 ch_i(v))\cdot e^{-c_1(v)}\cdot \ch(\ox(t))\cdot Td_X\\
    \end{array}
\end{equation}
which is uniquely determined by class $v$ and the polarization $\calo_X(1)$.

For any line bundle $L\in \Pic(X)$, define the following moduli space 
\[
\cals^{\delta}_{X,L}:=\{(E,s)\in \cals^{\delta}_X| \det(E)=L\}.
\]
Similarly, define the following moduli space:
\[
\calg_{X,L}:=\{E\in \calg_X|\det(E)=L\}.
\]
We have the following proposition connecting the moduli space of pairs to the Hilbert scheme.

\begin{proposition}{\label{Mod:Hilb}}
Fix a polynomial $P$ of degree $n$ and rank $2$, and a line bundle $L\in \Pic(X)$.
For every positive polynomial $\varepsilon$ of degree at most $n-2$, there is a morphism $\Gamma:\mathcal{S}_{X, L}^\delta(P)\to \calh_X(\tilde{P})$ where $\delta=W_{\emptyset}-\varepsilon$ and $\tilde{P}$ is a polynomial given by formula (\ref{tilde P}); in addition, the fiber $\Gamma^{-1}(C)$ is isomorphic to $\mathbb{P}H^0(\omega_C\otimes\omega_X^\vee\otimes L^{\vee})$.
\end{proposition}

\begin{proof}
Given a flat family of $\delta$-semistable pairs $(\mathbf{E},\mathbf{s})$ parametrized by a scheme $T$, consider the sheaf $\mathbf{I}:=\coker(\mathbf{s})$ on $X\times T$. Note that $\mathbf{I}$ is flat over $T$, since $\mathbf{s}:\mathbf{\calo}_{X\times T}\to\mathbf{E}$ is a monomorphism and both $\mathbf{\calo}_{X\times T}$ and $\mathbf{E}$ are flat over $T$.

Let $P':=P_E(t)-P_{\calo_X}(t)$.
According to Proposition \ref{prop:saturated}, every point $[(E,s)]$ in $\cals_{X,L}^\delta(P)$ is saturated. Therefore, $\mathbf{I}$ is a flat family of rank $1$ torsion-free sheaves parametrized by $T$ with Hilbert polynomial $P'$ and fixed determinant $L$. 
The natural morphism $\mathbf{I}\to\mathbf{I}^{\vee\vee}$ is injective because $\mathbf{I}\otimes\calo_{X\times\{t\}}$ is torsion-free for every $t\in T$; in addition, $\mathbf{I}^{\vee\vee}$ is an invertible sheaf \cite[Lemma 6.13]{10.4310/jdg/1214445046}, so we can use the natural short exact sequence (note that $(\mathbf{I}^{\vee\vee})^\vee\simeq\mathbf{I}^{\vee}$)
$$ 0 \lra \mathbf{I}\otimes\mathbf{I}^\vee \lra \mathbf{I}^{\vee\vee}\otimes\mathbf{I}^\vee \simeq \calo_{X\times T} \lra \calo_Z \lra 0, $$
with $Z\subset X\times T$ being a closed subscheme, and $\calo_Z$ is flat over $T$. 

Analogously to the proof of Proposition \ref{Mod:Gieseker}, there is a transformation of functors $\mathfrak{S}^{\delta}_{X, L}(P)\to \mathfrak{G}_X(P')$, where 
$\mathfrak{S}^{\delta}_{X,L}$ and
$\mathfrak{G}_X$ are the moduli functors for semistable pairs and Gieseker semistable sheaves.
The above short exact sequence and the flatness of $\mathbf{I}$ imply that the moduli functor $\mathfrak{G}_X(P')$ is equivalent to the Hilbert functor $\mathfrak{H}_X({\tilde{P}})$.

Therefore, we obtained a natural transformation between the moduli functors $\mathfrak{S}_{X, L}^{\delta}(P)\to \mathfrak{H}_X({\tilde P})$, which induces the desired morphism between the corresponding moduli spaces.
\end{proof}

\begin{remark}
    For any scheme $T$ and a flat family of pairs $\mathbf{\calo}_{X\times T}\overset{s}{\to} \mathbf{E}$ on $X\times T$, the cokernel of $s$ is flat over $T$ because of the functoriality of $\mathfrak{S}_{X}^{\delta}(P)$. 
    Besides, being torsion-free is an open condition in a flat family. Therefore, being saturated is an open condition in a flat family of pairs.   
\end{remark}

\begin{remark}
    When $\Pic(X)=\mathbb{Z}$, every line bundle $L\in \Pic(X)$ is uniquely determined by its degree.
    Therefore, there is no need to specify the line bundle $L$ in $\cals_{X,L}^{\delta}(P)$, and we have $\cals^{\delta}_{X,L}(v)=\cals^{\delta}_{X}(v)$. In addition, the morphism $\Gamma$ in Proposition \ref{Mod:Hilb} becomes 
    \[
    \Gamma: \cals^{\delta}_X(P)\to \calh_X(\tilde{P}).
    \]
\end{remark}

\begin{remark}
    In Proposition \ref{Mod:Hilb}, the fixed line bundle $L$ plays a crucial role in establishing a relationship between $\cals^{\delta}_{X, L}$ and $\calh_X(\tilde{P})$.
    Without $L$, this direct connection is lost.
    Consider the following transformations of moduli functors
    \begin{enumerate}
        \item Functor $F_1$: $
\mathfrak{S}^{\delta}_X(P) \to \mathfrak{G}_X(P')$, where $P':=P-P(\ox)$.

    \noindent For every scheme $T$, $F_1(T)$ maps a family of pairs $(\mathbf{E}, \mathbf{s})$ over $T$ to its cokernel $\mathbf{I}$. Note that 
    $\mathbf{I}$ is a family of rank $1$ torsion-free sheaves of the form $\cali_Y\otimes \calo_X(c_1(v))\otimes M$, where $Y\subset X$ has codimension at least $2$, $c_1(v)$ is the first Chern character of $v$, and $M$ is a line bundle in $\Pic^0(X)$. 
    \item 
    \text{Functor} $F_2$: $\mathfrak{G}_X(P')\to \mathfrak{P}$, where $\mathfrak{{P}}$ is the Picard functor.
    
    \noindent For every scheme $T$, $F_2(T)$ sends a family of rank $1$ sheaves $\mathbf{I}$ over $T$ to $\det(\mathbf{I})\otimes \calo_{X\times T}(-c_1)$, which is a family of line bundles of degree $0$.
     \end{enumerate}
     
    These two functors induce the following morphisms between moduli spaces:
    \begin{enumerate}
        \item $\Gamma_1: \cals^{\delta}_X(P)\to \calg_X(P').$
        \item $\Gamma_2: \calg_X(P')\to \Pic^0(X).$
    \end{enumerate}
    
    Moreover, the fibers of $\Gamma_1$ at $M\in \calg_X(P')$ and $\Gamma_2$ at $L\in \Pic^0(X)$ are $\mathbb{P}\Ext^1(M, \calo_X)$ and $\calh_X(\tilde{P})$ respectively. 
    In particular, if we consider the moduli space $\cals^{\delta}_{X,L}(P)$ or if $\Pic(X)=\mathbb{Z}$, then $\Gamma_2$ becomes trivial, and $\calg_X(P')\simeq \calh_X(\tilde{P})$. In this case, $\Gamma_1$ is reduced to $\Gamma$ in Proposition \ref{Mod:Hilb}.
    \end{remark}

Since being torsion-free and Gieseker semistable are open conditions in flat families, we have the following corollary.

\begin{corollary}{\label{ratl map 1}}
    Let $v$ be any numerical class of rank $2$ with $c_1(v)>0$ and $\delta$  be any polynomial in $\mathbb{Q}[t]_{>0}$ that is not a critical value for $v$. For every irreducible component $\calp\subset \mathcal{S}_{X,L}^{\delta}(v)$, we have
    \begin{enumerate}
        \item if there exists a saturated semistable pair $(E,s)\in \calp$, then there exists a rational map $\Gamma: \calp \dashrightarrow \calh_X(\tilde{P})$.
        \item if there exists a pair $(E,s)\in \calp$ such that $E$ is Gieseker semistable, then there is a rational map $\Psi: \calp \dashrightarrow \calg_X(P)$.
    \end{enumerate}
\end{corollary}

Now we describe in the following theorem the framework of wall-crossings for a rank $2$ class with an illustration in Figure \ref{Gieseker-Hilbert}. The theorem follows from the definitions of $W_G$, $W_T$ in Section $2.4$, and Propositions \ref{Mod:Gieseker} and \ref{Mod:Hilb}.

\begin{theorem}{\label{Gie-Hilb}}
Let $X$ be a smooth threefold with a fixed ample class $\ox(1)$ and a fixed line bundle $L\in \Pic(X)$. 
Then, on the positive axis of rational coefficient polynomials ($\mathbb{Q}[t]_{>0}$) ordered lexicographically, there is a collapsing wall $W_{\emptyset}$, after which no pair is $\delta$-semistable. 
There are two more walls $W_G \leq W_T$ (possibly empty) between $0$ and $W_{\emptyset}$, such that there exist morphisms from the moduli space $\cals^{\delta}_{X, L}(P)$ to $\calg_{X,L}(P)$ and $\calh_X(\tilde{P})$ for $\delta\in(0, W_1)$ and $\delta\in(W_T, W_{\emptyset})$. 
The fibers over $E\in \calg^{st}_{X,L}(P)$ and $[C]\in \calh_X(\tilde{P})$ are $\mathbb{P}H^0(X, E)$ and $\mathbb{P}\Ext^1(\cali_C\otimes L, \calo_X)$ respectively.
\end{theorem}

\begin{figure}[ht]
\begin{tikzpicture}
\draw[thick, ->] (-0.5,0)--(10,0);
\draw (9,-0.2) -- (9, 0.2);
\node at (9, -0.5){$W_{\emptyset}$};
\draw (2,-0.2) -- (2, 0.2);
\node at (2, -0.5){$W_{G}$};
\node at (4.5, -0.5){$\cdots$};
\draw (7,-0.2) -- (7, 0.2);
\node at (7, -0.5){$W_{T}$};
\node at (0.5, 0.5){$\mathbb{Q}[t]-$axis for $\delta$};
\node at (0, -0.5){$0$};
\draw (0,-0.2) -- (0, 0.2);
\draw [decorate,decoration={brace,amplitude=5pt,mirror,raise=5ex}]
   (7,0) -- (9,0) node[midway,yshift=-4em]{``Hilbert" chamber};
\draw [decorate,decoration={brace,amplitude=5pt,mirror,raise=5ex}]
   (0,0) -- (2,0) node[midway,yshift=-4em]{``Gieseker" chamber};
\end{tikzpicture}
\caption{Gieseker-Hilbert}
\label{Gieseker-Hilbert}
\end{figure}

\begin{remark}
    We are interested in the components $\calp\subset \cals^{\delta}_{X, L}(v)$ such that a general pair in $\calp$ is both very stable and saturated.
    This is when the modular Serre correspondence occurs, and there exist rational maps from $\calp$ to both the Gieseker moduli space and the Hilbert scheme (shown in Figure \ref{Serrecor1}). 
    These components, which always exist under general conditions, are contained in the moduli space $\cals^{\delta}_{X, L}(v)$ for all $\delta\in(0, W_{\emptyset})$. 
\end{remark}

\begin{figure}[ht]
\begin{tikzcd}   &\calp\arrow[dl, dotted, "\Psi", swap]\arrow[dr, dotted, "\Gamma"]&\\
\calg_X(v)&&\calh_X(\tilde{P})\\
\end{tikzcd}
\caption{Modular Serre correspondence}
\label{Serrecor1}
\end{figure}


\section{General wall-crossings at critical values} \label{sec:wallx}

This section studies general wall-crossings for a rank $2$ class on a smooth polarized projective threefold $(X, \calo_X(1))$ with $\Pic(X)=\mathbb{Z}$. 
We begin with a rank $2$ normalized class $v$ and then consider its twists $v(k)=v\otimes \calo_X(k)$ for $k\in \mathbb{Z}_+$ (as in Section \ref{sec:pairs}). This ensures that sheaves $E$ with $v(E)=v$ will have sufficiently many sections to facilitate wall-crossings. 
While $k=1$ is sufficient for all examples in Section \ref{sec:ex}, choosing an optimal twist in general is non-trivial: different $k$ correspond to different Hilbert schemes and extension classes.

We proceed as follows. First, we describe the general wall-crossing behavior for a rank $2$ class. Because this procedure connects the Gieseker moduli space to a Hilbert scheme, we next characterize the curves in the Hilbert scheme that are perturbed at a wall. Finally, we identify the set of bad pairs — those are just “temporarily stable" as the stability parameter varies.

For notational simplicity, we adopt the following conventions.
The moduli space $\cals^{\delta}_X(v(k))$ for $\delta=W$ is denoted by $\cals_{W}$, omitting $X$ and $v(k)$ if no confusion arises.
Similarly, $\cals_{W^+}$ (resp. $\cals_{W^-}$) denotes the moduli space $\cals^{\delta}_X(v(k))$ for $\delta$ in the chamber above $W$, i.e. $\delta\in (W,W')$ (resp. below $W$, i.e. $\delta\in (W'',W)$), where $W'$ and $W''$ are the adjacent walls to $W$.

\subsection{Wall-crossing phenomena}\label{subsec:wallxph}
Wall-crossings occur naturally when stability conditions vary.
Here, we describe walls for rank $2$ coherent pairs and the transformations for moduli spaces when crossing walls.

\subsubsection{Description of walls for rank 2 pairs}
Recall that a wall for $v(k)$ is defined by a short exact sequence as follows (Section \ref{rk 2 wall})
\begin{equation}{\label{wall2}}
W: \quad 0\lra (\mathcal{I}_A(l), 1) \lra (E(k), s) \lra (\mathcal{I}_B(2k-l+c_1), 0)\lra 0,
\end{equation}
where $v(E)=v$, $A,B\subset X$ are one-dimensional schemes (possibly empty), and $0\leq l \leq k$.
Denote by $W'$, the flipped sequence of $W$, i.e.
\begin{equation}{\label{flipped wall}}
W': \quad 0\lra (\mathcal{I}_B(2k-l+c_1), 0) \lra (F(k), s') \lra (\mathcal{I}_A(l), 1)\lra 0. 
\end{equation}
The critical value $\delta$ for $W$ (equivalently for $W'$) is the solution to the equation $$p^{\delta}_{(\cali_A(l), 1)}=p^{\delta}_{(\cali_{B}(2k-l+c_1), 0)}=p^{\delta}_{(E(k), s)}$$
By definition (see equation (\ref{Hilb poly})), we have
$$P_{\mathcal{I}_A(l)}(t)+\delta=\frac{P_{E(k)}(t)+\delta}{2}=P_{\mathcal{I}_B(2k-l+c_1)}(t),$$
and one computes directly that
$$\delta=P_{E(k)}(t)-2P_{\mathcal{I}_A(l)}(t)=-P_{E(k)}(t)+2P_{\mathcal{I}_B(2k-l+c_1)}(t).$$ 
Evidently, $\delta$ is decreasing  with respect to $P_{\mathcal{I}_A(l)}(t)$ and increasing with respect to $P_{\mathcal{I}_B(2k-l+c_1)}(t)$. 
We showed in Section \ref{sat wall} that the largest wall is the collapsing wall $W_{\emptyset}$, obtained when $l=0$ and $A=\emptyset$. 
As $\delta$ gets smaller, the second largest wall is defined by the pair $(\mathcal{I}_A(1), 1)$ (resp. $(\mathcal{I}_B(2k-1+c_1), 0)$), where $P_{\mathcal{I}_A(1)}(t)$ gets its minimum (resp. $P_{\mathcal{I}_B(2k-1+c_1)}(t)$ gets its maximum). 
This occurs when $B=\emptyset$, so the second largest wall is defined by the following sequence:

\begin{equation}{\label{secondwall}}
0\lra (\mathcal{I}_A(1), 1) \lra (E(k), s) \lra (\mathcal{O}_{\mathbb{P}^3}(2k-1+c_1), 0)\lra 0,
\end{equation}
where $A\subset \mathbb{P}^3$ is a one-dimensional scheme. 
Following this pattern, we introduce groups of walls as follows.

\begin{definition}{\label{wall group}}
For any $e\in \mathbb{Z}_{\geq 0}$, we say that a wall $W$ for $v(k)$ is in \textit{group $e$} if it is defined by a sequence in the following form
\[
0\lra (\mathcal{I}_A(e), 1) \lra (E(k), s) \lra (\cali_B(2k-e+c_1), 0)\lra 0,
\]
where $A, B\subset X$ are one-dimensional schemes. 
\end{definition}

Figure \ref{Groups of walls} below illustrates the locations of walls in different groups.

\begin{figure}[ht]
\begin{tikzpicture}
\draw[thick, ->] (-0.5,0)--(11.9,0);
\draw (11,-0.2) -- (11, 0.2);
\node at (11, -0.5){$W_{\emptyset}$};
\draw (10,-0.2) -- (10, 0.2);
\node at (10, -0.5){$W_{1,1}$};
\node at (9, -0.5){$\cdots$};
\draw (8,-0.2) -- (8, 0.2);
\node at (8, -0.5){$W_{1,n_1}$};
\draw (7,-0.2) -- (7, 0.2);
\node at (7, -0.5){$W_{2,1}$};
\node at (6, -0.5){$\cdots$};
\draw (5,-0.2) -- (5, 0.2);
\node at (5, -0.5){$W_{2,n_2}$};
\node at (0.5, 0.5){$\mathbb{Q}[t]-$axis for $\delta$};
\node at (0, -0.5){$0$};
\draw (0,-0.2) -- (0, 0.2);
\draw [decorate,decoration={brace,amplitude=5pt,mirror,raise=5ex}]
  (8,0) -- (10,0) node[midway,yshift=-4em]{group 1};
\draw [decorate,decoration={brace,amplitude=5pt,mirror,raise=5ex}]
  (10.5,0) -- (11.5,0) node[midway,yshift=-4em]{only wall};
\node at (11,-1.8){in group 0}; 
\draw [decorate,decoration={brace,amplitude=5pt,mirror,raise=5ex}]
  (5,0) -- (7,0) node[midway,yshift=-4em]{group 2};
\node at (2.5,-0.5){$\cdots$};  
\end{tikzpicture}
\caption{Groups of walls}
\label{Groups of walls}
\end{figure}

\begin{remark}
    For simplicity, we abuse the notation a bit for the rest of this paper. If there is no ambiguity, $W$ can denote either the short exact sequence (\ref{flipped wall}) or (\ref{secondwall}) when referring to the short exact sequence defining a wall. We also use $W$ to denote the critical value $\delta$ of that wall. 
\end{remark}

\subsubsection{Description of the contracting loci on two sides of a wall}

Let $W$ be a wall defined by the short exact sequence
\begin{equation}{\label{wall}}
W: \quad 0\lra (\mathcal{L}, 1) \lra (E,1)\lra (\mathcal{L}',0)\lra 0,
\end{equation}
where $\mathcal{L}, \mathcal{L}'\in \Coh(X)$ are torsion-free sheaves of rank $1$.
It is straightforward to check that when $\delta<{W}$ (resp. $\delta>{W}$), we have the inequality $p_{(\call, 1)}{(t)}<p_{(E, 1)}({t})$ (resp. $p_{(\call, 1)}{(t)}>p_{(E, 1)}({t})$). 

$\bullet$ When $\delta={W}$, there is a closed subscheme $\cals_{W}^{ss} \subset \cals_{W}$ parameterizing strictly semistable pairs $(\mathcal{L}, 1)\bigoplus (\mathcal{L}', 0)$ in the moduli space $\cals_{W}$. 

$\bullet$
For $\delta$ slightly smaller than  $W$, i.e. $\delta\in ({W}-\varepsilon, {W})$, denote the corresponding moduli space of semistable pairs by $\cals_{W^-}$. 
There is a (locally) closed subscheme $\cals_{W}^{-}\subset \cals_{W^-}$ parameterizing extensions in sequence (\ref{wall2}) with a morphism 
$$\phi_{W}^{-}: \cals_{W}^{-}\lra \cals_{W}^{ss}.$$ 
The fiber of $\phi^-_W$ over $(\call, 1)\bigoplus(\call', 0)\in \cals_W^{ss}$ is $\mathbb{P}\Ext^1((\mathcal{L}',0), (\mathcal{L}, 1))$.

$\bullet$
Similarly, when $\delta \in (W, W+\varepsilon)$, there is a (locally) closed subscheme $\cals_{W}^{+}\subset \cals_{W^+}$ parameterizing extensions in the flipped sequence  
$$W': \quad 0\lra (\mathcal{L}',0) \lra (E,1)\lra (\mathcal{L}, 1)\lra 0.$$ 
There is a morphism $\phi_{W}^+: \cals_{W}^+\to \cals_{W}^{ss}$ whose fiber over $(\mathcal{L}, 1)\bigoplus (\mathcal{L}', 0)\in \cals_W^{ss}$ is $\mathbb{P}\Ext^1((\mathcal{L}, 1), (\mathcal{L}',0))$.

$\bullet$
Let $\Phi^{\pm}_{W}: \cals_{W^{\pm}}\to \cals_{W}$ be the morphisms between the moduli spaces of pairs. 
We summarize all the above notations in the following commutative diagram:

\begin{figure}[ht]
\begin{tikzcd}
\cals_{W^{-}}\arrow[r,"\Phi^-_{W}"]&\cals_{W}&\cals_{W^{+}}\arrow[l, "\Phi^+_{W}" above]\\
\cals^-_{W}\arrow[r, "\phi^-_{W}"]\arrow[u,hook]&\cals^{ss}_{W}\arrow[u,hook]&\cals^+_{W}\arrow[l, "\phi^+_{W}" above]\arrow[u,hook].\\
\end{tikzcd}
\caption{Wall-crossing morphisms.}{\label{wall-crossing2}}
\end{figure}

\subsubsection{(Birational) transformations at wall-crossings}{\label{birat'l trans}}

When crossing a wall, how the subschemes $\cals_W^{\pm}\subset \cals_{W^{\pm}}$, see Figure (\ref{wall-crossing2}), are glued to the ambient spaces $\cals_{W^{\pm}}$ can be studied using the elementary modification. 
This technique for sheaves can be found in \cite[Page 41]{friedman1998algebraic}.
In the setting of stale pairs, this technique can be found in \cite[Section 3]{Thaddeus1992StablePL} for smooth curves, and in \cite[Section 4]{he1996espaces} for $\p2$. 
In the setting of Bridgeland stability, this technique in derived categories can be found in \cite{MR2998828}, \cite{2014MMP}, \cite{MR3803142}.

If $X$ is a smooth curve or surface, then the moduli space (stack) of semistable objects (either as pairs or complexes) on $X$ is irreducible, and crossing walls results in birational transformations for the moduli space. 
The transformation is either a \emph{flip} or a \emph{divisorial contraction}.
In this context, the term flip refers to the birational transformations arising from wall-crossing in the sense of  \cite[Definition 1.5]{bradlow2003coherent}. 
More precisely, a flip $f: \cals_{W^-}\dashrightarrow \cals_{W^+}$ (Figure \ref{flip}) means that the morphisms $\Phi^{\pm}_W$ are small contractions, i.e. the subschemes $\cals_W^{\pm}\subset \cals_{W^{\pm}}$ contracted by $\Phi^{\pm}_{W}$ have codimensions at least two.
We say that $f$ is a divisorial contraction if the subscheme $\cals_W^-\subset \cals_{W^-}$ contracted by $f$ has codimension $1$ and $\Phi_W^+$ is an isomorphism. 

\begin{figure}[ht]
\label{flip}
\begin{tikzcd}
\cals_{W^{-}}\arrow[dr,"\Phi^-_{W}"]\arrow[rr, dashed, "f"]&&\cals_{W^{+}}\arrow[dl, "\Phi^+_{W}" above]\\
\cals^-_{W}\arrow[dr, "\phi^-_{W}"]\arrow[u,hook]&\cals_{W}&\cals^+_{W}\arrow[dl, "\phi^+_{W}" above]\arrow[u,hook]\\
&\cals^{ss}_{W}\arrow[u,hook]&\\
\end{tikzcd}
\caption{Flip}
\end{figure}

The elementary modification for pairs on $\p2$ (\cite[Section 4]{he1996espaces}) works for $\mathbb{P}^3$ as well. 
So in later sections, we only focus on computing the extension classes to determine the type of transformations. 
For the details of this gluing technique, we refer to \cite[Section 4.8]{he1996espaces}.

On a smooth threefold, wall-crossings for moduli spaces of Bridgeland semistable objects were described in some situations.
Except for a flip (or a flop), it can also be a divisorial contraction or adding (removing) a component (\cite{schmidt2020bridgeland,MR3803142,mu2024new}).
We will see later in Sections 5 and 6 that wall-crossings for stable pairs on $\mathbb{P}^3$ are similar, which means it can be a flip, divisorial, or adding (removing) a component.

\subsection{One-dimensional schemes created when crossing a wall}\label{subsec:wall1dim}

Consider a smooth projective threefold $(X,\calo_X(1))$ with $\Pic(X)=\mathbb{Z}$; set  $d:=c_1(\calo_X(1))^3$.
Let $W$ is a wall defined by the sequence 
\begin{equation}{\label{wall1}}
0\lra (\mathcal{I}_A(k+l), 0)\lra (E(k),s)\lra (\mathcal{I}_B(k-l+c_1), 1) \lra 0,
\end{equation}
where $A,B\subset X$ are one-dimensional schemes (possibly empty) and $0\leq l \leq k+c_1$.
When $\delta$ gets larger and crosses $W$,
the (locally) closed scheme $\cals_W^+$ is created in the moduli space, see Figure \ref{wall-crossing2}.   
We know from Proposition \ref{bad pair} that a general stable pair  $(E(k), s)\in \cals_W^+$ is saturated, and hence it corresponds to a curve $Y\subset X$.
Next, we describe this one-dimensional scheme $Y$ created at a wall given by sequence (\ref{wall1}). 

\subsubsection{Description of $Y$}
Let $(E(k),s)\in \cals_W^+$ be a saturated pair corresponding to a one-dimensional scheme $Y$. 
By Definition \ref{defn:sat}, we have
$$ 0\lra\calo_X\overset{s}{\lra} E(k) \lra \cali_Y(2k+c_1) \lra 0. $$ 
Besides, sequence  (\ref{wall1}) implies the following diagram ($a:=k-l+c_1$)

\begin{equation}{\label{diagrampair}}
\begin{tikzcd}
{}&{}&\mathcal{O}_{X}\arrow{r}\arrow[d, "s"]&\mathcal{O}_{X}\arrow[d, "s_1"]&\\
0\arrow{r}&\mathcal{I}_A(k+l)\arrow{r}\arrow[d, "="]&E(k)\arrow{r}\arrow{d}&\mathcal{I}_B(a)\arrow{r}\arrow{d}&0\\
0\arrow{r}&\mathcal{I}_A(k+l)\arrow{r}&\mathcal{I}_Y(2k+c_1)\arrow{r}&\mathcal{I}_{B|S}(a)\arrow{r}&0,\\
\end{tikzcd}
\end{equation}
in which the section $s_1$ determines a hypersurface $S\subset X$ of degree $a$. 

Let $A_0, B_0$ (resp. $A_1, B_1$) be the maximum zero-dimensional (resp. one-dimensional) subschemes of $A$ and $B$, and let $d_A$ (resp. $d_B$) be the degree of $A_1$ (resp. $B_1$). The bottom row of diagram (\ref{diagrampair}) implies the following diagram 
\[
\begin{tikzcd}
{0}\arrow{r}&\mathcal{I}_A(k+l)\arrow{r}\arrow[d, "i_1"]&\mathcal{I}_{Y}(2k+c_1)\arrow{r}\arrow[d, "i_2"]&\mathcal{I}_{B|S}(a)\arrow{r}\arrow[d]&0\\
0\arrow{r}&\mathcal{O}_{X}(k+l)\arrow{r}\arrow{d}&\mathcal{O}_{X}(2k+c_1)\arrow{r}\arrow{d}&\mathcal{O}_S(2k+c_1)\arrow{r}\arrow{d}&0\\
0\arrow{r}&\mathcal{O}_A(k+l)\arrow{r}&\mathcal{O}_Y(2k+c_1)\arrow{r}&F\arrow{r}&0,\\
\end{tikzcd}
\]
in which $i_1$ and $i_2$ are injections to their reflexive hulls, and the third column is obtained via the snake lemma; the sheaf $F$ is just the cokernel of the monomorphism $\mathcal{O}_A(k+l)\to\mathcal{O}_Y(2k+c_1)$.
In addition, all three columns are exact sequences. 

The third column can be rewritten as 
$$ 0\lra \mathcal{I}_{B_0|S}(a^2H^2-B_1)\lra \mathcal{O}_S(2k+c_1)\lra F\lra 0. $$ 
Hence, $F$ fits into the sequence
$$ 0\lra \mathcal{O}_{B_0}\lra F\lra \mathcal{O}_C(2k+c_1)\lra 0, $$
where $C\in |\mathcal{O}_S(a(k+l)H^2+B_1)|$ is a divisor in $S$. This leads to the following diagram describing $Y$

\begin{equation}{\label{descriptionY}}
\begin{tikzcd}
&&&\mathcal{O}_{B_0}\arrow{d}&\\
0\arrow{r}&\mathcal{O}_A(-k+l-c_1)\arrow{r}&\mathcal{O}_Y\arrow{r}&F(-2k-c_1)\arrow[d]\arrow{r}&0\\
&&&\mathcal{O}_C.&\\
\end{tikzcd}
\end{equation}

Roughly speaking, $Y$ is the union (scheme theoretically) of $A$, $C$ and $B_0$, subject to the conditions that $C, B_0\subset S$ and the length of $A_1\cap C$ is $|A_1\cdot aH|=|S\cdot A_1|$.

\begin{remark}{\label{curvetopair}}
Diagram (\ref{descriptionY}) does not necessarily recover diagram (\ref{diagrampair}) or sequence  (\ref{wall1}). We will show a counter-example in the last section (compare with Lemma \ref{singularE} and Remark \ref{counterexample}).
\end{remark}

\begin{remark}
Assume that $\cdots<W_3 <W_2<W_1$ are walls in group $e>0$ (see Definition \ref{wall group}). Let $S\subset X$ be a hypersurface of degree $e$, and $C_i$ be the one-dimensional scheme corresponding to a general saturated pair in $\cals_{W_i}^+$. The above argument shows that $C_i$ is obtained from $C_{i-1}$ by setting some component in $C_{i-1}$ free from the hypersurface $S$.
\end{remark}

\subsubsection{Two fibrations for $\cals_W^+$} \label{subsec:fibrations}

Let $W\in \mathbb{Q}[t]_{>0}$ be a wall, and assume that all pairs in $\cals_W^+$ are saturated. 
Let 
$$\calh_W:=\{Y\subset X| Y\ \text{satisfies diagram (\ref{descriptionY}) for}\ W \}\subset \calh_X(\tilde{P})$$ 
be the subscheme parameterizing curves $Y$ satisfying diagram (\ref{descriptionY}) ($\tilde{P}$ is the Hilbert polynomial of $\calo_Y$). 
There are two morphisms $\cals_W^+\to \calh_X(\tilde{P})$ and $\cals_W^+\to \cals^{ss}_W$ whose fibers are
$\mathbb{P}\Ext^1(\cali_Y(2k+c_1), \calo_X)$
and 
$\mathbb{P}\Ext^1((\cali_B(a), 1), (\cali_A(k+l), 0))$
respectively (illustrated in Figure \ref{Fibrations}, $a=k-l+c_1$). 

\begin{figure}[ht] 
\begin{tikzcd}
\mathbb{P}\Ext^1((\cali_{B}(a), 1), (\cali_{A}(k+l), 0))\arrow[dr, "fiber"]&&\mathbb{P}\Ext^1(\mathcal{I}_{Y_i}(2k+c_1), \mathcal{O}_{\mathbb{P}^3})\arrow[dl, "fiber", swap]\\
&\cals^+_{W}\arrow[dr]\arrow[dl]&\\
\calh_W&&\cals^{ss}_{W}\\
\end{tikzcd}
\caption{Fibrations}
\label{Fibrations}
\end{figure}

\subsubsection{Criterion for Gieseker stability}
Lastly, in this subsection, we show a criterion when a rank $2$ sheaf $E$ is Gieseker semistable in terms of the corresponding curve.
In \cite[Proposition 3.1]{hartshorne1978stable}, the author gave a criterion for a rank $2$ vector bundle $E$ on $\p3$ to be $\mu$-(semi)stable in terms of the corresponding curve. 
Let $E$ be a rank $2$ vector bundle on $\p3$ and $c_1:=c_1(E)$. Assume that there is a one-dimensional scheme $Y\subset \p3$ satisfying
\[
0\lra \mathcal{O}_{\p3} \lra E \lra \mathcal{I}_{Y}(c_1) \lra 0.
\]

\begin{proposition}[\cite{hartshorne1978stable} Proposition 3.1]
Let $E$ be a rank 2 locally free sheaf. $E$ is $\mu$-stable (semistable) if and only if 

\begin{enumerate}
    \item $c_1(E)>0$ ($c_1(E)\geq 0$), and 
    \item $Y$ is not contained in any surface of degree $\leq \frac{1}{2}(c_1(E))$ ($<\frac{1}{2}(c_1(E))$).
    \end{enumerate}

\end{proposition}

On a smooth polarized threefold $(X, \calo_X(1))$ with Picard rank $1$, assume that for some $k\in \mathbb{Z}_{\geq 1}$, a coherent pair $(E(k), s)$ of class $v(k)$ corresponds to a one-dimensional scheme $Y\subset X$.
Then, we have the following result.

\begin{proposition}{\label{criterion:stab}}
$E$ is Gieseker semistable if $Y$ does not fit diagram (\ref{descriptionY}) for all possible one-dimensional schemes $A, C$ and zero-dimensional scheme $B_0$ such that
$B_0, C\subset S$ and $S\subset X$ is a hyper-surface of degree $k-l+c_1$.
\end{proposition}

\begin{proof}
Since the pair $(E(k), s)$ corresponds to a one-dimensional scheme $Y$, it must lie in the moduli space $\cals_{W_{\emptyset}^-}$. 
$Y$ does not fit into diagram (\ref{descriptionY}), so $(E(k), s)$ does not hit any wall as $\delta$ varies towards $0$. 
Therefore, $(E(k), s)\in \cals_{0^+}$, which implies that $E(k)$ is Gieseker semistable, and hence $E$ is Gieseker semistable. 
\end{proof}

\begin{remark}
The above criterion (Proposition \ref{criterion:stab}) may seem complicated to verify in general. 
In practice, we always try to choose the smallest $k\in \mathbb{Z}_{>0}$ such that $H^0(X, E(k))\neq 0$. This significantly limits the possible choices for $A$, $C$, and $B_0$, making the criterion more manageable. 
\end{remark}

\begin{remark}
The converse of Proposition \ref{criterion:stab} is not true in general because a one-dimensional scheme satisfying diagram (\ref{descriptionY}) does not always recover a wall (see Remark \ref{curvetopair}).
\end{remark}

\subsection{Bad pairs for wall-crossings}
As we mentioned at the end of Section \ref{subsec:critical}, there are possibly some “bad pairs" that are only “temporarily stable" as $\delta$ varies. We describe those kinds of pairs here.

\begin{definition}
    We say that a pair $(E(k),s)$ is \textit{bad} if it is only stable for $\delta\in (W_1, W_2)$, where $0<W_1<W_2<W_{\emptyset}$. 
\end{definition}

\begin{remark}
It follows from Theorem \ref{Gie-Hilb} that a bad pair $(E(k),s)$ is unsaturated, meaning that it does not correspond to a one-dimensional scheme.
Also, the underlying sheaf $E(k)$ in a bad pair is Gieseker unstable.
\end{remark}

The next two results indicate that a bad pair is created and eliminated at two walls in the same group, and the set of bad pairs is relatively small.

\begin{lemma}
    If a bad pair $(E(k),s)$ is $\delta$-stable for $\delta\in(W_1, W_2)$ for some $W_1$, $W_2$ satisfying $0<W_1<W_2<W_{\emptyset}$, then the walls $W_1$ and $W_2$ lie in the same group in the sense of Definition \ref{wall group}.
\end{lemma}

\begin{proof}
Without loss of generality, we assume that $c_1(v)=0$. The proof for $c_1=-1$ is similar.

Suppose $(E(k),s)\in \cals_{W}^+$ is a bad pair created at a wall $W$ defined by the sequence
\[
0\lra (\cali_A(2k-e), 0)\lra (E(k),s)\lra (\cali_B(e), 1)\lra 0,
\]
in which $A,B\subset \p3$ are one-dimensional schemes and $1\leq e\leq k$. 
Then, the cokernel $Q$ of $s$ is not torsion-free. Let $T\hookrightarrow Q$ be the maximum torsion sub-sheaf of $Q$. We have the following commutative diagram
\begin{center}
\begin{tikzcd}
    & \op3 \arrow[r]\arrow[d, "s"]&\op3\arrow[d, "S"]\\
    \cali_A(2k-e)\arrow[r]\arrow[d]&E(k)\arrow[r]\arrow[d]&\cali_B(e)\arrow[d]\\
    \cali_A(2k-e)\arrow[r]&Q\arrow[r]&\cali_{B|S}(e),\\
\end{tikzcd}
\end{center}
in which $S\subset \p3$ is a hyper-surface of degree $e$. The bottom row of the diagram  implies the diagram below ($\phi$ is injective since $Q/T$ is torsion-free of rank $1$)
\begin{center}
    \begin{tikzcd}
    & T \arrow[r]\arrow[d]&T\simeq \cali_{B'|S}(e)\arrow[d, "\phi"]\\
    \cali_A(2k-e)\arrow[r]\arrow[d]&Q\arrow[r]\arrow[d]&\cali_{B|S}(e)\arrow[d]\\
    \cali_A(2k-e)\arrow[r, "\phi"]&Q/T\arrow[r]&F,\\       \end{tikzcd}
\end{center}
where $T$ must be isomorphic to $\cali_{B'|S}(e)$ for some scheme $B'$ such that $B\subset B'\subset S$. The quotient sheaf $F\simeq\cali_{B|S}(e)/T$ is supported on the one-dimensional scheme $B'$.  
Hence, the torsion-free sheaf $Q/T$ has $c_1(Q/T)=c_1(\cali_A(2k-e))=2k-e$, and $Q/T\simeq \cali_{C}(2k-e)$ for some one-dimensional scheme $C\subset \p3$. Consider the composition $E\onto Q\onto Q/T$ and let $K$ be its kernel. 
$K$ is torsion-free, and one computes directly that $\ch(K)=\ch(\calo_{\p3})+\ch(\cali_{B'}(e))$, so $K\simeq \cali_{B'}(e)$.
Diagram (\ref{diag:saturation}) implies that 
$(E(k),s)$ is destabilized when crossing the following wall (call it $W'$)
\[
W': \quad 0\lra (\cali_{B'}(e), s')\lra (E(k),s)\lra (\cali_C(2k-e), 0)\lra 0.
\]

Therefore, the bad pair $(E(k),s)$ is stable for $\delta$ between $W$ and $W'$, and both of these two walls are in group $e$. 
\end{proof}

\begin{proposition}{\label{bad pair}}
    Any bad pair $(E(k),s)$ lies in the boundary of $\cals^+_{W}$ for some wall $W\in (0, W_{\emptyset})$, hence the set of bad pairs is never a component in the moduli space ${\cals}^{\delta}_X(P)$ for any $\delta\in (0, W_{\emptyset})$ that is not a critical value.
\end{proposition}

\begin{proof}
    Assume that a bad pair $(E(k),s)\in \cals^+_{W}$ is created at a wall $W$ defined by the following short exact sequence of pairs 
    $$0\lra (\cali_A(l), 0)\lra (E,s)\lra (\cali_B(k),s') \lra 0,$$
    in which $A, B\subset X$ are one-dimensional schemes, and $k,l\in \mathbb{Z}_+$. Let $Q$ be the cokernel of the section $s$ for $E(k)$, and we have the following diagram
\begin{center}
\begin{tikzcd}
        &\calo_X\arrow[r, "="]\arrow[d, "s"]&\calo_X\arrow[d, "s'"]\\
\cali_A(l)\arrow[r]\arrow[d]&E\arrow[r]\arrow[d]&\cali_B(k)\arrow[d]\\
\cali_A(l)\arrow[r]&Q\arrow[r]&\cali_{B|S}(k).\\
\end{tikzcd}
\end{center}
In the diagram, $S\subset X$ is a surface of degree $k$. Since $(E(k),s)$ is unsaturated, $Q$ contains a maximum torsion sub-sheaf $T$, and $T$ must be a sub-sheaf of $\cali_{B|S}(k)$. Therefore, we have the following diagram
\begin{center}
    \begin{tikzcd}
    &T\arrow[r]\arrow[d]&T\arrow[d]\\
\cali_A(l)\arrow[r]\arrow[d]&Q\arrow[r]\arrow[d]&\cali_{B|S}(k)\arrow[d]\\
\cali_A(l)\arrow[r]&Q/T\arrow[r]&Q_1.\\
    \end{tikzcd}
\end{center}
Here, $Q_1$ is a one-dimensional scheme whose support is contained in $S$. Note that $Q/T$ is a torsion-free sheaf of rank $1$, hence $Q/T\simeq \cali_Y(m)$ for some one-dimensional scheme $Y\subset X$ and $m\in \mathbb{Z}$. We must have $m=l$ since $c_1(Q_1)=0$. Hence, we have the short exact sequence 
$$0\lra Q_1\lra \calo_A(l) \lra \calo_Y(l) \lra 0.$$
This implies that the support of $Q_1$ is also contained in $A$.

Therefore, if we deform $A$ into $A'$ such that no component of $A'$ is contained in $S$, then $Q$ will be deformed to $Q'$, which is torsion-free of rank 1. Consequently, the bad pair $(E(k),s)$ will be deformed into a saturated pair $(E'(k), s')$, and this proves the claim.
\end{proof}

\begin{corollary}{\label{rat'l map 2}}
    For any $\delta\in (0,W_{\emptyset})$ that is not a critical value and for every irreducible component $\mathcal{P}$ of $\mathcal{S}^{\delta}_X(P)$, there exists a rational map from $\mathcal{P}$ to either the Hilbert scheme of one-dimensional schemes or the moduli space of Gieseker semistable sheaves. 
\end{corollary}


\section{Wall-crossings at critical values defined by 0-dimensional schemes} \label{sec:critical}

In this section, we work on $\p3$ with a normalized rank $2$ class $v$. 
We study a distinguished group of wall-crossings
\[
W_{l_{v(k)}}<\cdots <W_1<W_0,
\]
defined by one-dimensional schemes, which occur before the collapsing wall. These walls are associated with a twisted class $v(k)$ for an appropriate $k\in \mathbb{Z}_+$.

The goal of this section is as follows:
\begin{enumerate}
    \item Relate the wall-crossings at the distinguished walls $W_i$ to a stratification of the Hilbert scheme of curves. 
    \item Describe the birational transformations of a moduli space component when crossing these walls, and prove that those wall-crossings preserve the connectedness of the moduli space.
\end{enumerate}
Recall that for $\delta \in (W_T, W_{\emptyset})$, there exists a morphism from the moduli space of semistable pairs to the Hilbert scheme. 
We will start from this chamber and then cross walls successively as $\delta$ decreases towards $0$. 

\subsection{A sequence of walls defined by 0-dimensional schemes} \label{subsec:wall0dim}

We compute with the Chern classes $c_j$ ($j=0,1,2,3$) of $v$.
We have shown in Section \ref{subsec:wallxph} that the second largest wall for $v(k)$ is (compare with sequence  (\ref{secondwall}))
\begin{equation}{\label{tor-free wall}}
W_0: \quad 0\lra (\cali_A(1), 1)\lra (E(k),s)\lra (\calo_{\p3}(2k-1+c_1), 0)\lra 0.
\end{equation}
The subsequent walls are defined inductively for $i=1,2,..., l_{v(k)}$ as follows 
\begin{equation}{\label{0-dim wall}}
W_i: \quad 0\lra (\cali_{A_i}(1), 1)\lra (E(k),s)\lra (\cali_{P_i}(2k-1+c_1), 0)\lra 0,
\end{equation}
where each $A_i\subset \p3$ is a planar one-dimensional scheme, and $P_i\subset \p3$ is zero-dimensional with $|P_i|=i$.
The integer $l_{v(k)}$, which depends on the class $v(k)$, counts the number of these walls; it will be computed in Remark \ref{0-dim walls}. 
Note that $A_{l_{v(k)}}$ is pure one-dimensional, and they are strictly ordered $W_i<W_{i-1}$ for $i=1,2,3,...,l_{v(k)}$.
These walls are in group $1$ (Definition \ref{Groups of walls}) as shown in the picture below.

\begin{figure}[ht]
\begin{tikzpicture}
\draw[thick] (-0.5,0)--(11.9,0);
\draw (11,-0.2) -- (11, 0.2);
\node at (11, -0.5){$W_{\emptyset}$};
\draw (10,-0.2) -- (10, 0.2);
\node at (10, -0.5){$W_{0}$};
\draw (9,-0.2) -- (9, 0.2);
\node at (9, -0.5){$W_{1}$};
\draw (8,-0.2) -- (8, 0.2);
\node at (8, -0.5){$W_{2}$};
\node at (6, -0.5){$\cdots$};
\draw (4,-0.2) -- (4, 0.2);
\node at (4, -0.5){$W_{l_{v(k)}}$};
\node at (0.5, 0.5){$\mathbb{Q}[t]-$axis for $\delta$};
\node at (0, -0.5){$0$};
\draw (0,-0.2) -- (0, 0.2);
\draw [decorate,decoration={brace,amplitude=5pt,mirror,raise=5ex}]
  (3.5,0) -- (10.2,0) node[midway,yshift=-4em]{walls in group $1$};
\end{tikzpicture}
\end{figure}

\begin{remark}{\label{0-dim walls}}
The number $l_{v(k)}$ equals the length of the zero-dimensional subscheme contained in the planar curve $A_0$ from (\ref{0-dim wall}). 
Indeed, $A_{l_{v(k)}}$ is pure one-dimensional, and as the index decreases from $A_{l_{v(k)}}$ to $0$, each step adds one point to the curve until $A_0$ attains the maximal possible number of points.

A direct computation shows that:
\[
\begin{array}{rl}
P_{\calo_{A_0}}(t)=&P_{\calo_{\p3}}(t)+P_{\calo_{\p3}(2k-2+c_1)}(t)-P_{E(k-1)}(t)  \\
     =& (k^2+(c_1-2)k+3c_1-c_2+\frac{25}{3})t\\
     &+k^3+(\frac{3}{2}c_1-1)k^2+(\frac{19}{3}+3c_1-c_2)k\\
     &+\frac{1}{2}c_1c_2-\frac{1}{2}c_3-\frac{9}{2}c_1+3c_2-\frac{1}{2}c_1^2-\frac{19}{3}.
\end{array}
\]
Write $P_{\calo_{A_0}}(t)=d_{A_0}t+\chi_{A_0}$. Because ${A_0}$ is planar, the length of the maximum zero-dimensional subscheme in ${A_0}$ is 
$$\chi_A-1+\binom{d_A-1}{2},$$ 
which depends only on $k$ and the Chern class of $v$.
\end{remark}

The next Proposition shows that under a mild degree condition, the wall $W_0$ appears defined by a planar one-dimensional scheme $A\subset \p3$. 



\begin{proposition}{\label{1stwall}}
Suppose there exists a saturated pair $(E(k), s)$ corresponding to a one-dimensional scheme $Y$ (Definition \ref{defn:sat}). 
Let $d_Y$ be the degree of the maximum one-dimensional subscheme of $Y\subset \p3$. 
If 
$$d_Y-2k+1-c_1>0,$$
then there exists a one-dimensional planar scheme $A\subset \p3$ such that the flipped sequence
$$ 0\lra (\mathcal{O}_{\mathbb{P}^3}(2k-1+c_1), 0)\lra (E(k),s)\lra (\mathcal{I}_A(1), 1)\lra 0 $$
admits non-trivial extensions. 
Moreover, $Y$ is a pure one-dimensional planar curve if and only if $A$ is.
\end{proposition}

\begin{proof}
Suppose such an $A$ exists, then we have the following diagram
\[
\begin{tikzcd}
{}&{}&\mathcal{O}_{\mathbb{P}^3}\arrow{r}\arrow{d}&\mathcal{O}_{\mathbb{P}^3}\arrow{d}&\\
0\arrow{r}&\mathcal{O}_{\mathbb{P}^3}(2k-1+c_1)\arrow{r}&E(k)\arrow{r}\arrow{d}&\mathcal{I}_A(1)\arrow{r}&0\\
&&\mathcal{I}_Y(2k+c_1).&\\
\end{tikzcd}
\]
Let $P_{\mathcal{O}_Y}(t)=d_Yt+\chi_Y$ be the Hilbert polynomial of $Y$.
The Hilbert polynomial of $\mathcal{O}_A$ is as follows
\[
\begin{array}{rl}
P_{\mathcal{O}_A}(t)=&P_{\mathcal{O}_Y(2k-1+c_1)}(t)+P_{\mathcal{O}_{\mathbb{P}^3}}(t)-P_{\mathcal{O}_{\mathbb{P}^3}(-1)}(t)\\
&+P_{\mathcal{O}_{\mathbb{P}^3}(2k-2+c_1)}(t)-P_{\mathcal{O}_{\mathbb{P}^3}(2k-c_1)}(t)\\
=&(d_Y-2k+1-c_1)t+d_Y(2k-1+c_1)+\chi_Y\\
&-\frac{1}{2}(2k+1+c_1)(2k+c_1)+1.
\end{array}
\]

Define $d_A:=d_Y-2k+1-c_1$ and
$$ \chi_A:=d_Y(2k-1+c_1)+\chi_Y-\frac{1}{2}(2k+1+c_1)(2k+c_1)+1, $$
so that $P_{\mathcal{O}_A}(t)=d_At+\chi_A$.
We must have $d_A\geq 1$ since otherwise $A$ will be zero-dimensional, and 
$$\Ext^1(\cali_{A}(1), \calo_{\p3}(2k-1+c_1))\simeq \Ext^2(\calo_{\p3}(2k-1+c_1), \cali_{A}(1-4))^*=0,$$ 
leading to a contradiction.

Conversely, suppose $d_A:=d_Y-2k+1-c_1\geq 1$, we first show that there exists such a curve $A$ whose degree and Euler characteristic are $d_A$ and $\chi_A$ respectively. We prove the inequality $\chi_A\geq 1-\binom{d_A-1}{2}$ since the right-hand side is the minimum for the Euler characteristic of a degree $d_A$ curve. 

By definition, $\chi_A\geq 1-\binom{d_A-1}{2}$ can be rewritten as 
$$ d_Y(2k-1+c_1)+\chi_Y-\frac{1}{2}(2k+1+c_1)(2k+c_1)+1 \geq 1- \binom{d_Y-2k+1-c_1-1}{2}. $$
The right-hand side is expanded as 
$$1-\frac{1}{2}(2k+1+c_1)(2k+c_1) + d_Y(2k-1+c_1)-\frac{1}{2}d_Y^2+\frac{3}{2}d_Y.$$
Compare it with the left-hand side, and the inequality is reduced to  
$$ \chi_Y\geq -\frac{1}{2}d_Y^2+\frac{3}{2}d_Y=1-\binom{d_Y-1}{2}, $$
which clearly holds.
In particular, $\chi_Y=1-\binom{d_Y-1}{2}$ if and only if \linebreak $\chi_A=1-\binom{d_A-1}{2}$.
This implies that $A$ is a pure one-dimensional planar scheme if and only if $Y$ is.

Next, we show that $\Ext^1(\cali_{A}(1), \calo_{\p3}(l))>0$ for some one-dimensional scheme $A$ with the given $d_A$ and $\chi_A$. 
This will imply the existence of a non-trivial extension $E(k)$, and hence $(E(k),s)$ since $H^1(\mathbb{P}^3, \mathcal{O}_{\p3}(2k-1+c_1))=0$.
Moreover, the cokernel of $s$ in $(E(k), s)$ must be $\cali_{Y}(2k+c_1)$ since $\delta\in (W_0, W_{\emptyset})$. 

A simple computation shows that

\begin{align*}
    \Ext^1(\cali_{A}(1), \calo_{\p3}(l)) &\simeq \Ext^2(\calo_{\p3}(l)), \cali_{A}(1))^* = H^2(\p3, \cali_A(-3-l))^* \\
    &\simeq H^1(\p3, \calo_A(-3-l))^* = H^1(A, \calo_A(-3-l))^*. 
\end{align*}
The latter cohomology is independent of the zero-dimensional subschemes of $A$, so we may assume that $A\subset H$ is pure of degree $d_A$. 
Take the long exact sequence of cohomologies of the short exact sequence
\[
0\lra \calo_H(-d_A-3-l) \lra \calo_H(-3-l)\lra \calo_A(-3-l) \lra 0, 
\]
and we obtain the short exact sequence
$$ 0 \to H^1(H, \calo_A(-3-l)) \to H^2(H, \calo_H(-d_A-3-l)) \to H^2(H, \calo_H(-3-l)) \to 0 $$
since $ H^1(H, \calo_H(-3-l))=H^2(H, \calo_A(-3-l))=0$.

Now, we see that $h^1(H, \calo_A(-3-l))>0$ because $d_A\geq 1$ and 
$$ h^2(H, \calo_H(-d_A-3-l)) > h^2(H, \calo_H(-3-l)). $$
\end{proof}


\begin{remark}
The wall $W_0$ is the same as the wall $W_T$ defined in Section \ref{sec:pairs} if it exists. 
\end{remark}


\subsection{Stratification in $\calh_{\p3}(P)$ for zero-dimensional walls} 

We continue studying the curves created by $W_i$. 
Diagram (\ref{descriptionY}) implies that a stable pair $(E(k),s)\in \cals_{W_i}^+$ corresponds to a one-dimensional scheme $Y_i$ that fits into the sequence 
\[
0\lra \calo_{P_i}\lra \calo_{{Y_i}}\lra \calo_{C_i}\lra 0.
\]
In the above sequence, $P_i$ is a zero-dimensional scheme of length $i$; $C_i\subset H$ is a planar one-dimensional scheme, and it is the union of a zero-dimensional scheme $Q_i\subset H$ and a pure one-dimensional scheme $C_i'\subset H$ of degree $2k+c_1-1+\deg(A_i)$.  
For any $i<j$, the difference between $Y_i$ and $Y_j$ is the position of some zero-dimensional subschemes.

Let $P$ be the Hilbert polynomial of $\calo_{Y_i}$, and denote by $\calh_{pl}$ the components in $\mathcal{H}_{\p3}(P)$ parameterizing such one-dimensional schemes $Y_i$. To be more precise, we have the following definition.

\begin{definition}
Let $\calh_{pl}\subset \calh_{\p3}(P)$ be the component(s) that parameterize one-dimensional schemes $C\subset \p3$ whose maximum one-dimensional subscheme is planar.
\end{definition}

In this paper, we only consider the case that $\calh_{pl}$ is irreducible, i.e., any one-dimensional scheme $Y\in \calh_{pl}$ can be deformed to the union of a pure one-dimensional scheme and isolated points.
In this case, the subscheme of $\mathcal{S}_{W^-_{\emptyset}}$ lying over $\mathcal{H}_{pl}$, denoted by $\tilde{\mathcal{H}}_{pl}$, is irreducible.
We provide a proof below. 

\begin{lemma}{\label{irreducibility}}
    Let $\mathcal{H}_{pl}\subset \mathcal{H}_{\p3}(\tilde{P})$ be the subscheme parameterizing curves $Y$, whose maximum one-dimensional subscheme is planar. 
    Assume $\mathcal{H}_{pl}$ is irreducible, then the subscheme in $\mathcal{S}_{W^-_{\emptyset}}$ lying over $\mathcal{H}_{pl}$ is irreducible, and denote this subscheme by $\tilde{\mathcal{H}}_{pl}$.  
\end{lemma}
\begin{proof}
    We first prove that $ext^1(\mathcal{I}_Y(2k+c_1), \mathcal{O}_{\p3})$ is a constant for every $Y\in \mathcal{H}_{pl}$. 
    Serre duality implies that $\Ext^1(\mathcal{I}_Y(2k+c_1), \mathcal{O}_{\p3})\cong \Ext^2(\mathcal{O}_{\p3}, \mathcal{I}_Y(2k+c_1-4))^*$; the latter coincides with $H^2(\p3, \mathcal{I}_Y(2k+c_1-4))^*\cong H^1(\p3, \mathcal{O}_Y(2k+c_1-4))^*.$ 
    Let $d$ be the degree of the one-dimensional subscheme of $Y$. Notice that the term $H^1(\p3, \mathcal{O}_Y(2k+c_1-4))^*$ is independent of the zero-dimensional subscheme of $Y$, and we have the following short exact sequence
    \begin{align*}
    0\to H^1(\p3, \mathcal{O}_Y(2k+c_1-4))&\to H^2(\p2, \mathcal{O}_{\p2}(2k+c_1-d-4))\\
    &\to H^2(\p2, \mathcal{O}_{\p2}(2k+c_1-4))\to 0.
    \end{align*}
    One sees that $h^1(\p3, \mathcal{O}_Y(2k+c_1-4))$ only depends on the fixed numbers $k, c_1$ and $d$, hence a constant for $Y\in \mathcal{H}_{pl}$.

    Next, let 
    $\tilde{\mathcal{I}}\subset \p3\times \mathcal{H}_{pl}$ be the universal ideal sheaf, and $p: \p3\times \mathcal{H}_{pl}\to \p3$, $q: \p3\times \mathcal{H}_{pl}\to \mathcal{H}_{pl}$ are the two projections. 
    The sheaf 
    $$ \mathcal{E}xt^1_{q}(\tilde{\mathcal{I}}(p^*(\mathcal{O}_{\p3}(2k+c_1))), \mathcal{O}_{\p3\times \mathcal{H}_{pl}}) $$
    is locally free on $\mathcal{H}_{pl}$ since $R\mathcal{H}om(\tilde{\mathcal{I}}(p^*(\mathcal{O}_{\p3}(2k+c_1))), \mathcal{O}_{\p3\times \mathcal{H}_{pl}})$ is a perfect complex over $\mathcal{H}_{pl}$, and $ext^1(\mathcal{I}_Y(2k+c_1), \mathcal{O}_{\p3})$ is a constant. 
    We have that $\mathbb{P}(\mathcal{E}xt^1_{q}(\tilde{\mathcal{I}}(p^*(\mathcal{O}_{\p3}(2k+c_1))), \mathcal{O}_{\p3\times \mathcal{H}_{pl}}))\subset \mathcal{H}_{\p3}(\tilde{P})$ is the subscheme lying over $\mathcal{H}_{pl}$, and it is irreducible. 
\end{proof}

We introduce a stratification $\{\calz_i\}_{i=0, ..., l_{v(k)}}$ of $\calh_{pl}$ in the next definition.
These are the (locally) closed subschemes in $\calh_{pl}$ perturbed when crossing $W_i$.
\begin{definition}{\label{locusZi}}
Let $i$ be an integer that takes values in $\{0, 1,..., l_{v(k)}\}$.
\begin{enumerate}
    \item Define $\mathcal{X}_i\subset \calh_{pl}$ to be the subset parameterizing one-dimensional schemes $C_i=C'_i\cup P_i$, in which $C_i'\subset H$ is planar contained in a plane $H\subset \p3$ and $P_i$ is zero-dimensional with $|P_i|=i$ and $P_i\cap H=\emptyset$.

    \item Define $\mathcal{Z}_0:=\mathcal{X}_0$, and inductively, $\mathcal{Z}_i:=\bar{X}_i\backslash \bar{X}_{i-1}$.

    \item Define $V_i\subset \calh_{\p3}(P)$ to be the complement of $\mathcal{X}_i$ in $\mathcal{Z}_i$. 
\end{enumerate} 
\end{definition}

\begin{remark}
$\mathcal{X}_i\subset \calh_{pl}$ parameterizes curves that are the union of planar curves and $i$ points out of the plane.
$\calz_i\subset \calh_{pl}$ parameterizes curves that are the union of a planar curve and $i$ points not in the same plane, including the case of $\leq i$ embedded points pointing out of the plane. 
In addition, $\calz_0\subset \calh_{pl}$ is closed and $\calz_i\subset \calh_{pl}$ is locally closed for $i>0$.
\end{remark}

\begin{definition}{\label{locustilde}}
Define $\tilde{\calh}_{pl}\subset \cals_{W_{\emptyset}^-}$ to be the subschemes lying over $\calh_{pl}$. Lemma \ref{irreducibility} implies that it is irreducible. 
Similarly, define $\tilde{\mathcal{Z}}_i$, $\tilde{\mathcal{X}_i}$ and $\tilde{V_i}$ to be the subschemes of $\tilde{\calh}_{pl}$ lying over  $\mathcal{Z}_i$, $\mathcal{X}_i$ and $V_i$ respectively
for $i=0,1,2,...,l_{v(k)}$. 
\end{definition}

Clearly, we have that $\dim(\mathcal{Z}_{i+1})=\dim(\mathcal{Z}_{i})+1$, $\dim(\mathcal{Z}_{l_{v(k)}})=\dim(\calh_{pl})$, and $\mathcal{Z}_i$ lies in the boundary of $\mathcal{Z}_{i+1}$.


\begin{lemma}{\label{surjtocurve}}
    The (rational) map $\cals^+_{W_i} \dashrightarrow  \mathcal{Z}_i$ that sends a saturated pair $(E(k),s)$ to its corresponding one-dimensional scheme $Y$ is surjective. 
\end{lemma}

\begin{proof} 
Firstly, the map is rational since there may exist bad pairs in $\cals^+_{W_i}$.
Next, we prove that every one-dimensional scheme $Y\in \mathcal{Z}_i$ recovers a pair $(E(k),s)\in \cals^+_{W_i}$.

By diagram (\ref{diagrampair}), $Y\in \mathcal{Z}_i$ satisfies the following sequence
$$0\lra \cali_{P_i}(2k-1+c_1(E))\lra \cali_Y(2k+c_1)\lra \cali_{A_i|H}(1)  \lra 0.$$
Consider a non-trivial extension $E(k)$ in $0\to \calo_{\p3}\to E(k)\to \cali_Y(2k+c_1)\to 0$ for some $k>0$. These two sequences imply the following diagram
\begin{center}    
    \begin{tikzcd}
        &\calo_{\p3}\arrow[d]&\\
        &E(k)\arrow[d]&\\
        \cali_{P_i}(2k-1+c_1(E))\arrow[r, "i"]& \cali_Y(2k+c_1)\arrow[r]& \cali_{A_i|H}(1).\\
    \end{tikzcd}
\end{center}
Observe that the inclusion $i$ lifts to an inclusion $i': \cali_{P_i}(2k-1+c_1(E))\to E(k)$ because of the vanishing $\Ext^1(\cali_{P_1}(2k-1+c_1(E)), \calo_{\p3})=0$ and $E$ being torsion-free.
We complete the diagram as follows 
\begin{center}
    \begin{tikzcd}
&\calo_{\p3}\arrow[d]\arrow[r,"="]&\calo_{\p3}\arrow[d]\\
        \cali_{P_i}(2k-1+c_1(E))\arrow[r, "i'"]\arrow[d, "="]&E(k)\arrow[r]\arrow[d]&Q\arrow[d]\\
        \cali_{P_i}(2k-1+c_1(E))\arrow[r, "i"]& \cali_Y(2k+c_1)\arrow[r]& \cali_{A_i|H}(1).\\
    \end{tikzcd}
\end{center}

Lastly, we prove that $Q\simeq \cali_{A_i}(1)$.
It is easy to see that $Q\simeq \cali_{A_i}(1)$ is an extension of $\cali_{A_i|H}(1)$ by $\calo_{\p3}$. 
Moreover, 
\begin{align*}
\Ext^1(\cali_{A_1|H}(1), \calo_{\p3}) &\simeq \Ext^2(\calo_{\p3}, \cali_{A_i|H}(-3))^* =  H^2(H, \cali_{A_i|H}(-3))^* \\
&\simeq H^2(\cali_{A'_i|H}(-3-d_A))^* \simeq H^0(\calo_H(d_A))
\end{align*}
Let $A'\subset A_i$ and $A''\subset A_i$ denote the zero-dimensional and one-dimensional components of $A_i$, and $d_A$ be the degree of $A''$.
The computation shows that extensions $\{Q\}$ are one-to-one correspondence to $H^0(H, \calo_H(d_A))$. 
Therefore, $Q\simeq \cali_{A_i}(1)$, in which the one-dimensional part $A''$ corresponds to its class in $H^0(H, \calo_H(d_A))$, are all the extensions. This proves the claim.
\end{proof}

\begin{remark}{\label{dim Zi tilde}}
We provided in Section \ref{subsec:fibrations} two fibrations for $\cals^+_W$ (Figure \ref{Fibrations}). 
For the case $W=W_i$, the subscheme $\calh_W\subset \calh_{\p3}(P)$ is $\mathcal{Z}_i$, and the corresponding fibrations are shown in Figure \ref{SWi} below.
A simple computation shows that 
\begin{align*}
\Ext^1(\mathcal{I}_{Y_i}(2k+c_1), \mathcal{O}_{\mathbb{P}^3}) &\simeq \Ext^2(\mathcal{O}_{\mathbb{P}^3}, \mathcal{I}_{Y_i}(2k+c_1-4))^* \\
&\simeq H^2(\mathbb{P}^3, \mathcal{I}_{Y_i}(2k+c_1-4))^*\simeq H^1(\mathbb{P}^3, \mathcal{O}_{Y_i}(2k+c_1-4))^*
\end{align*}
which is independent of the $0-$dimensional subscheme $Q_i\subset Y_i$. 
Besides, Proposition \ref{bad pair} implies $\dim(\cals^+_{W_i})=\dim (\tilde{\mathcal{Z}}_i)$ since the set of bad pairs in $\cals^+_{W_i}$ is small.
Therefore, $\dim(\mathbb{P}\Ext^1(\mathcal{I}_{Y_i}(2k+c_1), \mathcal{O}_{\mathbb{P}^3}))$, the dimension of the fiber over $\mathcal{Z}_i$, is independent of the value of $i$ ($i=0, 1,..., l_{v(k)}$). 
In addition, $\tilde{\mathcal{Z}}_i$ lies in the boundary of $\tilde{\mathcal{Z}}_{i+1}$, and we have
$\dim(\tilde{\mathcal{Z}}_i)+1=\dim(\tilde{\mathcal{Z}}_{i+1})$ and $\dim(\tilde{\mathcal{Z}}_{l_{v(k)}})=\dim(\tilde{\calh}_{pl})$,
which are analogous to the properties for $\{{Z}_i\}$ ($i=0,1,...,l_{v(k)}$).
\end{remark}

\begin{figure}[ht] 
\begin{tikzcd}
\mathbb{P}\Ext^1((\cali_{A_i}(1), 1), (\cali_{P_i}(b), 0))\arrow[dr, "fiber"]&&\mathbb{P}\Ext^1(\mathcal{I}_{Y_i}(b+1), \mathcal{O}_{\mathbb{P}^3})\arrow[dl, "fiber"]\\
&\cals^+_{W_i}\arrow[dr]\arrow[dl, dashed]&\\
\mathcal{Z}_i&&\cals^{ss}_{W_i}\\
\end{tikzcd}\\
In the diagram, $b:=2k-1+c_1$.
\caption{Fibration of $\cals^+_{W_i}$}
\label{SWi}
\end{figure}

\begin{remark}
According to diagram (\ref{descriptionY}), a stable pair $(E(k),s)\in \cals^+_{W_i}$ corresponds to a one-dimensional scheme in $\bar{X}_i=\bar{\calz}_i=\mathcal{Z}_i\cup \mathcal{Z}_{i-1}\cup \cdots \cup \mathcal{Z}_0$. 
However, stable pairs in $\cals^{+}_{W_i}$ only correspond to one-dimensional schemes in $\mathcal{Z}_i$ (locally closed in $\tilde{\calh}_{pl}$) for the reason of stability.
To see this, when crossings walls as $\delta$ gets smaller, stable pairs $(E(k),s)$ corresponding to one-dimensional schemes in $\mathcal{Z}_j$ ($j<i$) are removed when crossing $W_j$ before hitting $W_i$, hence they are not contained in $\cals^+_{W_i}$ any more.
\end{remark}


Next, we give a simple criterion to check if there are bad pairs in $\cals^+_{W_i}$. 
Recall that $W_i$ is defined by the following sequence of pairs (compare with sequence  (\ref{0-dim wall}))
$$0\lra (\cali_{P_i}(2k-1+c_1), 0)\lra (E(k), s) \lra (\cali_{A_i}(1) ,1) \lra 0.$$

Let $A'$ (resp. $A''$)$\subset A_i$ be the maximum zero-dimensional (resp. one-dimensional) subscheme of $A_i$, and let $d$ be the degree of $A''$. 
For a 0-dimensional scheme $P$, define $P_H\subset P$ to be the subscheme supported on $H$, i.e. $\calo_{P_H}$ is the kernel of the morphism induced by the multiplication by $H$,
$$0\lra \calo_{P_H}\lra \calo_P \overset{\cdot H}{\lra} \calo_P. $$

\begin{lemma}{\label{criterion:sat}}
For $P_i$ and $A_i$ such that $P_i\cap A'=\emptyset$, the corresponding extensions $(E(k),s)\in \cals^+_{W_i}$ are saturated if and only if $H^1(\p3, \cali_{P_{i H}|H}(k+d-1))=0$. In particular, all stable pairs in $\cals^+_{W_0}$ and $\cals^+_{W_{l_{v(k)}}}$ are saturated. 
\end{lemma}
\begin{proof}
If a pair $(E(k),s)$ is saturated, then the cokernel of $s$ is the twisted ideal sheaf $\cali_Y(2k+c_1)$, satisfying the sequence 
$$ 0\lra \cali_{P_i}(2k-1+c_1)\lra \cali_{Y_i}(2k+c_1)\lra \cali_{A_i|H}(1) \lra 0. $$
We know that such a sequence recovers the wall $W_i$ from Lemma \ref{surjtocurve}. Hence, it suffices to prove that the vector space $\Ext^1(\cali_{A_i|H}(1), \cali_{P_i}(2k-1+c_1))$ corresponds to the space of all possible one-dimensional schemes $Y_i$ fitting into $W_i$.

Apply the functor $\Hom(-,\cali_{P_i}(2k-1+c_1))$ to the sequence 
$$ 0\lra \cali_{A'|H}(1-d)\lra \calo_H(1-d)\lra \calo_{A'} \lra 0,$$
and part of the induced long exact sequence is as follows
$$ 0 \lra \Ext^1(\calo_{A'},\cali_{P_i}(2k-1+c_1)) \lra \Ext^1(\calo_H(1-d),\cali_{P_i}(2k-1+c_1)) \lra $$
$$ \lra \Ext^1(\cali_{A'|H}(1-d), \cali_{P_i}(2k-1+c_1)) \lra \Ext^2(\calo_{A'}, \cali_{P_i}(2k-1+c_1))\lra \dots $$

We have $\Ext^1(\calo_{A'},\cali_{P_i}(2k-1+c_1))=\Ext^2(\calo_{A'},\cali_{P_i}(2k-1+c_1))=0$ since $P_i\cap A'=\emptyset$. Then,
\begin{align*}
\Ext^1(\cali_{A'|H}(1-d), \cali_{P_i}(2k-1+c_1)) &= \Ext^1(\calo_H(1-d),\cali_{P_i}(2k-1+c_1)) \\
&= H^0(\p3, \mathcal{E}xt^1(\op3((1-d), \cali_{P_i}(2k-1+c_1)).
\end{align*}
We compute $\mathcal{E}xt^1(\op3((1-d), \cali_{P_i}(2k-1+c_1))$ using the following resolution for $\calo_H(1-d)$ 
$$ 0\lra \op3(-d)\lra \op3(1-d)\lra \calo_H(1-d) \lra 0,$$ 
and $\mathcal{E}xt^1(\op3((1-d), \cali_{P_i}(2k-1+c_1))$ satisfies the sequence
$$ 0\lra \calo_{P_H}\lra \mathcal{E}xt^1(\op3((1-d), \cali_{P_i}(2k-1+c_1))\lra \cali_{P_H}(d+2k-1+c_1)\lra 0.$$
Recall that $\calo_{P_H}$ is the kernel of the following multiplication by $H$
\[0\lra \calo_{P_H}\lra \calo_P\overset{\cdot H}{\lra}\calo_P.
\]
Therefore,
$$ h^0(\p3, \mathcal{E}xt^1(\calo_H(1-d), \cali_{P_i}(2k-1+c_1)) = |P_H|+h^0(\p3, \cali_{P_H}(d+2k-1+c_1)) = $$
$$ = |P_H|+h^0(\p3, \calo_H(d+2k-1+c_1))-|P_H|+h^1(\p3, \calo_H(d+2k-1+c_1)) = $$
$$ = h^0(\p3, \calo_H(d+2k-1+c_1))+h^1(\p3, \calo_H(d+2k-1+c_1)). $$


The first term parameterizes all possible $Y_i$ since $Y_i$ is the union of a degree $(2k+d-1+c_1)$ curve $P_i$ in $H$ and a fixed zero-dimensional scheme $A_0$. 
Therefore, $(E(k), s)$ is saturated if and only if $h^1(\p3, \calo_H(d+2k-1+c_1))=0$.
In particular, $P_0=\emptyset$ at $W_0$ and $A'=\emptyset$ at $W_{l_{v(k)}}$, which imply that all pairs in $\cals^+_{W_0}$ and $\cals^+_{W_{l_{v(k)}}}$ are saturated.
\end{proof}

\subsection{Birational transformations for $\tilde{\calh}_{pl}$ when crossing $W_i$'s}{\label{sec: birat'l trans}}

We work with a fixed class $v(k)$, but we will mute it in any related notations if there is no ambiguity.
We will prove the main theorem of this section, that the wall-crossings at $W_i$ induce the following series of transformations for an open subscheme of $\tilde{\calh}_{pl}$:
\begin{itemize}
    \item Flip for $i=0,1,..., l_{v(k)} $;
    \item Divisorial contraction for $i=l_{v(k)}-1$;
    \item Removal of an open subscheme in $\tilde{\calz}_{l_{v(k)}}$ for $i=l_{v(k)}$.
\end{itemize}

This is proved by computing the dimensions of the fibers of the morphisms 
$$\phi^{\pm}_{W_i}: \cals^{\pm}_{W_i} \to \cals^{ss}_{W_i}.$$ 

We clarify some notations in the next Definition for later use.   

\begin{definition}
\begin{enumerate}
\item
Define $\mathcal{R}_{W_i}\subset \cals^{ss}_{W_i}$ to be the image of $\tilde{\mathcal{X}}_i\subset \tilde{\mathcal{Z}}_i\subset \cals^+_{W_i}$ under $\phi_{W_i}^+$. In other words, $\mathcal{R}_{W_i}$ parameterizes pairs of the form
$$ (\cali_{P_i}(2k+c_1-1), 0) ~\bigoplus~ (\cali_{A_i}(1), 1) $$
where $P_i\cap H=\emptyset$. 

\item
Define $\tilde{\mathcal{X}}^-_{W_i}\subset \cals^-_{W_i}$ to be the pre-image of $\mathcal{R}_{W_i}$ under $\phi_{W_i}^-: \cals^-_{W_i}\to \cals^{ss}_{W_i}$. 

\item
There is a morphism $\Gamma: \cals_{W_{\emptyset}^-}\to \calh_{\p3}(P)$ (Theorem \ref{Gie-Hilb}). 
Let $\calh_i\subset \calh^{P}_{\p3}$ \linebreak ($i=1,2,...$) be the components of the Hilbert scheme.
By the theorem of semi-continuity, there is an open dense subscheme $\calh^{\circ}_i\subset \calh_i$ for every $i$, over which the fibers of $\Gamma$ are constant. Using the argument in Lemma \ref{irreducibility}, the subschemes $\tilde{\calh}_i^{\circ}\subset\cals_{W_{\emptyset}^-}$ lying over $\calh_i^{\circ}$ are irreducible. 
Define $\tilde{\calh}_i\subset \cals_{W_{\emptyset}^-}$ to be the closure of $\tilde{\calh}_i^{\circ}$ in $\cals_{W_{\emptyset}^-}$. 
Then, we have $\cals_{W^-_{\emptyset}}=\bigcup_{i=1}^{l_{v(k)}}\tilde\calh_i$. 

\item
Let $C_i=(W_{i}, W_{i-1})$ ($i=1,2,..., l_{v(k)}$) be the chambers between consecutive walls. Denote by $\tilde\calh_i'$ or $\tilde\calh_i^{(1)}$ the component (possibly empty) of $\cals_{W_0^-}$ obtained from $\tilde\calh_i$ when crossing $W_0$. 
Inductively, let $\tilde\calh_i^{(s)}$ ($s\in \mathbb{Z}_{\geq 1}$) be the component of $\cals_{W^-_{s-1}}$ obtained from $\tilde\calh_i^{(s-1)}$ when crossing $W_{s-1}$.

\item
Define $\cals^{+}_{j,W_i}:=\cals^{+}_{W_i}\bigcap \tilde\calh^{(i)}_j$
and 
$\cals^{-}_{j,W_i}:=\cals^{-}_{W_i}\bigcap \tilde\calh^{(i+1)}_j.$

\item
Denote by $\tilde\calh_{j, W_i}$ ($i,j \in \mathbb{Z}_{\geq 0}$) the components of $\cals_{W_i}$ when $\delta=W_i$. 

\item
Define $\cals^{ss}_{j,W_i}:=\cals^{ss}_{W_i}\bigcap \tilde\calh_{j,W_i}$.
\end{enumerate}
\end{definition}

Putting all the above together, we have the following commutative diagram (Figure (\ref{diagramwall})) illustrating the wall-crossing at a wall $W_i$.
The first and second rows show the components of the moduli space under a wall-crossing. The third and fourth rows show the subschemes contracted when hitting a wall and the closed subscheme $\calr_{W_i}$ over which the dimension of the fibers may jump.

\begin{figure}[ht]
\begin{tikzcd}
\cals_{{W_i}^{-}}
\simeq \bigcup \tilde\calh_j^{(i+1)}
\arrow[r,"\Phi^-_{{W_i}}"]&
\cals_{{W_i}}
\simeq \bigcup \tilde\calh_{j, W_i}
&
\cals_{{W_i}^{+}}
\simeq \bigcup \tilde\calh_j^{(i)}
\arrow[l, "\Phi^+_{{W_i}}" above]\\
\tilde\calh_{j}^{(i+1)}\arrow[r, "\Phi^-_{{W_i}}"]\arrow[u,hook]&\tilde\calh_{j,W_i}\arrow[u,hook]&\tilde\calh_{j}^{(i)}\arrow[l, "\Phi^+_{{W_i}}" above]\arrow[u,hook]\\
\cals^-_{j,{W_i}}\arrow[r, "\phi^-_{{W_i}}"]\arrow[u,hook]&\cals^{ss}_{j,{W_i}}\arrow[u,hook]&\cals^+_{j,{W_i}}\arrow[l, "\phi^+_{{W_i}}" above]\arrow[u,hook]\\
\tilde{\mathcal{X}}^-_{{W_i}}\arrow[r, "\phi^-_{{W_i}, v(k)}"]\arrow[u,hook]&\mathcal{R}_{{W_i}}\arrow[u,hook]&\tilde{\mathcal{X}}_{{i}}\arrow[l, "\phi^+_{{W_i}}" above]\arrow[u,hook].\\
\end{tikzcd}
\caption{Moduli spaces and morphisms at a wall-crossing.} {\label{diagramwall}}
\end{figure}

The next Lemma shows that the dimension of the fiber of $\phi^+_{W_i}$ for $i=0,1,...,l_{v(k)}$ is at least $2$.

\begin{lemma}{\label{Ext+}}
Let $A_i\subset H$ be a one-dimensional scheme contained in a plane $H\subset \p3$, and $P_i\subset \mathbb{P}^3$ be a $0-$dimensional scheme such that $P_i\cap A_i=\emptyset$. 
Then, for $k\geq 1$ we have $\ext^1((\mathcal{I}_{A_i}(1), 1), (\mathcal{I}_{P_i}(2k-1+c_1(E)), 0))\geq 3$.
\end{lemma}

\begin{proof}
We use the long exact sequence in Lemma \ref{LES} for pairs $\Lambda=(\mathcal{I}_{A_i}(1), 1)$ and $\Lambda'=(\mathcal{I}_{P_i}(2k-1+c_1(E)), 0)$. 
For simplicity, let $l:=2k-1+c_1$. $l\geq 1$ since otherwise, the wall $W$ defined by the pairs $(\mathcal{I}_{A_i}(1), 1)$ and $(\mathcal{I}_{P_1}, 0)$ will be negative.
The corresponding long exact sequence is as follows:
$$ 0\lra \Hom(\Lambda, \Lambda') \lra \Hom(\mathcal{I}_{A_i}(1), \mathcal{I}_{P_i}(l)) \lra \Hom(\mathbb{C}, H^0(\mathbb{P}^3, \mathcal{I}_{P_i}(l))) \lra  $$
$$ \lra \Ext^1(\Lambda, \Lambda') \lra \Ext^1(\mathcal{I}_{A_i}(1), \mathcal{I}_{P_i}(l)) \lra $$
$$ \lra \Hom(\mathbb{C}, H^1(\mathbb{P}^3, \mathcal{I}_{P_i}(l))) \stackrel{\phi}{\lra} \Ext^2(\Lambda, \Lambda') \Ext^2(\Lambda, \Lambda') \lra \cdots  $$

Firstly, we have that $\Hom(\Lambda, \Lambda')=\Hom((\mathcal{I}_{A_i}(1), 1), (\mathcal{I}_{P_i}(l), 0))=0$; otherwise, $\Hom(\mathcal{I}_{A_i}(1), \mathcal{I}_{P_i}(l))\neq 0$ induces injective morphisms $\cali_{A_i}(1)\hookrightarrow \cali_{P_i}(l)$. 
Then, $H^0(\p3, \cali_{A_i}(1))\subset H^0(\p3, \cali_{P_i}(l))$, which implies that any nontrivial section $s_1\in H^0(\p3, \cali_{A_i}(1))$ induces a section $s_2\in H^0(\p3, \cali_{P_i}(l))$. 
This contradicts with the assumption $\Hom((\cali_{A_i}(1),1),(\cali_{P_i}(l)),0)\neq 0$.

Then, the above long exact sequence implies the following equation:
\begin{equation}{\label{ext1}}
\begin{array}{rl}
\ext^1(\Lambda, \Lambda')=&h^0(\mathbb{P}^3, \mathcal{I}_{P_i}(l))-h^1(\mathbb{P}^3, \mathcal{I}_{P_i}(l))- \text{hom}(\mathcal{I}_{A_i}(1), \mathcal{I}_{P_i}(l))\\
&+\ext^1(\mathcal{I}_{A_i}(1), \mathcal{I}_{P_i}(l))+\dim(\im(\phi))\\
=&\chi(\mathcal{I}_{P_i}(l))-\text{hom}(\mathcal{I}_{A_i}(1), \mathcal{I}_{P_i}(l))\\
&+\ext^1(\mathcal{I}_{A_i}(1), \mathcal{I}_{P_i}(l))+\dim(\im(\phi))\\
=&\chi(\mathcal{O}_{\mathbb{P}^3}(l))-\chi(\mathcal{O}_{P_i}(l))-\hom(\mathcal{I}_{A_i}(1), \mathcal{I}_{P_i}(l))\\
&+\ext^1(\mathcal{I}_{A_i}(1), \mathcal{I}_{P_i}(l))+\dim(\im(\phi)).\\
\end{array}
\end{equation}
The second equality follows from the fact that $H^2(\mathbb{P}^3, \mathcal{I}_{P_i}(l))=H^3(\mathbb{P}^3, \mathcal{I}_{P_i}(l))=0$, and the third equality follows from the additivity of $\chi$.

We simplify the terms $\Hom(\mathcal{I}_{A_i}(1), \mathcal{I}_{P_i}(l))$ and $\Ext^1(\mathcal{I}_{A_i}(1), \mathcal{I}_{P_i}(l))$ by applying the functor $\Hom(\mathcal{I}_{A_i}(1), -)$ to the sequence $0\to \mathcal{I}_{P_i}(l)\to \mathcal{O}_{\mathbb{P}^3}(l)\to \mathcal{O}_{P_i} \to 0$.
The corresponding long exact sequence is as follows:
$$ 0\lra \Hom(\mathcal{I}_{A_i}(1), \mathcal{I}_{P_i}(l)) \lra  \Hom(\mathcal{I}_{A_i}(1), \mathcal{O}_{\mathbb{P}^3}(l)) \lra  \Hom(\mathcal{I}_{A_i}(1), \mathcal{O}_{P_i}) \lra $$
$$ \lra \Ext^1(\mathcal{I}_{A_i}(1), \mathcal{I}_{P_i}(l)) \lra \Ext^1(\mathcal{I}_{A_i}(1), \mathcal{O}_{\mathbb{P}^3}(l)) \lra \Ext^1(\mathcal{I}_{A_i}(1), \mathcal{O}_{P_i}) \lra \cdots $$
and we have $\Hom(\mathcal{I}_{A_i}(1), \mathcal{O}_{P_i})\simeq\mathbb{C}^{|P_i|}$ and $\Ext^1(\mathcal{I}_{A_i}(1), \mathcal{O}_{P_i})=0$. 
We also have 
$\Hom(\mathcal{I}_{A_i}(1), \mathcal{O}_{\mathbb{P}^3}(l))=\Hom(\mathcal{O}_{\mathbb{P}^3}(1), \mathcal{O}_{\mathbb{P}^3}(l))=H^0(\mathbb{P}^3, \mathcal{O}_{\mathbb{P}^3}(l-1))$ 
and
\begin{align*}
\Ext^1(\mathcal{I}_{A_i}(1), \mathcal{O}_{\mathbb{P}^3}(l)) &= \Ext^2(\mathcal{O}_{A_i}(1), \mathcal{O}_{\mathbb{P}^3}(l))\simeq  \\
& \simeq \Ext^1(\mathcal{O}_{\mathbb{P}^3}(l), \mathcal{O}_{A_i}(1))^{*} = H^1(\mathbb{P}^3, \mathcal{O}_{A_i}(-3-l))^*.
\end{align*} 
Therefore, 
\begin{flushleft}
$ \text{hom}(\mathcal{I}_{A_i}(1), \mathcal{I}_{P_i}(l))-\ext^1(\mathcal{I}_{A_i}(1), \mathcal{I}_{P_i}(l)) = $
\end{flushleft}
\begin{flushright}
$=h^0(\mathbb{P}^3, \mathcal{O}_{\mathbb{P}^3}(l-1))-h^1(\mathbb{P}^3, \mathcal{O}_{A_i}(-3-l))-|P_i|. $
\end{flushright}
Take this result back to equation (\ref{ext1}), and we get (using $\dim(\im(\phi))\ge0$)
\[
\begin{array}{rl}
\ext^1(\Lambda, \Lambda')=&\chi(\mathcal{O}_{\mathbb{P}^3}(l))-\chi(\mathcal{O}_{P_i}(l))-\text{hom}(\mathcal{I}_{A_i}(1), \mathcal{I}_{P_i}(l))\\
&+\ext^1(\mathcal{I}_{A_i}(1), \mathcal{I}_{P_i}(l))+\dim(\im(\phi))\\
\ge&\chi(\mathcal{O}_{\mathbb{P}^3}(l))-|P_i|-h^0(\mathbb{P}^3, \mathcal{O}_{\mathbb{P}^3}(l-1)) + \! h^1(\mathbb{P}^3, \mathcal{O}_{A_i}(-3-l)) \! +|P_i|\\
\ge&\chi(\mathcal{O}_{\mathbb{P}^3}(l))-h^0(\mathbb{P}^3, \mathcal{O}_{\mathbb{P}^3}(l-1))+h^1(\mathbb{P}^3, \mathcal{O}_{A_i}(-3-l))\\
\ge&\binom{l+3}{3}-\binom{l+2}{3}+h^1(\mathbb{P}^3, \mathcal{O}_{A_i}(-3-l))\\
\ge&\frac{1}{2}(l+2)(l+1)+h^1(\mathbb{P}^3, \mathcal{O}_{A_i}(-3-l)) \ge 3\\
\end{array}
\]
as desired.
\end{proof}

\begin{remark}{\label{constfiber}}
Due to Proposition \ref{bad pair} and the fact that the projectivization $\mathbb{P}\Big(\Ext^1(\cali_{Y_i}(2k+c_1-1), \calo_{\p3})\Big)$ is independent of $i$, we have that the dimension of $\ext^1((\mathcal{I}_{A_i}(1), 1), (\mathcal{I}_{P_i}(2k-1+c_1), 0))$ does not depend on $P_i$ provided $P_i\cap A_i\neq \emptyset$. 
\end{remark}

The next Lemma computes the dimension of the fibers of the morphisms $\phi^-_{W_i}$ over $\mathcal{R}_{W_i}$
\[
\begin{tikzcd}
\cals^-_{W_i}\arrow[r, "\phi^-_{W_i}"]&\cals^{ss}_{W_i}\\
\tilde{\mathcal{X}}^-_{W_i}\arrow[r, "\phi^-_{W_i}"]\arrow[u, hook]&\mathcal{R}_{W_i}\arrow[u, hook].\\
\end{tikzcd}
\]

\begin{lemma}{\label{Ext-}}
Assume that $A_i\subset H$ is a one-dimensional scheme, $Q_i\subset A_i$ be the maximum $0-$dimensional subscheme of $A_i$, and $P_i\subset \mathbb{P}^3$ is a $0-$dimensional scheme with $P_i\cap A_i=\emptyset$. Then we have $\ext^1((\mathcal{I}_{P_i}(2k-1+c_1), 0),(\mathcal{I}_{A_i}(1), 1))=|Q_i|$. 
\end{lemma}

\begin{proof}
Apply the Long exact sequence in Lemma \ref{LES} for $\Lambda=(\mathcal{I}_{P_i}(2k-1+c_1), 0)$ and $\Lambda'=(\mathcal{I}_{A_i}(1), 1)$, and we get the isomorphism
\[
\Ext^1((\mathcal{I}_{P_i}(2k-1+c_1), 0),(\mathcal{I}_{A_i}(1), 1))\simeq \Ext^1(\mathcal{I}_{P_i}(2k-1+c_1),\mathcal{I}_{A_i}(1)).
\]
We compute $\Ext^1(\mathcal{I}_{P_i}(2k-1+c_1),\mathcal{I}_{A_i}(1))$. For simplicity, define $l:=2k-1+c_1\geq 1$.
Applying the functor $\Hom(-, \mathcal{I}_{A_i})$ to the short exact sequence
$$ 0\to \cali_{P_i}(l)\to \calo_{\p3}(l)\to \calo_{P_i}\to 0, $$
we obtain
\begin{align*}
\cdots &\lra \Ext^1(\mathcal{O}_{P_i}, \mathcal{I}_{A_i}(1)) \lra \Ext^1(\mathcal{O}_{\mathbb{P}^3}(l), \mathcal{I}_{A_i}(1))
\lra \Ext^1(\mathcal{I}_{P_i}(l), \mathcal{I}_{A_i}(1))\\
&\lra \Ext^2(\mathcal{O}_{P_i}, \mathcal{I}_{A_i}(1)) \lra \cdots
\end{align*}

Next, we prove that $\Ext^1(\mathcal{O}_{P_i}, \mathcal{I}_{A_i}(1))=\Ext^2(\mathcal{O}_{P_i}, \mathcal{I}_{A_i}(1))=0$. Apply the functor $\Hom(\mathcal{O}_{P_i}, -)$ to the sequence $0\to \mathcal{I}_{A_i}(1)\to \mathcal{O}_{\mathbb{P}^3}(1)\to \mathcal{O}_{A_i}(1) \to 0$, and we get a long exact sequence
$$ 0 \lra \Ext^1(\mathcal{O}_{P_i}, \mathcal{I}_{A_i}(1)) \lra \Ext^1(\mathcal{O}_{P_i}, \mathcal{O}_{\mathbb{P}^3}(1)) \lra \Ext^1(\mathcal{O}_{P_i}, \mathcal{O}_{A_i}(1)) \lra $$
$$ \lra  \Ext^2(\mathcal{O}_{P_i}, \mathcal{I}_{A_i}(1)) \lra \Ext^2(\mathcal{O}_{P_i}, \mathcal{O}_{\mathbb{P}^3}(1)) \lra \Ext^2(\mathcal{O}_{P_i}, \mathcal{O}_{A_i}(1)) \lra \cdots $$
In this sequence,  $\Ext^1(\mathcal{O}_{P_i}, \mathcal{O}_{\mathbb{P}^3}(1))=\Ext^2(\mathcal{O}_{P_i}, \mathcal{O}_{\mathbb{P}^3}(1))=0$ since $\dim P_i=0$,
and $\Ext^1(\mathcal{O}_{P_i}, \mathcal{O}_{A_i}(1))=0$ since $P_i\cap A_i=\emptyset$. This implies that
$$ \Ext^1(\mathcal{O}_{P_i}, \mathcal{I}_{A_i}(1))=\Ext^2(\mathcal{O}_{P_i}, \mathcal{I}_{A_i}(1))=0. $$ 

Therefore, $\Ext^1(\mathcal{I}_{P_i}(l), \mathcal{I}_{A_i}(1))\simeq \Ext^1(\mathcal{O}_{\mathbb{P}^3}(l), \mathcal{I}_{A_i}(1))\simeq H^1(\mathbb{P}^3, \mathcal{I}_{A_i}(-l+1))$. Consider the cohomologies of the sequence
$$ 0\lra \mathcal{I}_{A_i}(-l+1) \lra \mathcal{O}_{\mathbb{P}^3}(-l+1) \lra \mathcal{O}_{A_i}(-l+1)\lra 0, $$
which is
$$ 0\to H^0(\mathbb{P}^3, \mathcal{O}_{\mathbb{P}^3}(-l+1))\to H^0(\mathbb{P}^3, \mathcal{O}_{A_i}(-l+1))\to H^1(\mathbb{P}^3, \mathcal{I}_{A_i}(-l+1))\to 0. $$

For $l=1$, we have $h^1(\mathbb{P}^3, \mathcal{I}_{A_i})=h^0(\mathbb{P}^3, \mathcal{O}_{A_i})-h^0(\mathcal{O}_{\mathbb{P}^3}))=|Q_i|$. For $l\geq 2$, we have $h^1(\mathbb{P}^3, \mathcal{I}_{A_i}(-l+1))=h^0((\mathbb{P}^3, \mathcal{O}_{A_i}(-l+1)))=|Q_i|$.
\end{proof}

The results in Lemma \ref{surjtocurve}, Remark \ref{dim Zi tilde}, Lemma \ref{Ext+}, \ref{Ext-}, and the elementary modification in \cite[Section 4.8]{he1996espaces} imply the next result.

\begin{theorem}{\label{flipblowdownremoval}}
When $\delta$ gets smaller and crosses walls $W_i$ ($i=0,1,2,...,l_{v(k)}$), $\tilde{\calh}_{pl}$ undergoes a series of flips for $i=0,1,2,...,l_{v(k)}-2$.
An open subscheme of $\tilde{\calh}^{(l_{v(k)-1})}_{pl}$ undergoes a  divisorial contraction when crossing $W_{l_{v(k)}}-1$, and lastly, the open subscheme $\tilde{\mathcal{X}}_{l_{v(k)}}\subset \tilde{\calh}^{(l_{v(k)})}_{pl}$ is removed when crossing $W_{l_{v(k)}}$.
\end{theorem}

\begin{proof}
Lemma \ref{surjtocurve}, Remark \ref{dim Zi tilde}, Lemmas \ref{Ext+}, \ref{Ext-}, and the elementary modification in \cite[Section 4.8]{he1996espaces} imply that when crossing $W_i$ as $\delta$ gets smaller, the set $\displaystyle\bigcup_{j=0}^{l_{v(k)}} \tilde{\mathcal{X}}_j$ as an irreducible open subset of $\tilde{\calh}_{pl}$, undergoes flips for $i=0,1,2,..., l_{v(k)-2}$, a divisorial contraction for $i=l_{v(k)}-1$, and a removal of an open subscheme $\tilde{\mathcal{X}}_{l_{v(k)}}$ for $i=l_{v(k)}$. 

Furthermore, the dimensions of the fibers for the morphism
\[
\phi^{\pm}_{W_i}: \cals^{\pm}_{W_i}\to \cals^{ss}_{W_i}
\]
may jump at the boundary of $\calr_{W_i}\subset \cals^{ss}_{W_i}$. 
So for the total space $\widetilde{H}_{pl}$, when crossing $W_i$, it is still a flip for $i=0,1,2,..., l_{v(k)}-2$. 
However, it may not be a divisorial contraction (for $i=l_{v(k)}-1$) or a removal of an open subscheme (for $i=l_{v(k)}$) for the total space since the dimension of the fiber for $\phi^-_{W_i}$ may jump. We will see some counterexamples for this in Sections 6.2 and 6.3.
\end{proof}

\begin{remark}{\label{fiber dim jump}}
The dimension of the fibers of $\phi^-_{W_i}$ may jump on $\tilde{\mathcal{X}}^-_{W_i}\backslash \mathcal{R}_{W_i}$. We will see some examples in Section $6$.
\end{remark}


\begin{corollary}{\label{connectedness}}
Every wall-crossing at $W_i$ ($i=0,1,...,l_{v(k)}$) preserves the connectedness of the moduli space $\cals^{\delta}_{\p3}$.
\end{corollary}

\begin{proof}
It is evident that wall-crossings at $W_i$ for $i=0,1,..., l_{v(k)}-1$ preserve the connectedness of $\cals^{\delta}_{\p3}$. 
Wall-crossing at $W_{l_{v(k)}}$ removes the open subscheme $\tilde{\mathcal{X}}_{l_{v(k)}}$ that does not intersect with any other components in the moduli space of pairs. 
So removing it preserves the connectedness as well.  
\end{proof}


\section{Some examples on $\p3$} \label{sec:ex}

In this section, we illustrate the general theory of Sections \ref{sec:wallx} and \ref{sec:critical} with concrete examples on $\p3$. 
We work with specific rank $2$ numerical classes $v$. For the first three classes, the corresponding Hilbert schemes and a few of the corresponding Gieseker moduli spaces are well understood.  
For the last two classes, the Hilbert scheme is known, and we use wall-crossings to study the Gieseker moduli space.
For each class, we describe:
\begin{itemize}
\item The geometry of the associated Hilbert scheme $\calh_{\p3}^P$ and the Gieseker moduli space $\calg_{\p3}^v$.
\item The complete list of walls for a suitably twisted class $v(k)$.
\item The birational transformations of the moduli space of pairs $\cals^{\delta}_{\p3}$ as $\delta$ crosses each wall.
\end{itemize}
To simplify computations, we use Chern characters throughout.

\subsection{The class $v=(2,-1,-1/2,5/6)$} \label{subsec:no-walls}

\subsubsection{Description of the correspondence} 

Let us describe the one-dimensional schemes corresponding to a semistable sheaf $E$.

\begin{proposition}{\label{sheaf to line}}
If $E\in \Coh(\mathbb{P}^3)$ with $v(E)=v$, then $E$ is Gieseker semistable if and only if $E$ is a non-trivial extension in the following short exact sequence
\begin{equation}{\label{E-L}}
    0\lra \calo_{\p3}(-1) \lra E \lra \cali_{L} \lra 0,
\end{equation}
where $L\subset \p3$ is a line. 
Moreover, such sheaves are always reflexive.
\end{proposition}

\begin{proof}

Due to the choice of $v$, the following four notions of stability for $E$ are equivalent: $\mu$-semistable, $\mu$-stable, Gieseker semistable, and Gieseker stable.

Suppose $E$ is a non-trivial extension in the above sequence, and $E$ is unstable, destabilized by a rank $1$ sub-sheaf $F$.  
Then, $F\simeq \cali_X(k)$ where $X\subset \p3$ with $\dim(X)\leq 1$, and $k\geq 0$. 

$\bullet$ If $k\geq 1$, the composition $\cali_X(k)\to E \to \cali_L$ is zero, which induces a morphism $\cali_X(k)\to \calo_{\p3}(-1)$. This morphism is zero as well, which leads to a contradiction. 

$\bullet$ If $k=0$, the composition $\cali_X\to E \to \cali_L$ must be non-zero, otherwise, we must have a non-zero map $\cali_X \to \calo_{\p3}(-1)$, which is absurd. 
Since $\cali_L$ is torsion-free, the composition $\cali_X\to E \to \cali_L$ is an inclusion. This implies $L\subset X$ and the following diagram

\[
\begin{tikzcd}
&&0\arrow[d]&0\arrow[d]&\\
&&\cali_X \arrow[d]\arrow[r, "="]&\cali_X \arrow[d]\\
0 \arrow[r] & \calo_{\p3}(-1)\arrow[r, "s"] \arrow[d, "="]& E \arrow[d, "p"] \arrow[r] & \cali_L \arrow[r] \arrow[d] &0 \\
0\arrow[r]&\calo_{\p3}(-1) \arrow[r, "i"]& Q\arrow[r]\arrow[d] & \cali_{L|X}\arrow[r]\arrow[d] &0\\
&&0&0.&\\
\end{tikzcd}
\]
Since $\dim(X)\leq 1$ we have that $\Ext^1(\cali_{L|X}, \calo_{\p3}(-1))\simeq H^2(\cali_{L|X}(-3))^*=0$. 
Hence, $Q\simeq \calo_{\p3}(-1)\oplus \cali_{L|X}$ and $p\circ s=i$. But this implies that the sequence $0\to \calo_{\p3}(-1) \to E \to \cali_{L} \to 0$ is split, which is a contradiction. 
Therefore, any nontrivial extension $E$ in the sequence  (\ref{E-L}) is Gieseker semistable.

Conversely, if $E$ is Gieseker semistable with $v(E)=v$, we show below that it is reflexive. 
Then, it follows from \cite[Lemma 9.4]{hartshorne1980stable} that $E$ fits into the short exact sequence in display  (\ref{E-L}). 

Suppose $E$ is not reflexive, then consider the short exact sequence
$$0\lra E \lra E^{\vee\vee} \lra Q \lra 0.$$
$E^{\vee \vee}$ is $\mu$-stable and $Q$ is supported on a one-dimensional scheme. 
The Bogomolov inequality $$(H^2\ch^1(E^{\vee \vee}))^2-2H\ch_0(E^{\vee \vee})\ch_2(E^{\vee \vee})\geq 0$$
implies that $H\ch_2(E^{\vee \vee})\leq \frac{1}{4}$, hence, $H\ch_2(Q)\leq \frac{3}{4}$. 
On the other hand, $H\ch_2(Q)\in \mathbb{Z}_{\geq 0}$, so we must have $\ch_2(Q)=0$, and dim$(Q)=0$. 
Moreover, \cite[Theorem 8.2 (d)]{hartshorne1980stable} implies $c_3(E^{\vee \vee})\leq 1$ (equivalently $\ch_3(E^{\vee \vee})\leq \frac{5}{6}$).
Therefore, the assumption $c_3(E)=1$ ($\ch_3(E)=\frac{5}{6}$) forces $Q$ to be empty, and $E\simeq E^{\vee\vee}$ is reflexive.
\end{proof}

\subsubsection{All possible walls}
We consider wall-crossings for the class $v(1)$ since Proposition \ref{sheaf to line} implies that $H^0(\p3, E)=0$ and $H^0(\p3, E(1))\neq 0$. 
The collapsing wall $W_{\emptyset}$ is defined by $(\calo_{\p3}, 1)\hookrightarrow (E(1), 1)$, 
and a direct computation shows that $W_{\emptyset}=\frac{1}{2}t^2+\frac{3}{2}t+1$.

Indeed, there are no other walls other than the collapsing wall. 
This is because a wall in $(0, W_{\emptyset})$ is defined by $(\cali_A(k),1)\hookrightarrow(E(1), 1)$ for some one-dimensional scheme $A\subset \p3$ and $k\geq 1$. 
One computes directly that the corresponding critical value for this wall is  $\delta=P_{E(1)}(t)-2P_{\cali_A(k)}(t)$, which is negative for all $k\geq 1$. 
Hence, there are no walls in group $k$ ($k\geq 1$), which means that the collapsing wall is the unique wall.
One also sees this fact from Proposition \ref{Lem:1wallleft}.
Let $W_G>0$ be the smallest wall. Then, the sheaf $E$ in a $\delta$-stable pair $(E(1), s)$ ($\delta\in (0, W_G)$) is Gieseker semistable. 
Proposition \ref{sheaf to line} implies that $E$ is reflexive. 
Moreover, Proposition \ref{Lem:1wallleft} shows that the first wall from the left is defined by $(\calo_{\p3}, \mathbf{1})\hookrightarrow (E(1), 1)$, which coincides with the collapsing wall.
This shows that the collapsing wall is the unique wall for $v(1)$.

\subsubsection{Interpretation of the moduli space of stable pairs}
We show below that the moduli space of pairs $\cals^{\delta}_{\p3}$ ($\delta\in (0, W_{\emptyset})$) is the flag variety $\mathcal{F}:=\{P\in \p1 \subset \p3\}$, where $P\in \mathbb{P}^1$ is a point. 

Since $(0, W_{\emptyset})$ is the only chamber, there exist the two morphisms $\Psi: \cals^{\delta}_{\p3}\to \calg_{\p3}$ and $\Gamma: \cals^{\delta}_{\p3}\to \calh_{\p3}^{t+1}$. 
Consider the morphism $\Psi: \cals^{\delta}_{\p3} \to \calg_{\p3}$ to the Gieseker moduli space. The fiber over a stable sheaf $E(1)$ is $\mathbb{P}(H^0(\p3, E(1)))$. 
Proposition \ref{sheaf to line} implies that $E$ fits into the sequence $0\to \calo_{\p3}\to E(1) \to \cali_L(1) \to 0$, and $H^0(\p3, E(1))\simeq\mathbb{C}^3$.
Therefore, $E$ has the following resolution:

$$\calo_{\p3}(-2) \overset{N}{\to} \calo_{\p3}(-1)^{\bigoplus 3} \to E.$$
$N$ is the matrix $N=(x_1, x_2, x_3)^T$ whose entries are three linearly independent sections $x_1, x_2, x_3\in H^0(\p3, \calo_{\p3}(1))$. 
Thus, $\calg_{\p3}\simeq \p3$. Moreover, for a fixed $E$ corresponding to the triple $(x_1, x_2, x_3)$, consider the short exact sequence of complexes 
$$0 {\lra} \calo_{\p3} \overset{s}{\lra} E(1) \lra \cali_L(1) \simeq \left[\calo_{\p3}(-2) \overset{(x_1, x_2)}{\lra} \calo_{\p3}(-1)^{\bigoplus 2}\right] \lra 0.$$
We see that a section $s\in H^0(\p3, E(1))$ corresponds to a line $L\subset \p3$, defined by $Z(x_1, x_2)$, passing through the point $P=Z(x_1, x_2, x_3)$. 
Therefore, $\cals^{\delta}_{\p3}\simeq \mathcal{F}$, and $\Psi: \cals^{\delta}_{\p3} \to \calg_{\p3}$ is the map $\mathcal{F}=\{P \in L \subset \p3\} \to \{P\in \p3 \}$, whose fiber at $P$ is the variety $\{L\subset \p3| P\in L \}\simeq \p2$.

Similarly, consider the morphism $\Gamma: \cals^{\delta}_{\p3} \to \calh^{t+1}_{\p3}=Gr(2,4)$, where the fiber at a line $L\subset \p3$ is $\mathbb{P}(\Ext^1(\cali_L(1), \calo_{\p3}))=\p1$. 
For each line $L=Z(x_1, x_2)$, an extension class in $\Ext^1(\cali_L(1), \calo_{\p3})$ corresponds to choosing an $x_3$ such that $\{x_1, x_2, x_3\}$ are linearly independent.
This is equivalent to choosing a point $P$ in $L$. 
This implies the other fibration $\mathcal{F}\to \{L\subset \p3\}$, whose fibers are $\{P\in \p3|P\in L$\}.


\subsection{The class \texorpdfstring{$v=(2,0,-1,0)$}{v=(2,0,-1,0)}} \label{subsec:2}

 \texorpdfstring{$v=(2,0,-1,0)$}{v=(2,0,-1,0)}

\subsubsection{Moduli spaces of stable pairs for $\delta= 0^+$ and $\delta= W_{\emptyset}^-$}

We start with the moduli spaces that admit morphisms to the Gieseker moduli space and the Hilbert scheme.

It is proved in \cite{jardim2017two} that every rank $2$ semistable sheaf $E$ with class $v$ is a null-correlation sheaf, meaning that $E$ fits into the following short exact sequence,
\begin{equation}{\label{nullcor}}
0\lra \calo_{\p3}(-1) \overset{\sigma}{\lra} \Omega^1_{\p3}(1)\lra E \lra 0.
\end{equation}
$E$ is uniquely determined by the section $\sigma$ up to a scalar multiplication, so $\calg^v_{\p3}\simeq \mathbb{P}(H^0(\p3, \Omega^1_{\p3}(2)))=\p5$. It is evident that for such a semistable sheaf $E$, $H^0(\p3, E)=0$ and $H^0(\p3, E(1))\simeq\mathbb{C}^5$. So we work on the wall-crossings for the class $v(1)$. 

$\bullet$
For $\delta=0^+$, sequence  (\ref{nullcor}) implies that there is a morphism $\cals^{0^+}_{\p3}\to \p5$ with fibers being $\p4$.
So $\cals^{0^+}_{\p3}$ is irreducible (Lemma \ref{irreducibility}) and $\dim(\cals^{0^+}_{\p3}$)$=9$.

$\bullet$
For $\delta=W_{\emptyset}^-$, consider the sequence $0\to \calo_{\p3} \to E(1) \to \cali_Y(1)\to 0$, where $Y\subset \p3$ is a one-dimensional scheme. 
A direct computation shows that the Hilbert polynomial of $\calo_{Y}$ is $P_{\calo_Y}(t)=2t+2$. 
According to \cite[Corollary $1.6$]{nollet1997hilbert}, a general one-dimensional scheme $Y$ falls into the following two possibilities, 
\begin{enumerate}
\item[$H_1$:] $Y$ is a union of two skew lines with $p_a(Y)=-1$.
\item[$H_2$:] $Y$ is a union of a plane quadric curve and a point $P\in \p3$.
\end{enumerate}

Denote by $H_1$ and $H_2$ the corresponding components in $\calh^{2t+2}_{\p3}$ for one-dimensional schemes of types $H_1$ and $H_2$.
The dimensions of $H_1$ and $H_2$ are $8$ and $11$ respectively, and $H_1\bigcap H_2$ is a divisor in $H_1$. 
Moreover, the fiber of the morphism
\[
\Gamma: \cals^{W^-_{\emptyset}}_{\p3}\to \calh^{2t+2}_{\p3}
\]
at $[\cali_Y]\in \calh^{2t+2}_{\p3}$ is the extension classes $\mathbb{P}\Ext^1(\cali_Y(2), \calo_{\p3})$. 
One computes that

\[
\Ext^1(\cali_Y(2), \calo_{\p3})=
\begin{cases}
\mathbb{C}^2 & \text{if \ } Y \in H_1\backslash H_2\\
\mathbb{C}^3 & \text{if \ } Y \in H_2.\\
\end{cases} 
\]
Therefore, $\cals^{W_{\emptyset}^-}_{\p3}$ has two components, and we denote them by $\tilde\calh_1$ and $\tilde\calh_2$ (Section \ref{sec: birat'l trans}).
There are morphisms $\tilde\calh_1\to \calh_1$ and $\tilde\calh_2\to \calh_2$, whose (general) fibers are $\p1$ and $\p2$ respectively.
Furthermore, the dimensions of $\tilde\calh_1$ and $\tilde\calh_2$ are $9$ and $13$.

\subsubsection{Finding all the walls for $v(1)$}

Firstly, the collapsing wall is defined by the sub-pair $(\calo_{\p3}, \mathbf{1})\hookrightarrow(E(1), 1)$, and it is easy to compute the corresponding critical value, $\delta=P_{E(1)}(t)-2P_{\calo_{\p3}}(t)=t^2+4t+3$. 

Next, a general numerical wall between $0$ and $W_{\emptyset}$ is given by a short exact sequence of pairs of the following form 

$$0\lra (\cali_X(k), 1) \lra (E(1), 1) \lra (\cali_Z(2-k), 0)  \lra 0,$$
where $X, Z\subset \p3$ are one-dimensional schemes and $k\in \mathbb{Z}$.
The critical value is $\delta=P_{E(1)}(t)-2P_{\cali_X(k)}(t)$. The integer
$k$ must be equal to $1$ in order to have $0<\delta<W_{\emptyset}$, and in this case, $P_{\calo_X}(t)+P_{\calo_Z}(t)=t+2$. 
We should also have $P_{\calo_X}(t)>P_{\calo_Z}(t)$ because of the equation $P_{\cali_X}(t)+\delta=P_{\cali_Z}(t)$ ($\delta>0$). 
Therefore, there are two possible walls between $0$ and $W_{\emptyset}$ as shown in the following chart. 

\medskip

\begin{center}
\renewcommand{\arraystretch}{1.2}
\begin{tabular}{|l|l|l|}
\hline
Hilbert polynomials  & Descriptions for $X$ and $Z$& Critical value \\
\hline
$P_{\calo_X}(t)= t+2$, $P_{\calo_Z}(t)=0$ & $X= \text{line}\cup \text{point}$; $Z= \emptyset$& $\delta= t+1$\\
\hline
$P_{\calo_X}(t)= t+1$, $P_{\calo_Z}(t)=1$ & $X=$ \text{a line}; $Z= $\text{a point}&$\delta=t-1$\\
\hline
\end{tabular}
\end{center}

\medskip

We summarize all the information in the following chart.
It turns out that $W_0$ and $W_1$ are the walls defined by $0-$dimensional schemes in Section \ref{subsec:wall0dim}.  

\medskip

\begin{center}
\renewcommand{\arraystretch}{1.5}
\begin{tabular}{|l|l|l|}
\hline
Walls & Wall-defining sequences & Critical values \\
\hline
$W_1:$ & $0 \lra (\cali_{\p1}(1), 1) \lra (E(1), 1) \lra (\cali_{P}(1), 0)\lra 0$ & $\delta=t-1$\\
\hline
$W_0:$ &  $0 \lra (\cali_{\p1\cup P}(1), 1) \lra (E(1), 1) \lra (\calo_{\p3}(1), 0)\lra 0$ & $\delta=t+1$\\
\hline
$W_{\emptyset}:$ &  $0 \lra (\calo_{\p3}, 1) \lra (E(1), 1) \lra (\cali_{Y}(2), 0)\lra 0$ & $\delta=t^2+4t+3$\\
\hline
\end{tabular}
\end{center}

\subsubsection{Wall-crossings}

We start with some computations.

\begin{lemma}{\label{(010)}} 
Let $P\in \p3$ and $\p1\subset \p3$ be a point and a line in $\p3$. We have the following computational results.

\begin{enumerate}
\item $\Ext^1((\cali_{\mathbb{P}^1\cup P}(1),1), (\calo_{\p3}(1),0))\simeq\mathbb{C}^6$

\item $\Ext^1((\calo_{\p3}(1),0)), (\cali_{\p1\cup P}(1),1)\simeq\mathbb{C}$

\item $\Ext^1((\cali_{\p1}(1), 1) ,(\cali_P(1), 0))\simeq\mathbb{C}^6$ 
\item $\Ext^1((\cali_P(1), 0), (\cali_{\p1}(1), 1))=
\begin{cases}
\mathbb{C}& \text{for \ } P\in \p1\\
0& \text{for \ } P\not\in \p1.  
\end{cases}$ 

\end{enumerate}
\end{lemma}

\begin{proof}
For $\Ext^1((\cali_{\mathbb{P}^1\cup P}(1),1), (\calo_{\p3}(1),0))$ in (1), consider the following long exact sequence in Lemma \ref{LES}:
$$ \Hom((\cali_{\p1\cup P}(1), 1), (\calo_{\p3}(1), 0))=0 \lra  \Hom(\cali_{\p1\cup P}(1), \calo_{\p3}(1)) \simeq \mathbb{C} \lra $$
$$ \lra \Hom(\mathbb{C}, H^0(\calo_{\p3}(1)))\simeq\mathbb{C}^4
\lra \Ext^1((\cali_{\p1\cup P}(1), 1), (\calo_{\p3}(1), 0)) \lra $$
$$ \lra \Ext^1(\cali_{\p1\cup P}(1), \calo_{\p3}(1))\simeq\mathbb{C}^3 \lra 0 = \Hom(\mathbb{C}, H^1(\calo_{\p3}(1))). $$
We have that $\Ext^1((\cali_{\mathbb{P}^1\cup P}(1),1), (\calo_{\p3}(1),0))\simeq\mathbb{C}^6$.
 Remark \ref{constfiber} implies that $\Ext^1((\cali_{\p1}(1), 1) ,(\cali_P(1), 0))\simeq\mathbb{C}^6$.
The two extension groups below follow from Lemma \ref{Ext-}.
$$ \Ext^1((\calo_{\p3}(1),0)), (\cali_{\p1\cup P}(1),1)\simeq\mathbb{C}$$
$$\Ext^1((\cali_P(1), 0), (\cali_{\p1}(1), 1))=0\ \text{for}\ P\not\in\p1. $$
Finally,
$\Ext^1((\cali_P(1), 0), (\cali_{\p1}(1), 1))\simeq \Ext^1(\cali_P(1), \cali_{\p1}(1))$. One computes easily that $\Ext^1(\cali_P(1), \cali_{\p1}(1))\simeq\mathbb{C}$ for $P\in\p1$.
\end{proof}

Lemma \ref{(010)} and Theorem \ref{flipblowdownremoval} imply that when crossing $W_0$ and $W_1$, $\widetilde{H}_2$ undergoes a divisorial contraction and a removal of the open subscheme $\tilde{\mathcal{X}}_1$ lying over $\mathcal{X}_1=\{[Y]\in H_2| Y=C \bigcup P, P\not\in C \}\subset H_2$ ($C\subset \p3$ is a conic curve).

\begin{remark}{\label{H1blowdown}}
When crossing $W_0$, the other component $\tilde\calh_1$ undergoes a divisorial contraction as well. The intersection $\tilde\calh_1\cap \cals_{W_0}^{+}$ is $\p1-$fibered over the subscheme in $\calh^{2t+2}_{\p3}$ parameterizing planar one-dimensional schemes that are the union of a singular conic and an embedded point at the singularity. 
We have that dim$(\tilde\calh_1\cap \cals_{W_0}^{+})=8$, so $\tilde\calh_1\cap \cals_{W_0}^{+}$ is a divisor in $\tilde\calh_1$. 
Therefore, when crossing $W_0$, the divisor $\tilde\calh_1\cap \cals_{W_0}^{+}\subset \tilde\calh_1$ is contracted, meaning that $\tilde\calh_1$ undergoes a divisorial contraction as well.  
\end{remark}

We summarize the complete wall-crossings for $v(1)$ as follows.

    $\bullet$ When $\delta\in (W_0, W_{\emptyset})$, $\cals^{\delta}_{\p3}\simeq \tilde\calh_1\cup \tilde\calh_2$, in which the two components are lying over $H_1$ and $H_2$. The dimensions $\tilde\calh_1$ and $\tilde\calh_2$ are $9$ and $11$ respectively. 
    
    $\bullet$ When $\delta$ crosses $W_0$ and lands in $(W_1, W_0)$, $\tilde\calh_2$ is contracted along the divisor $\tilde{\mathcal{Z}}_0$, and $\tilde\calh_1$ is contracted along the divisor $\tilde\calh_1\cap \tilde{\mathcal{Z}}_0$. 
    Denote by $\tilde\calh_1'$ and $\tilde\calh_2'$ the contractions of $\tilde\calh_1$ and $\tilde\calh_2$ respectively.

    The following diagram illustrates the contraction $\tilde\calh_2\to \tilde\calh_2'$ when crossing $W_0$.

    \[
    \begin{tikzcd}
\tilde\calh_2'\arrow[r,"\simeq"]&\tilde\calh_{2,W_0}&\tilde\calh_2\arrow[l,"blow-down",swap]&\\
\cals^-_{W_0}\arrow[u,hook]\arrow[r,"\simeq"]&\begin{matrix}\cals^{ss}_{W_0}\\ \simeq 
\calh^{t+2}(\mathcal{U}_{\mathbb{P}^3}/{\mathbb{P}^{3\vee}})\end{matrix} \arrow[u,hook]&\cals^+_{W_0}\arrow[u, hook]\arrow[l,"\phi^+_{W_0}",swap]&\p5\arrow[l,swap, "fiber"].\\
    \end{tikzcd}
    \]
    
    The strictly semistable locus at $W_0$, denoted by $\cals^{ss}_{W_0}$, parameterizes pairs $(\cali_{\mathbb{P}^1\cup P}(1), 1)\bigoplus(\calo_{\mathbb{P}^3}(1),0)$. So it is isomorphic to the relative Hilbert scheme $\calh^{t+2}(\mathcal{U}_{\mathbb{P}^3}/{\mathbb{P}^{3\vee}})$, where $\mathcal{U}_{\mathbb{P}^3}\subset \p3\times\mathbb{P}^{3\vee}$ is the universal plane in $\mathbb{P}^3$. 

    Up to this point, for $\delta\in (W_1,W_{0})$, $$\cals_{W^-_0}\simeq \tilde\calh_1'\cup \tilde\calh_2'\simeq \tilde\calh_1'\cup (\cals^-_{W_0}\cup \tilde{\mathcal{Z}}_1).$$
    The closed subscheme $\tilde{\mathcal{Z}}_0\subset \cals_{W_{\emptyset}^-}$ is replaced by $\cals_{W_0}^-$ when crossing $W_0$.
    
    $\bullet$ When $\delta$ crosses $W_1$ and lands in $(0, W_1)$, $\tilde{\mathcal{X}}_1$ ($\tilde{\mathcal{X}}_1\subset \tilde{\mathcal{Z}}_1\subset \tilde\calh_2$) is removed, while $\tilde{V}_1\subset \tilde{\mathcal{Z}}_1$ is contracted at $W_1$ and then replaced by $\tilde{V}^-_1$ when $\delta$ lands in $(0, W_1)$.
    Note that $\tilde{V}_1$ is also contained in $\tilde\calh_1'$. 
    Therefore, for $\delta\in (0, W_1)$ the moduli space is  
    \[
    \cals_{0^+}\simeq \tilde\calh_1''\simeq (\tilde\calh_1\backslash (\tilde{\mathcal{Z}}_0\cup\tilde{\mathcal{Z}}_1))\cup \cals^-_{W_0}\cup \tilde{V}^-_1.
    \]
    Stable pairs in the first term $\tilde\calh_1\backslash (\tilde{\mathcal{Z}}_0\cup\tilde{\mathcal{Z}}_1)$ are both very stable (i.e., the underlying sheaves are semistable) and saturated, and their corresponding one-dimensional schemes are two skew lines. 
    Stable pairs in $\cals^-_{W_0}\cup \tilde{V}^-_1$ are unsaturated, but they are very stable.
    Besides, this moduli space is irreducible, admitting a $\mathbb{P}^4$ fibration over $\mathbb{P}^5$.


\subsection{The class \texorpdfstring{$v=(2,0,-3,4)$}{v=(2,0,-3,4)}} \label{subsec:3}

\subsubsection{The moduli space of pairs lying over the Gieseker moduli}

It is proved in \cite[Section 4.3, 4.5]{Schmidt2018RankTS} that the Gieseker moduli space $\calg^v_{\p3}$ is smooth, irreducible of dimension $21$. 
Moreover, for every semistable sheaf $E$ with $v(E)=v$, we have $H^0(\mathbb{P}^3, E)=0$ and $H^0(\mathbb{P}^3, E(1))\simeq\mathbb{C}^3$. 
Therefore, we consider wall-crossings for class $v(1)$, and the moduli space $\cals_{0^+}$ is irreducible of dimension $23$.

\subsubsection{The moduli space of pairs lying over the Hilbert scheme}

Let $(E(1), s)$ be a saturated pair corresponding to a one-dimensional scheme $Y\subset \mathbb{P}^3$, in which $v(E)=v$. By definition, the correspondence is given by the sequence
$$ 0\lra \mathcal{O}_{\mathbb{P}^3}\lra E(1) \lra \mathcal{I}_Y(2) \lra 0. $$ 
A direct computation shows that the Hilbert polynomial of $Y$ is $P_{\mathcal{O}_Y}(t)=4t$. It is proved in \cite{chen2012detaching} that the Hilbert scheme $\calh^{4t}_{\mathbb{P}^3}$ has two components $H_1$ and $H_2$. The dimensions and the general one-dimensional schemes in these two components are in the following chart.
\begin{center}
\begin{tabular}{|c|c|c|}
\hline
Components& Dimension& A general one-dimensional scheme in the component\\
\hline
    $H_1$ &  $23$ & Plane quartic curve with two points\\
    \hline
    $H_2$ & $16$& Elliptic quartic (CM)\\
\hline
\end{tabular}
\end{center}
A general one-dimensional scheme in $H_2$ satisfies the short exact sequence (\cite[Example 2.8]{chen2012detaching})
\[
0\lra \mathcal{O}_{\mathbb{P}^3}(-4)\oplus\mathcal{O}_{\mathbb{P}^3}(-3)\lra \mathcal{O}_{\mathbb{P}^3}(-3)\oplus\mathcal{O}_{\mathbb{P}^3}(-2)^{\oplus 2}\lra \mathcal{I}_{Y} \lra 0,
\]
and $\Ext^1(\cali_Y,\cali_Y)\simeq\mathbb{C}^{16}$.
Furthermore, we have the following results for the extension classes:
\[
\Ext^1(\mathcal{I}_Y(2), \mathcal{O}_{\mathbb{P}^3})=
\begin{cases}
    \mathbb{C}^{10}& \text{for}\ Y\in H_1 \\
    \mathbb{C}^{8}& \text{for}\ Y\in H_2\backslash H_1.\\
\end{cases}
\]
Therefore, there are two components $\tilde\calh_1$ and $\tilde\calh_2$ in the moduli space $\cals_{W_{\emptyset}^-}$, and their dimensions are $32$ and $23$ respectively.

\subsubsection{Finding all the walls}

Again, a wall is defined by a short exact sequence of pairs in the following form:
\[
0\lra (\mathcal{I}_A(k), 1)\lra (E(1), s)\lra (\mathcal{I}_B(2-k), 0) \lra 0
\]
for some one-dimensional schemes $A, B\subset \mathbb{P}^3$ (possibly empty) and $k\geq 0$. 
Indeed, $k$ must be equal to $1$ because for $k=0$, the wall has to be the collapsing wall with $A=\emptyset$; and for $k\geq 2$ we have $\delta=P_{E(1)}(t)-2P_{\mathcal{I}_A(k)}(t)<0$.
Thus, a wall between $0$ and $W_{\emptyset}$ is defined by the sequence
$$ 0\lra (\mathcal{I}_A(1), 1)\lra (E(1), s)\lra (\mathcal{I}_B(1), 0) \lra 0.$$ 
Moreover, one computes that $P_{\mathcal{O}_A}(t)+P_{\mathcal{O}_B}(t)=-P_E(t)+2P_{\mathcal{O}_{\mathbb{P}^3}}(t)=3t+2$. We also require that $P_{\mathcal{O}_A(1)}(t)<P_{\mathcal{O}_B(1)}(t)$ since $P_{\mathcal{O}_A(1)}(t)+\delta=P_{\mathcal{O}_B(1)}(t)$ ($\delta>0$).
Therefore, there are four possibilities for $A$ and $B$ as follows

\medskip

\begin{equation}{\label{wallsfor038}}
\begin{array}{|c|c|c|c|c|}
\hline
&W_3&W_2&W_1&W_0\\
\hline
P_{\mathcal{O}_A}(t)=& 2t+1&3t&3t+1&3t+2\\
\hline
P_{\mathcal{O}_B}(t)=& t+1 &2&1&0\\
\hline
\end{array}
\end{equation}

\medskip

In summary, there are five walls for $v(1)$, including the collapsing wall $W_{\emptyset}$ and $W_i$ ($i=0,1,2,3$) as given in Chart (\ref{wallsfor038}) above.

\subsubsection{Wall-crossings}

We studied the wall-crossings at $W_0$, $W_1$, and $W_2$ in Section \ref{sec: birat'l trans}, where (an open subscheme of) the component $\tilde\calh_{pl}$ undergoes a flip, divisorial contraction, and removal of an open subscheme. 
Here, we will compute the extension groups at those walls to complete the wall-crossing picture. 
We will also show that when crossing $W_3$, the corresponding component undergoes a flip. 

For simplicity, define the chambers $C_i$ to be the region between $W_i$ and $W_{i-1}$ for $i=0,1,2,3,4$ ($W_4:=0$ and $W_{-1}:=W_{\emptyset}$) as shown in the picture below.

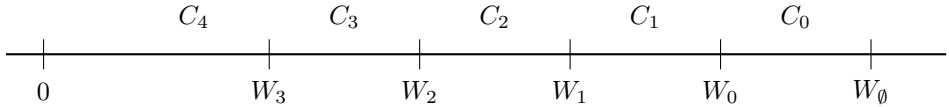
\begin{figure}[ht]
\begin{tikzpicture}
\draw[thick] (-0.5,0)--(12,0);
\draw (11,-0.2) -- (11, 0.2);
\node at (11, -0.5){$W_{\emptyset}$};
\draw (9,-0.2) -- (9, 0.2);
\node at (9, -0.5){$W_{0}$};
\node at (10, 0.5){$C_0$};
\draw (7,-0.2) -- (7, 0.2);
\node at (7, -0.5){$W_{1}$};
\node at (8, 0.5){$C_1$};
\draw (5,-0.2) -- (5, 0.2);
\node at (5, -0.5){$W_{2}$};
\node at (6, 0.5){$C_2$};
\draw (3,-0.2) -- (3, 0.2);
\node at (3, -0.5){$W_{3}$};
\node at (4, 0.5){$C_3$};
\node at (0, -0.5){$0$};
\draw (0,-0.2) -- (0, 0.2);
\node at (2, 0.5){$C_4$};
\end{tikzpicture}
\end{figure}

Let $A_i$ be a scheme whose Hilbert polynomial is 
\[
P_{\mathcal{O}_{A_i}}(t)=
\begin{cases}
3t+2 & \text{for} \ i=0\\
3t+1 & \text{for} \ i=1\\
3t & \text{for} \ i=2\\
2t+1 & \text{for} \ i=3.\\
\end{cases}
\]
Similarly, let $B_i$ be a scheme whose Hilbert polynomial is 
\[
P_{\mathcal{O}_{B_i}}(t)=
\begin{cases}
0 & \text{for} \ i=0\\
1 & \text{for} \ i=1\\
2 & \text{for} \ i=2\\
t+1& \text{for} \ i=3.\\
\end{cases}
\]

$\bullet$ \textbf{Wall-crossing at $W_0$.} 
We start with describing the strictly semistable locus at $W_0$, $\cals^{ss}_{{W_0}}\subset \cals_{{W_0}}$.
Recall that $W_0$ is defined as 
\begin{equation}{\label{Wall0}}
W_0: \ 0\lra (\mathcal{I}_{A_0}(1), 1)\lra (E(1), s)\lra (\mathcal{O}_{\mathbb{P}^3}(1), 0) \lra 0.
\end{equation}
According to \cite[Proposition 3.1]{nollet1997hilbert}, such a scheme $A_0$ is the (scheme theoretically) union of a plane cubic curve and two points in the same plane $H$. 
Since the strictly semistable pairs at ${W_0}$ are $(\mathcal{I}_{A_0}(1), 1)\oplus(\mathcal{O}_{\mathbb{P}^3}, 0)$, which are parameterized by $\calh^{3t+2}(\mathcal{U}_{\mathbb{P}^3}/\mathbb{P}^{3\vee})$. 
Here, $\mathcal{U}_{\mathbb{P}^3}$ is the universal plane in $\mathbb{P}^3\times \mathbb{P}^{3\vee}$, and $\calh^{3t+2}(\mathcal{U}_{\mathbb{P}^3}/\mathbb{P}^{3\vee})$ denotes the relative Hilbert scheme. 
We have
\[
\cals^{ss}_{W_0}\simeq \calh^{3t+2}(\mathcal{U}_{\mathbb{P}^3}/\mathbb{P}^{3\vee}),
\]
and its dimension is 16.

The next lemma computes the dimensions of the fibers for $\phi_{W_0}^{\pm}$.

\begin{lemma}{\label{ExtatW0}}
We have the following computational results.
\[
\begin{cases}
\Ext^1((\mathcal{I}_{A_0}(1), 1), (\mathcal{O}_{\mathbb{P}^3}(1), 0))\simeq\mathbb{C}^{15}\\
\Ext^1((\mathcal{O}_{\mathbb{P}^3}(1), 0), (\mathcal{I}_{A_0}(1), 1))\simeq\mathbb{C}^2.
\end{cases}
\]
\end{lemma}

\begin{proof}
Let $\Lambda:=(\mathcal{I}_{A_0}(1), 1)$ and $\Lambda':=(\mathcal{O}_{\mathbb{P}^3}(1), 0)$. 
Apply the long exact sequence in Lemma \ref{LES}, and we get 
$$ 0 \lra \Hom(\Lambda, \Lambda') \lra \Hom(\mathcal{I}_{A_0}(1), \mathcal{O}_{\mathbb{P}^3}(1)) \lra \Hom(\mathbb{C}, H^0(\mathbb{P}^3, \mathcal{O}_{\mathbb{P}^3}(1))) \lra $$
$$ \lra \Ext^1(\Lambda, \Lambda') \lra \Ext^1(\mathcal{I}_{A_0}(1), \mathcal{O}_{\mathbb{P}^3}(1)) \lra 0 = \Hom(\mathbb{C}, H^1(\mathbb{P}^3, \mathcal{O}_{\mathbb{P}^3}(1))) $$
In the sequence, $\Hom(\Lambda, \Lambda')=0$ since otherwise, the non-zero section for $\mathcal{I}_{A_0}(1)$ and the inclusion $\mathcal{I}_{A_0}(1)\hookrightarrow \mathcal{O}_{\mathbb{P}^3}(1)$ will induce a non-zero section for $\mathcal{O}_{\mathbb{P}^3}(1)$. It is evident that $\Hom(\mathcal{I}_{A_0}(1), \mathcal{O}_{\mathbb{P}^3}(1))\simeq\mathbb{C}$ and $\Hom(\mathbb{C}, H^0(\mathbb{P}^3, \mathcal{O}_{\mathbb{P}^3}(1)))\simeq\mathbb{C}^4$. $\Ext^1(\mathcal{I}_{A_0}(1), \mathcal{O}_{\mathbb{P}^3}(1))\simeq H^2(\mathbb{P}^3, \mathcal{I}_{A_0}(-4))^*\simeq H^1(\mathbb{P}^3, \mathcal{O}_A(-4))^*$. Let $A'_0$ be the pure one-dimensional subscheme of $A_0$, then  $H^1(\mathbb{P}^3, \mathcal{O}_{A_0}(-4))^*\simeq H^1(\mathbb{P}^3, \mathcal{O}_{A'_0}(-4))^*$. Since $A'_0$ is a plane cubic curve, we have that
$$ H^1(\mathbb{P}^3, \mathcal{O}_{A'_0}(-4))^*\simeq H^1(A'_0, \mathcal{O}_{A'_0}(-12))^*\simeq\mathbb{C}^{12}. $$ Therefore, $\Ext^1(\Lambda, \Lambda')\simeq\mathbb{C}^{15}$.

For the extension group $\Ext^1((\mathcal{O}_{\mathbb{P}^3}(1), 0), (\mathcal{I}_{A_0}(1), 1))$, apply the long exact sequence in Lemma \ref{LES} and one gets $$ \Ext^1((\mathcal{O}_{\mathbb{P}^3}(1), 0), (\mathcal{I}_{A_0}(1), 1))\simeq \Ext^1(\calo_{\p3}(1), \cali_{A_0}(1))=H^1(\p3, \cali_{A_0})\simeq\mathbb{C}^{2}. $$
\end{proof}

\begin{remark}
One can also compute the dimension of the fiber of $\phi^+_{W_0}$ based on the geometric interpretation for $\cals^+_{{W_0}}$ in Remark \ref{dim Zi tilde}.
$\cals^+_{{W_0}}\simeq \tilde{\mathcal{Z}}_0$ parameterizes pairs $(E(1), s)$ that correspond to one-dimensional schemes $Y$ in ${Z}_0\subset \calh^{4t}_{\mathbb{P}^3}$ (Definition \ref{locusZi}). 
Such a one-dimensional scheme $Y$ is the union of a planar quartic curve with two points in the same plane.
Thus, $\mathcal{Z}_0\simeq \calh^{4t}(\mathcal{U}_{\mathbb{P}^3}/\mathbb{P}^{3\vee})$ and $\dim(\mathcal{Z}_0$)$=21$. 

Then, the fiber for the morphism $\phi^+_{W_0}: \cals^+_{{W_0}}\to \mathcal{Z}_0\simeq \calh^{4t}(\mathcal{U}_{\mathbb{P}^3}/\mathbb{P}^{3\vee})$ is
$$ \mathbb{P}\big(\Ext^1(\mathcal{I}_Y(2), \mathcal{O}_{\mathbb{P}^3})\big) \simeq \mathbb{P}\big( H^2(\mathbb{P}^3, \mathcal{I}_Y(-2))^*\big)\simeq \mathbb{P}\big(H^1(\mathbb{P}^3, \mathcal{O}_Y(-2))^*\big)=\mathbb{P}^9, $$ 
which implies that the dimension of $\cals^+_{{W_0}}$ is $30$. 
Therefore, the dimension of the fiber for the morphism $\phi^+_{W_0}: \cals^+_{W_0}\to \cals^{ss}_{W_0}\simeq \calh^{3t+2}(\mathcal{U}_{\p3}/\mathbb{P}^{3\vee})$ is $30-16=14$; here, we used that $\dim\calh^{3t+2}(\mathcal{U}_{\p3}/\mathbb{P}^{3\vee})=16$. 
\end{remark}

When $\delta$ crosses $W_0$, the component $\tilde\calh_1$ becomes $\tilde\calh_1'=(\tilde\calh_1\backslash\tilde{\mathcal{Z}}_0)\cup \cals^-_{W_0}$. The following diagram illustrates the wall-crossing for $\tilde\calh_1$ at $W_0$, and the numbers in the parentheses denote the dimensions.
\[
\begin{tikzcd}
\mathbb{P}^1\arrow[dr,"fiber",swap]& \tilde\calh_1' (32)\arrow[dr]&&\tilde\calh_1 (32)\arrow[dl,swap]\arrow[ll, "flip",swap]&\mathbb{P}^{14}\arrow[dl,"fiber"]\\
&\cals^-_{W_0} (17)\arrow[u,hook]\arrow[dr, "\phi^-_{W_0}"]&\tilde\calh_{1,W_0}(32)&\cals^+_{W_0} (30)\simeq \tilde{\mathcal{Z}}_0\arrow[u,hook]\arrow[dl,"\phi^+_{W_0}"]&\\
&&\begin{matrix}\cals^{ss}_{W_0}\\ \simeq 
\calh^{3t+2}_{\mathcal{U}_{\mathbb{P}^3}/{\mathbb{P}^{3\vee}}} (16)\end{matrix} \arrow[u,hook]&&\\
\end{tikzcd}
\]

\begin{remark}
When crossing $W_0$, the other component $\tilde\calh_2$ undergoes a contraction. This is because $\tilde\calh_2\cap \cals^+_{W_0}\neq \emptyset$ and dim$(\tilde\calh_1'|_s)=32$ where $s$ is a general point in $\cals^-_{W_0}$. 
To compute dim$(\tilde\calh_1'|_s)=32$, Theorem \ref{tanobs1} implies that dim$(\tilde\calh_1'|_s)=\ext^1(Q, E)$, in which $Q$ and $E$ satisfy the sequences 
\[
(1) \quad 0\lra \cali_{A_0|H} \lra Q\lra \calo_{\p3}\lra 0 \quad\quad (2)\quad 0\lra \cali_{A_0} \lra E \lra \calo_{\p3} \lra 0. 
\] 
Apply $\Hom(-, E)$ to the first sequence, and we get 
\[
0\to \Ext^1(\calo_{\p3}, E)\simeq\mathbb{C}^1 \to \Ext^1(Q, E)\to \Ext^1(\cali_{A_0|H}, E) \to \Ext^2(\calo_{\p3}, E)\simeq\mathbb{C}^1.
\]
We compute $\Ext^1(\cali_{A_0|H}, E)$ by applying $\Hom(\cali_{A_0|H}, -)$ to the second sequence, and we have the following part of the long exact sequence
$$ 0 \lra \Ext^1(\cali_{A_0|H}, \cali_{A_0}) \simeq \mathbb{C}^{16} \lra \Ext^1(\cali_{A_0|H},E) \lra $$
$$ \lra \Ext^1(\cali_{A_0|H}, \calo_{\p3})\simeq\mathbb{C}^{15} \lra \Ext^2(\cali_{A_0|H}, \cali_{A_0})\simeq\mathbb{C}^6 \lra \cdots $$
This gives $\ext^1(\cali_{A_0|H}, E)\leq 31$, hence $\ext^1(Q, E)\leq 32$. On the other hand, $\ext^1(Q, E)\geq 32$ since $s\in \tilde\calh_1'$. So we must have $\dim(\tilde\calh_1'|_s)=\ext^1(Q, E)= 32$.
\end{remark}

In summary, the moduli space for $\delta\in(W_1, W_0)$ is 
$$\cals^{\delta}_{\p3}=((\tilde\calh_1\backslash \tilde{\mathcal{Z}}_0)\cup \cals^-_{W_0})\cup (\tilde\calh_2\backslash(\tilde\calh_2\cap \tilde{\mathcal{Z}}_0)).$$

$\bullet$ \textbf{Wall-crossing at $W_1$.}
The following sequence defines $W_1$.
\[
W_1: \ 0\lra (\mathcal{I}_{A_1}(1), 1)\lra (E(1), s)\lra (\mathcal{I}_P(1), 0)\lra 0,
\]
where $A_1\subset H$ is the union of a planar cubic curve, and a point $Q$ in the same plane, and $P\in \mathbb{P}^3$ is a point. $P$ is either outside $H$ or is an embedded point (in $A_1$ or set-theoretically coincide with $Q$), pointing out of the plane.

The strictly semistable locus $\cals^{ss}_{{W_1}}$ parameterizes pairs $(\mathcal{I}_{A_1}(1), 1)\oplus (\mathcal{I}_P(1), 0)$. So $\cals^{ss}_{{W_1}} \simeq \mathbb{P}^3\times \calh^{3t+1}_{\mathcal{U}_{\mathbb{P}^3}/\mathbb{P}^{3\vee}}$, which has dimension $17$.

We show some computational results in the next lemma to compute the fibers of the morphisms
$$\phi_{W_1}^{\pm}: \cals^{\pm}_{W_1}\to \cals^{ss}_{W_1}.$$
We skip the proof since it is similar to the one in Lemma \ref{ExtatW0}.
Let $A'_1\subset A_1$ be the pure one-dimensional subscheme. 

\begin{lemma}{\label{ExtatW1}}
$\Ext^1((\mathcal{I}_{A_1}(1), 1), (\mathcal{I}_{P}(1), 0))\simeq\mathbb{C}^{15}$.

$\Ext^1((\mathcal{I}_{P}(1), 0), (\mathcal{I}_{A_1}(1), 1))\simeq \Ext^1(\cali_P(1), \cali_{A_1})$
$=\begin{cases}
\mathbb{C} & \text{if}\ P\not\in A_1\\
\mathbb{C}^3 & \text{if}\ P=Q\ \text{and}\ P\not\in A'_1\\
\mathbb{C} & \text{if}\ P\neq Q\ \text{and}\ P\in A'_1\\
\mathbb{C}^6 & \text{if}\ P=Q \in A'_1.\\
\end{cases}$
\end{lemma}

Lemma \ref{ExtatW1} implies the following two diagrams
\begin{figure}[ht]
\begin{tikzcd}
\tilde\calh_1''(32)\arrow[r]&\tilde\calh_{1,W_1}(32)&\tilde\calh_1'(32)\arrow[l,swap,"\Phi^+_{W_1}"]&\\
\cals^-_{W_1}(17)\arrow[u,hook]\arrow[r]&\begin{matrix}\cals^{ss}_{W_1}(17)\\ \simeq 
\calh^{3t+1}_{\mathcal{U}_{\mathbb{P}^3}/{\mathbb{P}^{3\vee}}}\times \mathbb{P}^3\end{matrix}\arrow[u, hook]&\cals^+_{W_1}\simeq \tilde{\mathcal{Z}}_1(31)\arrow[u,hook]\arrow[l,"\phi^+_{W_1}",swap]&\\
\tilde{\mathcal{X}}^-_{W_1}(17)\arrow[u, hook]\arrow[r, "\simeq"]&\mathcal{R}_{W_1}(17)\arrow[u, hook]&\tilde{\mathcal{X}}_1\arrow[l, "\phi^+_{W_1}"](31)\arrow[u, hook]&\mathbb{P}^{14}\arrow[l, "fiber"]\\
\end{tikzcd}
\caption{Fibration over $\mathcal{R}_{W_1}$}
\label{Fibration RW1}
\end{figure}

\begin{figure}[ht]
\begin{tikzcd}
&\tilde\calh_1''\arrow[r]&\tilde\calh_{1,W_1}&\tilde\calh_1'\arrow[l,swap,"\Phi^+_{W_1}"]&\\
F\arrow[r, "fiber"]&\cals^-_{W_1}\backslash\tilde{\mathcal{X}}^-_{W_1}\arrow[u,hook]\arrow[r, "\phi^-_{W_1}"]&\cals^{ss}_{W_1}\backslash\mathcal{R}_{W_1} 
\arrow[u, hook]&\tilde{V}_1\arrow[u,hook]\arrow[l,"\phi^+_{W_1}",swap]&\mathbb{P}^{14}\arrow[l, "fiber"]\\
\end{tikzcd}
\caption{Fibration over $\cals^{ss}_{W_1}\backslash \mathcal{R}_{W_1}$}
\label{Fiber complement}
\end{figure}

Notations in the above two diagrams are from Definition \ref{locustilde} and Section \ref{sec: birat'l trans}. 
The fiber $F$ in Figure \ref{Fiber complement} could be $\mathbb{P}^{0}, \mathbb{P}^{2}$ or $\mathbb{P}^{5}$ depending on the position of $P$.
We see that the dimension of the fibers for $\phi^-_{W_1}$ jumps in the boundary of $\mathcal{S}^{ss}_{W_1}$ (see Remark \ref{fiber dim jump}).
Indeed, $\cals^+_{W_1}\simeq \tilde{\mathcal{Z}}_1\subset \tilde\calh_1'$ is a divisor following from the simple computation that $\dim(\tilde\calh_1')=32$ and $\dim(\cals^+_{W_1})=\dim(\cals^{ss}_{W_1})+\dim({\rm fiber~of~ }\phi^+_{W_1})=17+14=31$ (or one computes that $\dim(\tilde{\mathcal{Z}}_1)=31$). 
The first diagram (Figure \ref{Fibration RW1}) shows that the open subscheme $\tilde{\mathcal{X}}_1\simeq \tilde\calh_1'\backslash \tilde{V}_1\subset \tilde\calh_1'$ undergoes a divisorial contraction when crossing $W_1$ as stated in Theorem \ref{flipblowdownremoval}.

Moreover, the other component $\tilde\calh_2'$ is contracted along $\tilde\calh_2'\cap \tilde{\mathcal{Z}}_1$ when crossing $W_1$, and the moduli space for $\delta\in(W_2, W_1)$ is 
$$\cals_{W_1^-}\simeq (\tilde\calh_1'\backslash \tilde{\mathcal{Z}}_1\cup \cals^-_{W_1})\cup(\tilde\calh_2'\backslash(\tilde\calh_2'\cap \tilde{\mathcal{Z}}_1)).$$

$\bullet$ \textbf{Wall-crossing at $W_2$.} Recall that $W_2$ is defined by the sequence 
\[
0\lra (\mathcal{I}_{A_2}(1), 1)\lra (E(1), s)\lra (\mathcal{I}_{P_2}(1), 0) \lra 0,
\]
where $A_2$ is a planar one-dimensional scheme with $P_{\mathcal{O}_{A_2}}(t)=3t$, and $P_2\subset \mathbb{P}^3$ is a $0-$dimensional scheme of length $2$. 

$\cals^{ss}_{{W_2}}$ parameterizes pairs $(\mathcal{I}_{A_2}(1), 1)\oplus (\mathcal{I}_{P_2}(1), 0)$, so $\cals^{ss}_{{W_2}}\simeq \calh^{3t}_{\p3}\times \mathbb{P}^3 \times\mathbb{P}^3$, which is smooth of dimension $18$. 
Lemma \ref{Ext-} implies that wall-crossing at $W_2$ removes $ \tilde{\mathcal{X}}_2\subset \cals^+_{{W_2}}\simeq \tilde{\mathcal{Z}}_2$. 
We provide some computational results in the next lemma without proof to compute the fibers of $\phi^+_{W_2}$ over $\tilde{V}_2$. 
For the case $P=Q$, denote by $\bar{P}$ the corresponding fat point.

\begin{lemma}{\label{ExtatW2}}
$\Ext^1((\mathcal{I}_{A_2}(1), 1), (\mathcal{I}_{P_2}(1), 0)))\simeq\mathbb{C}^{15}$.
\[
\begin{array}{rl}
\Ext^1((\mathcal{I}_{P_2}(1), 0), (\mathcal{I}_{A_2}(1), 1))&\simeq \Ext^1(\cali_{P_2}(1), \cali_{A_2}(1))\\
&=\begin{cases}
0 & \text{if} \ P \ \text{or}\ Q\not\in A_2\\
\mathbb{C} &  \text{if}\  P,Q \in A_2,\ \text{and} \ P\neq Q\\
\mathbb{C}^2 &  \text{if}\  P=Q\in A_2, \ \text{and}\ \bar{P}\not\subset H.\\
\end{cases}\\
\end{array}
\]
\end{lemma}

Following the above computational results, we illustrate the wall-crossing for $\tilde\calh_1''$ at $W_2$ in the following two diagrams.

\[
\begin{tikzcd}
&\tilde\calh_1^{(3)}\arrow[r, hook]&\tilde\calh_{1,W_2}(18)&\tilde\calh_1''(32)\arrow[l,swap,"\Phi^+_{W_2}"]&\\
F\arrow[r,"fiber"]&\cals^-_{W_2}\arrow[u,hook, "\simeq"]\arrow[r]&\begin{matrix}\cals^{ss}_{W_2} (18)\\ \simeq 
\calh^{3t}_{\p3}\times \mathbb{P}^3\times \p3 \end{matrix}\arrow[u, hook]&\cals^+_{W_2}\simeq \tilde{\mathcal{Z}}_2 (32)\arrow[u,hook]\arrow[l,"\phi^+_{W_2}",swap]&\\
&\tilde{\mathcal{X}}^-_{W_2}=\emptyset\arrow[u, hook]\arrow[r]&\mathcal{R}_{W_2} (18)\arrow[u, hook]&\tilde{\mathcal{X}}_2(32)\arrow[l, "\phi^+_{W_2}", swap]\arrow[u, hook]&\mathbb{P}^{14}\arrow[l, "fiber", swap]\\
\end{tikzcd}
\]

\[
\begin{tikzcd}
&\tilde\calh_1^{(3)}\arrow[r]&\tilde\calh_{1,W_2}&\tilde\calh_1''\arrow[l,swap,"\Phi^+_{W_2}"]&\\
F\arrow[r, "fiber"]&\cals^-_{W_2}\arrow[u,hook]\arrow[r, "\phi^-_{W_2}"]&\cals^{ss}_{W_2}\backslash\mathcal{R}_{W_2} 
\arrow[u, hook]&\tilde{V}_2\arrow[u,hook]\arrow[l,"\phi^+_{W_2}",swap]&\mathbb{P}^{14}\arrow[l, "fiber", swap]\\
\end{tikzcd}
\]

One computes that dim$(\tilde{\mathcal{X}}_2)=32$, and $\tilde{\mathcal{X}}_2\subset \tilde\calh_1''$ is an open subscheme.
The first diagram shows that the open subscheme $\tilde{\mathcal{X}}_2\simeq\tilde\calh_1''\backslash \tilde{V}_2$ is contracted when hitting $W_2$, and then removed when $\delta$ crosses $W_2$. 
The second diagram shows that the subscheme $\tilde{V}_2$ is contracted when hitting 
$W_2$, and then replaced by $\cals^-_{W_2}$ lying over $\cals^{ss}_{W_2}\backslash \mathcal{R}_{W_2}$ when crossing $W_2$. 
We see again that the dimension of the fibers for $\phi^-_{W_2}$ jumps on the boundary of $\cals^{ss}_{W_2}$ (see Remark \ref{fiber dim jump}).

Indeed, $\tilde\calh_1^{(3)}$ is no longer a component in the moduli space, but a sub-variety of $\tilde\calh_2^{(3)}$. Up to this point, the moduli space for $\delta\in(W_3, W_2)$ is 
\[
\begin{array}{rl}
\cals^{\delta}_{\p3}\simeq& (\tilde\calh_2''\backslash \tilde{\mathcal{Z}}_2) \cup \cals^-_{W_2}\\
\simeq&(\tilde\calh_2'\backslash (\tilde{\mathcal{Z}}_1\cup\tilde{\mathcal{Z}}_2)) \cup \cals^-_{W_1}\cup \cals^-_{W_2}\\
\simeq&(\tilde\calh_2\backslash (\tilde{\mathcal{Z}}_0\cup\tilde{\mathcal{Z}}_1\cup\tilde{\mathcal{Z}}_2)) \cup \cals^-_{W_0}\cup \cals^-_{W_1}\cup \cals^-_{W_2}\\
\simeq& (\tilde\calh_2\backslash(\tilde\calh_1\cap \tilde\calh_2))\cup  \cals^-_{W_0}\cup \cals^-_{W_1}\cup \cals^-_{W_2}\\
\simeq&\tilde\calh_2^{(3)}.
\end{array}
\]
In the stratification of $\cals^{\delta}_{\p3}$, every stable pair in  $\tilde\calh_2\backslash(\tilde\calh_1\cap \tilde\calh_2)$ is saturated, while every pair in $\cals^-_{W_0}\cup \cals^-_{W_1}\cup \cals^-_{W_2}$ is unsaturated.

$\bullet$ \textbf{Wall-crossing at $W_3$.} $W_3$ is defined by the following sequence
\begin{equation}{\label{W3}}
W_3: \quad 0\lra (\mathcal{I}_C(1), 1)\lra (E(1), s)\lra (\mathcal{I}_{L}(1),0)\lra 0,
\end{equation}
where $L\subset \mathbb{P}^3$ is a line and $C\subset \mathbb{P}^3$ is a conic curve. 

Again, we start with describing the strictly semistable locus at the wall. 
$\cals^{ss}_{{W_3}}$ parameterizes pairs $(\mathcal{I}_C(1), 1)\oplus(\mathcal{I}_{L}(1),0)$, hence, $\cals^{ss}_{{W_3}}\simeq Gr(2,4)\times \calh^{2t+1}_{\mathbb{P}^3}$, which is smooth of dimension $12$. 
We compute the fibers of the morphisms $\phi^{\pm}_{W_3}: \cals^{\pm}_{{W_3}}\to \cals^{ss}_{{W_3}}$ in the following lemma.

\begin{lemma}{\label{ExtW3}}
We have the following results 
$$\begin{cases}
   \Ext^1((\mathcal{I}_C(1), 1), (\mathcal{I}_{L}(1),0))\simeq\mathbb{C}^9\\
   \Ext^1((\mathcal{I}_{L}(1),0), (\mathcal{I}_C(1), 1))\simeq\mathbb{C}^3.  
\end{cases}
$$
\end{lemma}

\begin{proof}
The long exact sequence in Lemma \ref{LES} implies that
$$ \Ext^1((\mathcal{I}_{L}(1),0), (\mathcal{I}_C(1), 1))=\Ext^1(\mathcal{I}_{L}(1), \mathcal{I}_C(1)), $$
and it is straightforward to compute that $\Ext^1(\mathcal{I}_{L}(1), \mathcal{I}_C(1))\simeq\mathbb{C}^{3}$. 

For $\Ext^1((\mathcal{I}_C(1), 1), (\mathcal{I}_{L}(1),0))$, apply the long exact sequence in Lemma \ref{LES}, and part of the long exact sequence is as follows
\[
0\to H^0(\mathbb{P}^3, \mathcal{I}_C(1)) \to \Ext^1((\mathcal{I}_C(1), 1), (\mathcal{I}_{L}(1),0))\to \Ext^1(\mathcal{I}_C(1), \mathcal{I}_{L}(1))\to 0.
\]
The claim follows from $h^0(\mathbb{P}^3, \mathcal{I}_L(1))=2$ and $\ext^1(\mathcal{I}_C(1), \mathcal{I}_{L}(1))=7$.
\end{proof}

Therefore, $\dim(\cals^+_{W_3})=12+8=20$ and $\dim(\cals^-_{W_3})=12+2=14$, and the component $\tilde\calh_2^{(3)}$ undergoes a flip when crossing $W_3$. We describe this wall-crossing in the following diagram. 
\[
\begin{tikzcd}
\mathbb{P}^2\arrow[dr,"fiber",swap]&\tilde\calh_2^{(4)}(23)\arrow[dr]&&\tilde\calh_2^{(3)}(23)\arrow[dl,swap]\arrow[ll, "flip",swap]&\mathbb{P}^{8}\arrow[dl,"fiber"]\\
&\cals^-_{W_3}(14)\arrow[u,hook]\arrow[dr, "\phi^-_{W_3}"]&\tilde\calh_{2,W_3}(23)&\cals^+_{W_3}(20)\arrow[u,hook]\arrow[dl,"\phi^+_{W_3}"]&\\
&&\begin{matrix}\cals^{ss}_{W_3}(12)\\ \simeq 
Gr(2,4)\times \calh^{2t+1}_{\p3}\end{matrix} \arrow[u,hook]&&\\
\end{tikzcd}
\]

Next, we reinterpret the morphism from $\cals^+_{{W_3}}$ to the Hilbert scheme and describe the corresponding one-dimensional schemes. 
A saturated pair $(F(1),s)$ in $\cals^+_{{W_3}}$ corresponds to a one-dimensional scheme $Y$ that satisfies the following sequence (see diagram \eqref{descriptionY})
\[
0\lra \mathcal{O}_L(-1)\lra \mathcal{O}_Y\lra \mathcal{O}_E \lra 0.
\]
In the sequence, $E\subset H\subset \mathbb{P}^3$ is a planar cubic curve, and $L\subset \p3$ is a line such that $L\not\subset H$. 
We show in the next two lemmas that $E$ is the union of a line and a conic in $H$. 

\begin{proposition}{\label{W3+curve}}
Every pair $(F(1),s)$ in $\cals^+_{W_3}$ is saturated, meaning that it corresponds to a one-dimensional scheme $Y\subset \p3$. 
\end{proposition}

\begin{proof}
Recall that 
$W_3$ is defined by the following short exact sequence of pairs (sequence  (\ref{W3}))
\[
W_3: \quad 0\lra (\cali_{L}(1), 0) \lra (F(1), 1) \lra (\cali_{C}(1), 1)\lra 0.
\] 
The cokernel $Q$ of the pair $(F(1), s)$ fits into the sequence 
$$0\lra \cali_{L}(1) \lra Q \lra \cali_{C|H}(1)\simeq \calo_H(-1) \lra 0.$$

On the one hand, let $Y$ be the union of $L$ and $E$, where $L\not\subset H$ and $E\subset H$ is a planar cubic curve, such that $L$ intersects $E$ at a point $P$. $Q\simeq \cali_Y(2)$ satisfies above sequence, i.e.
$$0\lra \cali_{L}(1) \lra \cali_Y(2) \lra \calo_H(-1)\simeq \cali_{E|H}(2)\lra 0.$$
So, $\ext^1(\calo_H(-1), \cali_L(1))\geq 8$, which is the dimension of the set 
$$ \{E\subset H ~|~ [E]\in |\calo_H(3)|, P\in E\}. $$
On the other hand, apply $\Hom(\calo_H(-1), -)$ to the sequence
$$ 0\lra \cali_L(1) \lra \calo_{\p3}(1) \lra \calo_L(1) \lra 0, $$ 
to get 
$$ 0 \to \Ext^1(\calo_H(-1), \cali_L(1)) \to \Ext^1(\calo_H(-1), \calo_{\p3}(1)) \to \Ext^1(\calo_H(-1), \calo_L(1)) \to 0$$
since $\Ext^2(\calo_H(-1), \cali_L(1))=0$.

In addition, one can check that $\ext^1(\calo_H(-1), \calo_{\p3}(1))= h^0(H, \calo_H(3))=9$ and $\ext^1(\calo_H(-1), \calo_L(1))=1$, so that $\ext^1(\calo_H(-1), \cali_L(1))=8$. Therefore, $Q\simeq \cali_Y(2)$ are all nontrivial extensions.
\end{proof}

Let $\mathcal{R}\subset \calh^{4t}_{\p3}$ be the locally closed subscheme parameterizing one-dimensional schemes of the form $Y=L\cup L'\cup C$, where $L$ and $L'\subset H$ are lines with $L\not\subset H$, and $C\subset H$ is a conic. 
Define $\tilde{\mathcal{R}}\subset \cals_{W_{\emptyset}^-}$ to be the locally closed subscheme lying over $\mathcal{R}$. 
We prove in the next lemma that $Y$ fits into sequence (\ref{W3}) if and only if $Y \in \mathcal{R}$.

\begin{lemma}{\label{singularE}}
The one-dimensional scheme $Y$, which corresponds to a pair $(F(1), s)\in \cals^+_{W_3}$, is the union of $L$ and $E$ at a point $P$, where $E$ is the union of a line $L'$ and a conic $C$ in $H$.
\end{lemma}

\begin{proof}
We have shown that $\dim(\cals^{+}_{W_3})=20$, and  
$\cals^{+}_{W_3}$ is lying over a subscheme $\calr'\subset \calh^{4t}_{\p3}$ with fibers being $\mathbb{P}\Big(\Ext^1(\cali_{Y}(2), \calo_{\p3})\Big) = \mathbb{P}^7$. Therefore, the expected dimension of the $\calr'$ is $13$.

It is straightforward to compute that 
$\dim\calr=\{[Y]\in \calh^{4t}_{\p3}|Y=L\cup L'\cup C\}=13$. 
We prove next that such a one-dimensional scheme $Y$ can recover the sequence $W_3$.
For $Y=L\cup E$, where $E=L'\cup C$, we have a morphism $\calo_Y(2)\to \calo_E(2)\to 0$, which factors through $\calo_H(2)$ since $E\subset H$. Hence, we have the following diagram 
\[
\begin{tikzcd}
\cali_L(1)\arrow{d}\arrow{r}&\cali_Y(2)\arrow{d}\arrow{r}&\cali_{E|H}(2)\arrow{d}\\
\calo_{\p3}(1)\arrow{d}\arrow{r}&\calo_{\p3}(2)\arrow{d}\arrow{r}&\calo_H(2)\arrow{d}\\
\calo_L(1)\arrow{r}&\calo_Y(2)\arrow{r}&\calo_E(2)\\
\end{tikzcd}
\]
with all rows and columns being exact. Note that $\cali_{E|H}(2)\simeq \cali_{C|H}(1)$ because $E=L'\cup C$. Consider the following diagram
\[
\begin{tikzcd}
&&\calo_{\p3}\arrow{d}\\
&&\cali_{C}(1)\arrow{d}\\
\cali_L(1)\arrow{r}&\cali_Y(2)\arrow{r}&\cali_{C|H}(1)\\
\end{tikzcd}
\]
and apply the functor $\Hom(-, \cali_L(1))$ to the column sequence in the diagram. We have
$$\Ext^1(\cali_{C|H}(1), \cali_L(1)) \lra \Ext^1(\cali_{C}(1), \cali_L(1)) \lra \Ext^1(\calo_{\p3}, \cali_L(1))=0.$$ 
So, $\cali_Y(2)$ induces an extension $E\in \Ext^1(\cali_{C}(1), \cali_L(1))$, 
which indicates that $Y$ recovers sequence (\ref{W3}).  
Therefore, $Y$ fits into sequence (\ref{W3}) if and only if \linebreak $Y=L\cup L'\cup C$. 
\end{proof}

\begin{remark}{\label{counterexample}}
This is an example explaining Remark \ref{curvetopair} that a one-dimensional scheme satisfying diagram \ref{descriptionY} may not recover the wall.
\end{remark}

To summarize, there is a morphism from the moduli space $\cals_{0^+}$ for $\delta\in(0, W_3)$ to $\calg_{\p3}$. It also has the stratification given by:
\[
\cals_{0^+}\simeq (\tilde\calh_2\backslash((\tilde\calh_1\cap \tilde\calh_2)\cup \tilde{\mathcal{R}}))\cup (\cals^-_{W_0}\cup \cals^-_{W_1}\cup \cals^-_{W_2}\cup \cals^-_{W_3}),
\]
where every pair from the first part, $\tilde\calh_2\backslash((\tilde\calh_1\cap \tilde\calh_2)\cup \tilde{\mathcal{R}})$, is very stable and saturated, while every pair from the second part, $\cals^-_{W_0}\cup \cals^-_{W_1}\cup \cals^-_{W_2}\cup \cals^-_{W_3}$, is unsaturated but very stable.


\subsection{The class \texorpdfstring{$v=(2,0,-2,0)$}{v=(2,0,-2,0)}} \label{subsec:4}

In \cite[Section 6]{jardim2017two}, it was shown that the moduli space of Gieseker semistable sheaves contains three components, labeled there as $\overline{\cali(2)}, \overline{\calt(2,1)}$, and $\overline{\calt(2,2)}$. 
In this subsection, instead of showing the complete wall-crossings, we will focus on proving that the three components in \cite[Section 6]{jardim2017two} are the only components using wall-crossings. Meanwhile, we will see that a component in the Hilbert scheme corresponding to non-reduced curves is removed when crossing a wall.
We start with proving the fact that $H^0(\p3, E(1))\neq 0$ for $E\in \calg^v_{\p3}$. This will indicate that we study wall-crossings for the class $v(1)$.

\begin{proposition}{\label{spectrum:020}}
For every $E\in \calg^v_{\p3}$, we have $h^0(\p3, E(1))> 0$.
\end{proposition}
\begin{proof}
We start with computing $\chi(E(1))$ and $\chi(E(-1))$ using Riemann-Roch theorem for later use. 
It is straightforward to compute that
    $$\chi(E(1))=2 \quad \text{and} \quad \chi(E(-1))=-2.$$

Next, we use the spectrum studied in \cite{almeida2019spectrum}. 
The splitting type of $E$ is $(a_1,a_2)=(0,0)$ (\cite[Proposition 6]{almeida2019spectrum}). This implies that $m=2$ for the spectrum $\{k_1, k_2, ..., k_m\}$ of $E$. So, the spectrum is $(k_1, k_2)\in \mathbb{Z}^2$ with $k_1\leq k_2$.
Let $s=h^0(\p3, \mathcal{E}xt^2(E, \calo_{\p3}))\geq 0$, and we have (\cite[Proposition 9]{almeida2019spectrum}) $$s\leq \frac{c_2^2+c_2}{2}=3.$$
By \cite[Proposition 7]{almeida2019spectrum},
$$0=c_3(E)=-2(k_1+k_2)-2s,$$
which implies that $-3\leq k_1+k_2=-s\leq 0$.

Then, \cite[Proposition 8]{almeida2019spectrum} implies that there are only the following two possibilities for the spectrum $\{k_1, k_2\}$,
\begin{enumerate}
    \item $k_1=-1$ $\Rightarrow$ $k_2=-1,0,1$

    \item $k_1=-2$ $\Rightarrow$ $k_2=-1$.
\end{enumerate}
This is because, if $k_1\leq -3$, then $\{k_1, k_1+1, ..., -1\}$ are all in the spectrum (\cite[Proposition 8, (b)]{almeida2019spectrum}). In particular, there are at least $3$ numbers in the sequence, which is a contradiction. Similarly, if $k_2\geq 2$, then all of $\{1,2,..., k_2\}$ must be in the spectrum. 
The only possibility will be $k_1=1, k_2=2$, which contradicts $k_1+k_2\leq 0$. Also, we must have $k_2\geq -1$, otherwise, $k_1+k_2\leq -5$ which contradicts to $k_1+k_2\geq -3$.

Therefore, the only possibilities are (1) $k_1=-2$, (2) $k_1=-1$, (3) $k_2=1$ (4) $k_2=0$ (5) $k_2=-1$. The corresponding spectra are as follows:
\begin{center}
\begin{tabular}{|l|l|}
\hline
   $k_1=-2$  & $k_2=-1$\\
   \hline
   $k_1=-1$  & $k_2=-1$\ \text{or}\ $0$\ \text{or}\ $1$\\
   \hline
   $k_2=1$  & $k_1= 0$\ \text{or}\ $-1$\\
   \hline
   $k_2=0$  & $k_1=-1$\\
   \hline
   $k_2=-1$ & $k_1=-2$\\
   \hline
\end{tabular}
\end{center}

Finally, using \cite[Theorem 1]{almeida2019spectrum}, we have 
$$h^2(E(1))=\sum_{i=1}^2h^1(\calo_{\p1}(k_i+2)),$$
which is equal to $0$ for all possible $(k_1, k_2)$ in the above chart. 
Lastly, we use the value $\chi(E(1))=2$ and get

\[
\begin{aligned}
    2&=\chi(E(1))=h^0(E(1))-h^1(E(1))+h^2(E(1))-h^3(E(1))\\
    &=h^0(E(1))-h^1(E(1))-h^3(E(1)).\\
\end{aligned}
\]
Therefore, $h^0(E(1))=2+h^1(E(1))+h^3(E(1))\geq 2$.
\end{proof}

Next, we consider the corresponding Hilbert scheme for $v(1)$, which is $\calh^{3t+3}_{\p3}$. 
According to \cite{chen2012detaching} and \cite[Proposition 3.1, 3.2]{nollet1997hilbert}, the Hilbert scheme has five components.
Denote by ${\calh}_1$, $\calh_2$, $\calh_3$, $\calh_4$, and $\calh_5$ the five components, and we describe the general one-dimensional schemes in those components in the following chart.

\medskip

\begin{center}
\begin{tabular}{|c|l|}
\hline
$\calh_1$& plane cubic with three points\\
\hline
$\calh_2$& twisted cubic (ACM) with two points\\
\hline
$\calh_3$& conic curve with a line and one point\\
\hline
$\calh_4$& three disjoint lines\\
\hline
$\calh_5$& double line $Z$ with $p_a(Z)=-3$ union a line\\
     & intersecting in a double point\\
\hline
\end{tabular}
\end{center}

\medskip

One-dimensional schemes in $\calh_4$, $\calh_3$ and $\calh_2$ correspond to sheaves in $\cali(2)$, $\calt(2,1)$ and $\calt(2,2)$ respectively (compare with \cite[Section 6]{jardim2017two}). 
Consider the wall-crossings for the class $v(1)$. 
One computes that there are six walls as follows:
\[
W_{\emptyset}: \quad 0\lra (\cali_{Y}(2), 0) \lra (E(1), 1) \lra (\calo_{\p3}, 1) \lra 0
\]
in which $Y\subset \p3$ is a one-dimensional scheme with $P_{\calo_Y}(t)=3t+3$.
\[
W_i: \quad 0\lra (\cali_{P_i}(1), 0) \lra (E(1), 1) \lra (\cali_{A_i}(1), 1) \lra 0
\]
for $i=0,1,2,3$. These walls are the ones defined in Section \ref{sec: birat'l trans}.
\[
W_4: \quad 0\lra (\cali_{L}(1), 0) \lra (E(1), 1) \lra (\cali_{L' \cup P_2}(1), 1) \lra 0
\]
where $L$ and $L'$ are lines, $P_2$ is a zero-dimensional scheme of length $2$. 

Denote by $\tilde\calh_i$ ($i=1,2,3,4,5$) the components in $\cals_{W^-_{\emptyset}}$ lying over $\calh_i$.
Before describing the wall-crossings, we prove a lemma concerning the non-reduced curves in $\calh_5$. Let $W\in \calh_5$ be a general one-dimensional scheme. According to \cite[Proposition 3.2]{nollet1997hilbert}, $W$ is the union of a double line $Z$ and a line $L$ intersecting at a double point. We have $p_a(Z)=-3$ and $P_{\calo_Z}(t)=2t+4$. Denote by $L'$ and $Q$ the reduction of the schemes $Z$ and $W$, respectively. So, $L'$ is a line and $Q$ is the union of $L$ and $L'$.

\begin{proposition}{\label{nonred}}
Let $E_W\in \Coh(\p3)$ be a rank $2$ torsion-free sheaf, satisfying the sequence
$$ 0\lra \calo_{\p3}(-1)\lra E_W \lra \cali_W(1) \lra 0.$$
Then there is an inclusion 
$$ \cali_{L'}\hookrightarrow E_W.$$
In particular, a pair $(E_W(1), s)$ satisfying the above sequence fits into the sequence $W_4$, hence $E_W$ is Gieseker unstable.
\end{proposition}
\begin{proof}
Dualize the sequence $0\to \calo_{\p3}(-1)\to E_W \to \cali_W(1) \to 0$, and we get 
\begin{equation}{\label{LES1}}
0\to \op3(-1)\to E_W^{\vee}\to \op3(1)\xrightarrow{\eta}\begin{matrix}\cale xt^1(\cali_W(1), \op3)\\=\omega_W(3)\end{matrix}\to \cale xt^1(E_W, \op3)\to 0.
\end{equation}

Consider the inclusion $\cali_W\hookrightarrow \cali_Q$. The cokernel is torsion-free supported on $L'$ with Hilbert polynomial $t+2$. Therefore, the completeness of the inclusion is 
    $$0\lra \cali_W\lra \cali_Q\lra \calo_{L'}(1) \lra 0,$$
    which is equivalent to 
    $$0\lra \calo_{L'}(1) \lra \calo_W\lra \calo_Q\lra 0.$$
Dualizing the above sequence by applying the functor $\calh om(-, \omega_{\p3})$, and we get
    $$0\lra \omega_Q\lra \omega_W\lra \omega_{L'}(-1)\lra 0.$$
Tensoring the sequence by $\op3(2)$, and we have
$$0\lra \omega_Q(2)\lra \omega_W(2)\lra \omega_{L'}(1)\lra 0.$$

Note that the morphism $\eta$ in sequence (\ref{LES1}) can be regarded as a nontrivial section in
$$ H^0(\p3,\omega_W(3)) \simeq H^0(\p3, \omega_Q(2))\simeq H^0(\p3, \calo_Q(1))\simeq H^0(\p3, \calo_{L'}(1))\oplus H^0(\p3,\calo_L). $$
Therefore, there are the following two possibilities for the image of $\eta$ 
$$(1)\ \im(\eta)=\calo_Q(1)\quad (2)\ \im(\eta)=\calo_{L'}(1).$$

\begin{enumerate}
\item
If $\im(\eta)=\calo_Q(1)$, then the exact sequence in display (\ref{LES1}) gives 
$$0\lra \op3(-1) \lra E^{\vee}_W\lra \cali_Q(1) \lra 0.$$
In addition, since $E_W^\vee$ is a reflexive rank 2 sheaf with trivial determinant, we have that $E_W^{\vee\vee}\simeq E_W^\vee$.

Since $Q$ is planar, we have $h^0(\p3, E_W^{\vee\vee})= h^0(\p3, \cali_Q(1))=1$, and a nonzero morphism $\op3\to E^{\vee\vee}_W$. The cokernel of the morphism must be torsion-free; otherwise, it factors through $\op3(-1)$ because $E_W^{\vee\vee}$is reflexive. This contradicts the fact that $c_1(E)=0$. Hence, we have one of the following two sequences
$$0\lra \op3 \lra E_W^{\vee\vee}\lra \cali_L \lra 0,  \ \text{or} \ 0\lra \op3 \lra E_W^{\vee\vee}\lra \cali_{L'} \lra 0,$$
and 
$$0\lra E\lra E^{\vee\vee}\lra T\lra 0,$$
in which $P_T(t)=t+3$. We must have $T\simeq \calo_{L'}(2)$ due to the bottom row of the following commutative diagram
\begin{center}
\begin{tikzcd}
\op3(-1)\arrow[d]\arrow[r,"="]&\op3(-1)\arrow[d]&&\\
E_W\arrow[r, hook]\arrow[d]&E_W^{\vee\vee}\arrow[d]\arrow[r, two heads]&T\arrow[d, "="]\\
\cali_W(1)\arrow[r, hook]&\cali_Q(1)\arrow[r, two heads]&T.\\
\end{tikzcd}
\end{center}

Therefore, consider the following diagram
\begin{center}
    \begin{tikzcd}
        &\op3\arrow[d]\arrow[dr, dotted, "\phi"]&\\
        E_W\arrow[r, hook]&E_W^{\vee\vee}\arrow[r,two heads]\arrow[d, two heads]&\calo_{L'}\\
        &\cali_L&\\
    \end{tikzcd}
\end{center}
we have that $\ker(\phi)=\cali_{L'}$ and $\coker(\phi)=\calo_{2P}$ ($2P$ is a 0-dimensional scheme of length 2). It induces the following sequence 
$$0\lra \cali_{L'}\lra E_W\lra \cali_{L\cup 2P}\lra 0,$$
in which $2P\subset L'$. 

\item 
If $\im(\eta)=\calo_{L'}(1)$, then sequence (\ref{LES1}) induces
$$0\lra \op3(-1)\lra E_W^{\vee}\lra \cali_L(1)\lra 0. $$
This implies that $E^{\vee}\simeq \op3\oplus\op3$. Consider the diagram
\begin{center}
\begin{tikzcd}
\op3(-1)\arrow[d]\arrow[r,"="]&\op3(-1)\arrow[d]&&\\
E_W\arrow[r, hook]\arrow[d]&E_W^{\vee\vee}\arrow[d]\arrow[r, two heads]&T\arrow[d, "="]\\
\cali_W(1)\arrow[r, hook]&\cali_L(1)\arrow[r, two heads]&T,\\
\end{tikzcd}
\end{center}
and we must have $T\simeq \calo_Z$. This leads to the diagram 
\begin{center}
    \begin{tikzcd}
        &\op3\arrow[d]\arrow[dr, dotted, "\phi"]&\\
        E_W\arrow[r, hook]&E_W^{\vee\vee}=\op3^{\oplus 2}\arrow[r,two heads]\arrow[d, two heads]&\calo_{Z}\\
        &\op3.&\\
    \end{tikzcd}
\end{center}
The image of $\phi$ is either $\calo_Z$ or $\calo_{L'}$.
\begin{enumerate}
    \item[(a)] If $\im(\phi)=\calo_{L'}$, then the above diagram leads to the sequence 
    $$0\lra \cali_{L'}\lra E_W\lra \cali_{L'\cup 2P} \lra 0.$$
    
    \item[(b)] If $\im(\phi)=\calo_{Z}$, then the above sequence leads to 
    $$0\lra \cali_{Z}\lra E_W\lra \op3\lra 0.$$
    Apply the functor $\Hom(-, \cali_{L'})$, and we have $$ \Hom(E_W, \cali_{L'})=\Hom(\cali_Z, \cali_{L'})\simeq\mathbb{C}.$$ This induces the following diagram
    \begin{center}
        \begin{tikzcd}
            &\ker(\psi)\arrow[d]&\\
            \cali_Z\arrow[r, hook]\arrow[d, "="]&E_W\arrow[r, two heads]\arrow[d, "\psi"]&\op3\arrow[d, "\alpha"]\\
            \cali_Z\arrow[r, hook]&\cali_{L'}\arrow[r, two heads]&\calo_{L'}(2).\\
        \end{tikzcd}
    \end{center}
It is evident that $\ker(\alpha)=\cali_{L'}$ and $\coker(\alpha)=\coker(\psi)=\calo_{2P}$.
Therefore, we have 
$$0\lra \cali_{L'}\lra E_W\lra \cali_{L'\cup 2P} \lra 0.$$
\end{enumerate}

To summarize, in all the above situations, we always have $\cali_{L'}\hookrightarrow E_W$.
\end{enumerate}
\end{proof}

Lastly, we briefly sketch the wall-crossings for $v(1)$ as $\delta$ gets smaller.

$\bullet$ When crossing $W_0$, $W_1$, $W_2$, and $W_3$, $\tilde\calh_1$ undergoes a flip, another flip, a blow down, and lastly, a removal of an open subscheme. The other components undergo some birational transformations or stay the same.
Besides, crossing $W_3$ will not create a new component from $\tilde{\calh}^{(3)}_1$. This is because $\tilde{\mathcal{X}}_3\subset \tilde{\calh}^{(3)}_1$ is contracted when hitting $W_3$, and then removed when crossing $W_3$. The subscheme $\tilde{V}_3=\tilde{\mathcal{Z}}_3\backslash \tilde{\mathcal{X}}_3$ may become components when crossing $W_3$. However, this is impossible because each point in $\tilde{V}_3$ is contained in $\calh_i$ for some $i=2,3,4,5$, and $\tilde{\calh}_i$ undergoes a birational transformation when crossing $W_3$. Therefore, there will be no components created. 

$\bullet$ When crossing $W_4$, $\tilde\calh_2^{(4)}$ undergoes a flip. $\tilde\calh_5^{(4)}$ is contracted when hitting 
$W_4$ due to Proposition \ref{nonred}, and then replaced by a subscheme $\tilde{\calg}_W\subset \cals_{0^+}$. There is a morphism $\tilde{\calg}_W\to \calg_W\subset \calg^v_{\p3}$ to a subscheme $\calg_W$ of the Gieseker moduli space. 
Indeed, $\calg_W$ is not a component itself, because otherwise, we can show that $\dim(\calg_W)\leq 13$.
Consider the following two fibrations for $\tilde{\calg}_W$, 
$$(a)\ \begin{matrix}\text{Fiber}\\ \mathbb{P}\Ext^1((\cali_{L'\cup 2P}(1), 1),(\cali_L(1), 0))\\=\mathbb{P}^4\end{matrix} \to\begin{matrix} \text{Base}\\ \{(\cali_{L'\cup 2P}(1), 1)\oplus(\cali_L(1), 0)\}\\ \dim=9\end{matrix}$$

$$(b)\ \begin{matrix}\text{Fiber}\\ \mathbb{P}H^0(\p3, E(1))\\\dim\geq 0\end{matrix} \to\begin{matrix} \text{Base}\\ \calg_W\\{}\end{matrix}.$$
The fibration in item (a) implies that $\dim(\tilde{\calg}_W)=13$, while the fibration in item (b) shows that $\dim(\calg_W)\leq 13$.
On the other hand, one computes that $\ext^1(E,E)\geq 17$ for $E$ satisfying the sequence $0\to \cali_{L'\cup 2P}\to E\to \cali_L\to 0$. This contradicts to $\dim(\calg_W)\leq 13$.

$\bullet$ For $\delta\in (0, W_4)$, there are three components remaining, namely $\tilde\calh_2^{(5)}$, $\tilde\calh_3^{(5)}$, and $\tilde\calh_4^{(5)}$. 
Therefore, we conclude that the components of $\calg^v_{\p3}$ found in \cite{jardim2017two} are its entirety.


\subsection{The class \texorpdfstring{$v=(2,0,-2,1)$}{v=(2,0,-2,1)}} \label{subsec:5}

We show in this last subsection that the Gieseker moduli space for class $v$ has two generically saturated components.

We show in the next Lemma a non-vanishing result on the global sections.
\begin{proposition}{\label{spectrum:022}}
For every $E\in \calg^v_{\p3}$, we have $h^0(\p3, E(1))> 0$.
\end{proposition}

\begin{proof}
    The proof is similar to Proposition \ref{spectrum:020}.
    A direct computation shows that $\chi(E(-1))=-1$, and $\chi(E(1))=3$. We will show next that $h^2(E(1))=0$, which implies  $h^0(E(1))\geq 3>0$.

    Again, we will prove the vanishing of $h^2(E(1))$ by finding all possible spectra for $E$.
    The splitting type of $E$ is $a_1=0, a_2=0$ with $m=2$. The result in \cite[Proposition 7]{almeida2019spectrum} implies that
    $$2=c_3=-2(k_1+k_2)-2s,$$
    which simplifies to $k_1+k_2=-s-1\leq -1$ ($k_1\leq k_2$).

    By \cite[Proposition 8]{almeida2019spectrum}, we must have $k_1\geq -2$, otherwise, $\{-3,-2,-1\}$ are in the spectrum, which is a contradiction. Also, we cannot have $k_2\leq 2$, otherwise, the only possibility is $k_1=1, k_2=2$, which contradict to $k_1+k_2\leq -1.$
    We list all possibilities in the chart below
    \begin{center}
        \begin{tabular}{|c|c|}
        \hline
         $k_1=-1$    & $k_2=-1$ or $0$ or $1$ \\
         \hline
         $k_1=-2$    & $k_2=-1$\\
        \hline
        \end{tabular}
    \end{center}

   Lastly, \cite[Theorem 1]{almeida2019spectrum} implies that 
$$h^2(E(1))=\sum_{i=1}^2h^1(\calo_{\p1}(k_i+2)),$$
    which is $0$ for all possible $(k_1,k_2)$ in the chart, hence the claim follows.
\end{proof}

\subsubsection{Wall-crossings}
According to Proposition \ref{spectrum:022}, we consider the wall-crossings for the twisted class $v(1)$, where $v=[2, 0, -2, 1]$. 
Let $Y$ be the one-dimensional scheme corresponding to a stable pair $(E(1),s)$, we have

$$0\lra \calo_{\p3} \lra E(1) \lra \cali_Y(2) \lra 0.$$

It is straightforward to check that $P_{\calo_Y}(t)=3t+2$. There are three components in $\calh^{3t+2}_{\p3}$ (\cite{nollet1997hilbert}) shown below 

\medskip

\begin{center}
\begin{tabular}{|c|l|}
\hline
Components& Description of a general one-dimensional scheme\\
\hline
$\calh_1$  & planar cubic curve with points \\
\hline
$\calh_2$  & twisted cubic (rational curve) with a point \\
\hline
$\calh_3$  & conic curve with a line\\
\hline
\end{tabular}
\end{center}

\medskip

Moreover, there are five walls for $v(1)$ shown as follows:

\medskip

\begin{center}
\begin{tabular}{|c|l|}
\hline
$W_{\emptyset}$     & The collapsing wall \\
\hline
$W_0$     & $0\lra (\op3(1),0)\lra (E(1),1)\lra (\cali_{C \cup 2P}(1), 1)\lra 0$\\
\hline
$W_1$     & $0\lra (\cali_Q(1),0)\lra (E(1),1)\lra (\cali_{C \cup P}(1), 1)\lra 0$\\
\hline
$W_2$     & $0\lra (\cali_{2Q}(1),0)\lra (E(1),1)\lra (\cali_{C}(1), 1)\lra 0$\\
\hline
$W_3$     & $0\lra (\cali_L(1),0)\lra (E(1),1)\lra (\cali_{L \cup P}(1), 1)\lra 0$\\
\hline
\end{tabular}
\end{center}

\medskip

In the above chart, $L$ is a line, $C$ is a conic, $P, Q, 2P, 2Q$ are 0-dimensional schemes in $\p3$ with the lengths of them are $1,1,2,2$ respectively.

Let $\tilde\calh_i\subset \cals_{\delta^-_{\emptyset}}$ be the components lying over $\calh_i$.
Crossing $W_0$, $W_1$, and $W_2$ will kill the component $\tilde\calh_1$. 
Then, crossing $W_3$ results in a flip for $\tilde\calh_2$. Finally, when $\delta\in (0, W_3)$, there are two components $\tilde\calh'_2$ and $\tilde\calh'_3$ (birational to $\tilde\calh_2$ and $\tilde\calh_3$) in the moduli space of pairs. Correspondingly, there are two generically saturated components in the Gieseker moduli space.


\bibliographystyle{alpha}
\bibliography{ref.bib}

\end{document}